\documentclass[10pt, a4paper, onecolumn]{IEEEtran}

\usepackage{graphicx}
\usepackage{amsmath}
\usepackage{amsfonts}
\usepackage{amssymb}
\usepackage{color}
\usepackage{rotating}

\newtheorem{lemma}{Lemma}
\newtheorem{theorem}{Theorem}
\newtheorem{corollary}{Corollary}
\newtheorem{definition}{Definition}

\newtheorem{proposition}{Proposition}
\newtheorem{claim}{Claim}

\DeclareMathOperator{\rank}{rank}
\title{Network Coding meets Decentralized Control: Network Linearization and Capacity-Stabilizablilty Equivalence} 
\author{Se Yong Park (separk@eecs.berkeley.edu), Anant Sahai (sahai@eecs.berkeley.edu)}

\begin{document}
\maketitle
\abstract
We take a unified view of network coding and decentralized control. Precisely speaking, we consider both as linear time-invariant systems by appropriately restricting channels and coding schemes of network coding to be linear time-invariant, and the plant and controllers of decentralized control to be linear time-invariant as well. First, we apply linear system theory to network coding. This gives a novel way of converting an arbitrary relay network to an equivalent acyclic single-hop relay network, which we call \textit{Network Linearization}. Based on network linearization, we prove that the fundamental design limit, mincut, is achievable by a linear time-invariant network-coding scheme regardless of the network topology.

Then, we use the network-coding to view decentralized linear systems. We argue that linear time-invariant controllers in a decentralized linear system ``communicate" via linear network coding to stabilize the plant. To justify this argument, we give an algorithm to ``externalize" the implicit communication between the controllers that we believe must be occurring to stabilize the plant. Based on this, we show that the stabilizability condition for decentralized linear systems comes from an underlying communication limit, which can be described by the algebraic mincut-maxflow theorem. With this re-interpretation in hand, we also consider \textit{stabilizability over LTI networks} to emphasize the connection with network coding. In particular, in broadcast and unicast problems, unintended messages at the receivers will be modeled as secrecy constraints.

\section{Introduction}
This paper is inspired by the similarity between the algebraic characterization of fixed modes~\cite{Anderson_Algebraic} in decentralized control problems and the min-cut bound in information theory~\cite{Cover}.

Consider a standard decentralized linear system
\begin{align}
&x[n+1]=Ax[n]+B_1 u_1[n] + \cdots + B_v u_v[n] \\
&y_1[n]=C_1 x[n]\\
&\vdots\\
&y_v[n]=C_v x[n].
\end{align}
Then, the algebraic condition for $\lambda$ to be a fixed mode~\cite[Theorem 4.1]{Anderson_Algebraic} is
\begin{align}
\min_{V \subseteq \{1,2,\cdots,v \}} \rank
\begin{bmatrix}
A-\lambda I & B_V \\
C_{V^c}& 0 \\
\end{bmatrix} \geq \dim(A). \label{eqn:intro:1}
\end{align}
Consider a communication relay network shown in \cite[Theorem 15.10.1]{Cover} where the input to the channel at the relay node $i$ is $X_i$ and the output from the channel at the relay node $i$ is $Y_i$. Then, the information-theoretic min-cut bound~\cite[Theorem 15.10.1]{Cover} is
\begin{align}
\min_{V \subseteq \{1,2,\cdots,v \}} I(X_V;Y_{V^c} | X_{V^c}) \geq  \sum_{i \in V ,j \in V^c} R_{ij}. \label{eqn:intro:2}
\end{align}
We can see that the left-hand sides of both \eqref{eqn:intro:1} and
\eqref{eqn:intro:2} have a minimization over all subsets $V$. Moreover, in
noiseless relay networks the mutual information is essentially equal
to the rank of an appropriate channel matrix\footnote{Information is
  traditionally measured in bits and the rate of bits that a channel
  can carry is computed by the mutual information $I(X;Y)$. However,
  in continuous-alphabet channels like the AWGN (additive white
  Gaussian noise) channel, the mutual information depends crucially on
  the signal-to-noise ratio and scales as $\log$ SNR. It was noticed
  that when the channel has multiple-inputs and multiple-outputs (MIMO) ---
  like when there are multiple antennas involved in wireless
  communication ---  the mutual information increases as the rank
  of the channel matrix {\bf times} $\log$ SNR. This fact inspired the
  creation of the finite-field noiseless MIMO channel model, within
  which the mutual information is equal to the rank of the channel
  matrix multiplied by the $\log$ of the field size. Therefore, the
  rank can be considered another measure for information, as measured
  in units of dimensions or degrees-of-freedom. We refer the reader to \cite{Tse} for further details.}~\cite{Tse}. Therefore, the left-hand sides of
\eqref{eqn:intro:1} and \eqref{eqn:intro:2} can be considered to be
exactly the same. Identifying the right hand sides of \eqref{eqn:intro:1} and
\eqref{eqn:intro:2} with each other, we can see that the dimension of
$A$ corresponds to a rate of total information flow. Moreover, fixed
modes are closely connected to stabilizability. Thus, we can
conjecture that a decentralized system is stabilizable if and only if
enough information flow can be supported to stabilize the plant, and vice versa. In this
paper, we make this conjecture rigorous.

First, let's review perspectives on information flow in communication
networks. Historically, information in a network was believed to
behave like a physical commodity. The network was modeled using a graph,
and the information was thought of as commodities to be transported from
the source to the destination by routing them through the nodes. The
most important result is the celebrated mincut-maxflow
theorem~\cite{Shannon_Maxflow,Ford_Maxflow}, which reveals that the
maximum amount of commodity flow through a graph is equal to the
minimum cut of the graph. Moreover, this maximum flow is achievable by
a routing scheme. For decades, this optimality result made researchers stick to
routing solutions even for information.

However, in \cite{Ahlswede_Network} it was found that information flow
in networks does not really behave like physical commodities do.
Obviously, we can copy information. But going further, we can also process and mix information. The famous \textit{butterfly example} shows that for multiple-source multiple-destination cases,
there is a gain by allowing relays to mix their incoming signals instead of just routing them.

Even if physical commodity flows (which we can only route) and
information flows (which we can copy, process and mix) are different, the
graph-theoretic concepts and insights originally developed for
commodity flows continue to be helpful. The main difference is that
the amount of flow, which is naturally measured by the number (or
weight or volume) of commodities in physical commodity flows, must
instead be measured in ``dimensions'' of the signal for information
flows. However, the mincut-maxflow theorem remains the main tool to
understand network information flows. For example, in the multicast
problem the relevant mincut is the minimum of the
mincut to each destination, and the mincut-maxflow theorem still
holds~\cite{Ahlswede_Network}. Moreover, this maximum flow is
achievable by linear time-invariant network
coding~\cite{Koetter_Algebraic}.

Once information-theorists had the freedom to mix and process signals
inside the nodes that they could design, they also started to consider
such operations as potentially existing outside these
nodes~\cite{Bobak_Computation}. The signals from the relay nodes could
be broadcast to multiple receiving nodes or superposed with other
signals at a receiving node. In fact, such extensions were a natural
fit to wireless communication~\cite{Salman_Wireless}. The operations
outside the nodes modeled communication channels and such wireless
channel models had long been valuable even when restricted to be
linear time-invariant.

At this point, we can see the similarities between network-coding
problems~\cite{Salman_Wireless} and decentralized-linear-control
problems~\cite{Wang_Stabilization}. The network channels (which
we cannot design) can be considered as the linear plant. The source,
relays and destination nodes (which we can design) can be considered as
decentralized controllers. Just as decentralized controllers process and
combine their observations to generate their control inputs, the relay
nodes process and combine their incoming signals from the channel to
generate their outgoing signals.

Despite these similarities, many differences between the communication
and control problems had been preventing a firm connection being made
between them. First of all, network-coding information-theorists work
in finite fields, whereas control-theorists default to infinite fields
like the reals or complex numbers. Moreover, information-theorists
tend not to have any explicit state in the system, preferring an
input-output perspective. Most importantly, the information-theorists
have a clearly specified source and destination, and their goal is
to push information from one to the other. The control-theorists tend
not to have explicit sources and destinations, and instead there is a
dynamic evolution that needs to be controlled or stabilized.

The main goal of this paper is to bridge these differences and make a
concrete connection between network coding and decentralized linear
control. We first apply linear-system-theoretic ideas to network coding to propose \textit{network linearization} as an algorithm to convert an arbitrary-topology network to an equivalent acyclic single-hop relay network. Based on this, we prove an algebraic mincut-maxflow theorem, Theorem~\ref{thm:mincut}.

Then, we apply network coding ideas to decentralized linear systems. As shown in Theorem~\ref{thm:equivalence1} and
\ref{thm:equivalence2}, we prove that if a decentralized linear
system is LTI\footnote{It is in our focus on stabilizability using
only linear {\em time-invariant} control laws that the results in this
paper differ from the results in \cite{Yuksel_Decentralized} where
{\em time-varying} control laws are permitted. The overall
perspectives however are compatible in that we are also
interested in cutsets and information flows.}-stabilizable, then there
must exist a corresponding implicit information flow sufficient to
stabilize the system.

The rest of the paper is organized as follows: In Section~\ref{sec:LTI_networks}, we introduce the definition of LTI networks and prove an algebraic mincut-maxflow theorem based on network linearization. We also compare network linearization with the known idea of network unfolding. In Section~\ref{sec:prelim2}, we introduce some preliminary facts about decentralized linear systems. Section~\ref{sec:example} shows a representative example that clearly illustrates the implicit information flows in
decentralized systems. Section~\ref{sec:external} gives the capacity-stabilizability equivalence
theorem. In Section~\ref{sec:stablilizationoverLTI}, we consider the stabilizability problem with an explicit communication network, and convert networking results to the equivalent stabilizability results.

\section{LTI Communication Networks}
\label{sec:LTI_networks}
\subsection{Definitions and Algebraic Mincut-Maxflow Theorem}
\label{sec:prelim}

An LTI communication network is a collection of transmitters, relays, and receivers --- which will be called nodes.\footnote{The LTI networks considered here are essentially the same as the
linear deterministic model studied in \cite{Salman_Wireless} except that
our LTI networks restrict the relay design to be linear time-invariant
and the underlying field can also be real or complex as well as a finite field.}

Each node has input and output ports. These connect to the channels. Each node generates a signal and sends it to the channels through its output ports, which are simultaneously the input to the channels. In this paper, we model signals elements from a field $\mathbb{F}$ and time is discrete. The transmitted signals go through the channels and arrive the channel outputs, which are simultaneously the input ports of the nodes. We take a channel-centric perspective in this paper's notation.

The relationship between the input and output signals of the channels is given by nature. In LTI communication networks, the input-output relationships of the channels are linear time-invariant. Thus, they can be described by transfer functions. Furthermore, since we will focus on discrete-time systems, by $z$-transform the transfer functions can be represented by rational functions in $z$. 

Even though the channels are given by nature, we still have design freedom for the nodes. Each node can choose the input signals to the channels as arbitrary causal functions on the output signals from the channels. In LTI networks, the node operation is restricted to be linear time-invariant. In other words, the nodes can be thought as causal linear time-invariant filters between the output signals from the channels into the input signals to the channels. To reflect this design freedom, we will assign different variables $k_i$ for the transfer functions inside the nodes.

We focus on LTI point-to-point communication networks with one transmitter and one receiver, and we denote the network as $\mathcal{N}(z)$. Let's formally define LTI point-to-point networks using graph notation. The input and output ports of the nodes can be modeled as vertices. The transfer functions connecting them can be thought as directed edges. Consider a digraph $(W,E)$ with a totally ordered set of vertices (ports) $W$ and a set of edges $E$. $W$ is partitioned according to which node that port belongs to. 

In other words, for LTI network with $v$ relays, $W$ can be partitioned into the sets $N_{tx}, N_1, \cdots, N_{v}, N_{rx}$, i.e. $N_i \subseteq W$, $N_i \cap N_j = \emptyset$ for $i \neq j$, and $\bigcup_{i  \in \{tx, 1, \cdots, v, rx \}} N_i = W$. Thus, a set of vertices $N_i$ corresponds to a node. 

To simplify the notations, we will use the subscript $``tx"$ and $-1$ interchangeably. 
Likewise, we will also use the subscript $v+1$ for the subscript $``rx"$, i.e. $N_{tx}=N_0$ and $N_{rx}=N_{v+1}$. 

For a given node $N_i$, the elements of $N_i$ are again partitioned into two subsets $N_{i,in}$ and $N_{i,out}$ which are called input and output vertices of the node $i$. The inputs and the outputs are defined in a channel-centric perspective. So input vertex is an output port of a node, and an output vertex is an input port of the node. $N_{i,in}$ represent the signals going out from the node $i$ \emph{into} the channels and $N_{i,out}$ represent the signals coming \emph{out} from the channels into the node $i$. 

The transmitter node does not receive signals and the receiver node does not transmit signals, so $N_{tx,out}=\emptyset$ and $N_{rx,in}=\emptyset$. We denote the number of the input and output vertices of the node $i$ as $d_{i,in}$ and $d_{i,out}$, i.e. $d_{i,in}:=|N_{i,in}|$, and $d_{i,out}:=|N_{i,out}|$.

Let the signals take values from a field $\mathbb{F}$, $z$ be the dummy variable for $z$-transforms, and $K=\{ k_1, k_2, k_3, \cdots \}$ be a set of variables to represent the gains inside the nodes. 
We also define $\mathbb{F}[z]$, $\mathbb{F}[K]$, $\mathbb{F}[z,K]$ as the field of all rational functions in variables $z$, $K$, $\{z\} \cup K$ with coefficients in $\mathbb{F}$ respectively.

Each edge which connects the ports of the nodes can be written as a triplet $(w',w'',h_{w',w''}(z,K)) \in E$ where $w',w'' \in W$ and $h_{w',w''}(z,K) \in \mathbb{F}[z] \cup K$. Here, $w'$ is the starting port of the edge, $w''$ is called the end of the edge, and $h_{w',w''}(z,K)$ is the gain of the connection.


Since a lack of physical connection between two vertices $w'$ and $w''$ can be represented as $h_{w',w''}(z,K)=0$, we assume that every input vertex is connected to every output vertex, including ``self-loops" connecting the input vertices to its own output vertices. There are two kinds of edges. One kind of edges is the transfer functions connecting the input vertices to the output vertices ---channel transfer functions. They are given by nature and described by z-transforms ---rational functions on $z$. Formally, for all $i,j\in \{ 0,\cdots, v+1\}$ and $w' \in N_{in,i} , w'' \in N_{out,j}$,
\begin{align}
(w',w'',h_{w',w''}(z,K)) \in E \mbox{ and } h_{w',w''}(z,K) \in \mathbb{F}[z]. \nonumber
\end{align}

The other kind of edge is inside each node. There we have design freedom. To reflect this, for each node let there exist edges fully connecting its output vertices to its input vertices. The transfer functions associated with these edges are in the form of $k_i \in K$ and distinct. Since the transmitter and receiver has only one kind of ports, $N_{tx}$ and $N_{rx}$ do not have internal edges. 

This distinct transfer function assumption guarantees enough design freedom at the relays since we can assign different transfer functions to different edges. Formally, for all $i \in \{ 1,\cdots, v \}$ and $w' \in N_{out,i},w'' \in N_{in,i}$,
$(w',w'',h_{w',w''}(z,K)) \in E \mbox{ and } h_{w',w''}(z,K)=k_{w',w''}$ where $k_{w',w''} \in K$. If $(w_1',w_1'')$ and $(w_2',w_2'')$ are distinct internal edges, $h_{w_1',w_1''} \neq h_{w_2',w_2''}$. These internal edges represent the potential LTI communication schemes. In a fully realized network with a specific communication scheme, each element of the $K_i$ will be replaced with a specific element in $\mathbb{F}[z]$.

At each vertex and edge, the signal is processed as follows: Each vertex $w \in W$ adds all the signals coming from the edges whose head is $w$ and transmits to the edges whose tail is $w$. Each edge $e \in E$ multiplies the signal coming from its tail with its transfer function and transmits to its head.

Denote a transfer function matrix from the input vertices of the node $N_i$ to the output vertices of the node $N_j$ as $H_{i,j}(z)$. In the same way, we denote a transfer function from a set (ordered set) of nodes $A$ to a set (ordered set) of nodes $B$ as $H_{A,B}(z)$. We also denote the transfer function matrix from the output vertices (input ports) of $N_i$ to the input vertices (output vertices) of $N_i$ as $K_i$. Then, $H_{i,j}(z) \in \mathbb{F}[z]^{d_{j,out} \times d_{i,in}}$ and $K_i \in \mathbb{F}[K]^{d_{i,in} \times d_{i,out}}$. For briefness, we write $H_{i,j}(z)$ as $H_{i,j}$ when it does not cause confusion. $K_i$ are given in forms of $\begin{bmatrix} k_{i_1} & k_{i_2} & \cdots \\ k_{i_1} & k_{i_2} & \cdots \\ \vdots & \vdots & \ddots \end{bmatrix}$.

As mentioned above, by considering the transfer functions of the internal edges as different bare dummy variables in $K$, we reflect the design freedom of the relay nodes. Moreover, the capacity of network ---the rank of the transfer function matrix--- will be maximized by considering the transfer functions of the internal edges as variables in $K$. Precisely, let $K_i(z) \in \mathbb{F}[z]^{d_{i,in} \times d_{i,out}}$ be a matrix whose size is the same as $K_i$ but the elements of the matrix belong to $\mathbb{F}[z]$.
Denote the transfer functions from the transmitter to the receiver of $\mathcal{N}(z)$ as $G(z, K)$ and $G(z, K(z))$ in each case. Then, we have the  following relationship:
\begin{lemma}
Let $G(z, K)$ be given as above. Then, we have the following relationship between the rank of $G(z, K)$ and $G(z, K(z))$.
\begin{align}
\rank G(z, K)= \underset{K_i(z) \in \mathbb{F}[z]^{d_{i,in} \times d_{i,out}}}{\max} \rank G(z, K(z)).
\end{align}
\label{lem:keylemma}
\end{lemma}
\begin{proof}
The proof is essentially the same as \cite[Lemma~1]{Koetter_Algebraic}. For all $K_i(z) \in \mathbb{F}[z]^{d_{i,in} \times d_{i,out}}$ the independent columns in $G(z,K(z))$ are still independent even if we consider the elements of $K_i$ as variables. Therefore, for all $K_i(z) \in \mathbb{F}[z]^{d_{i,in} \times d_{i,out}}$, $\rank G(z, K) \geq \rank G(z, K(z))$. 

Moreover, the rational function field $\mathbb{F}[z]$ has infinite number of elements and the dimension of the algebraic variety that makes $G(z,K)$ lose its rank is strictly smaller than the dimension of $K_i$'s. Therefore, there exists $K_i(z) \in \mathbb{F}[z]^{d_{i,in} \times d_{i,out}}$ such that $\rank G(z, K) = \rank G(z,K(z))$. Thus, the lemma is true.
\end{proof}

Figure~\ref{fig:LN_ptop} shows the graphical representation of LTI communication network. The squares represent the nodes of the LTI networks. The empty circles attached to the squares represents the input vertices (output ports) from the nodes to the channels. The circles with plus represents the output vertices (input ports) from the channels to the nodes. The arrows outside the nodes (connecting empty circles to plus circles) represent the communication channels, and the arrows inside the nodes (connecting plus circles to empty circles) represent the communication schemes. The scalars (or matrices) written on the arrows represent the transfer functions (or transfer function matrices). We also denote $I_m$ as a $m \times m$ identity matrix. 

Let $G(z,K)$ be the transfer function from the input vertices of the transmitter node to the output vertices of the receiver node. $G(z,K)$ can be written in terms of $H_{i,j}$ and $K_i$~\cite{Oppenheim}.
\begin{theorem}
With the above definitions, the transfer function matrix $G(z,K)$ is given as
\begin{align}
&G(z,K)=
\begin{bmatrix}
H_{1,rx}K_1 & \cdots & H_{v,rx}K_v
\end{bmatrix} \nonumber \\
&\left(
I -
\begin{bmatrix}
H_{1,1} K_1 & \cdots & H_{v,1} K_v  \\
\vdots & \ddots & \vdots \\
H_{1,v} K_1  & \cdots & H_{v,v} K_v
\end{bmatrix}
\right)^{-1}
\begin{bmatrix}
H_{tx,1}\\
\vdots \\
H_{tx,v}
\end{bmatrix}+H_{tx,rx}
. \nonumber
\end{align}
\label{thm:transfer}
\end{theorem}
\begin{proof}
As illustrated in Fig.~\ref{fig:LN_ptop}, let $U$, $X_i$ and $Y$ be vectors of signals at the input vertices of the transmitter, the output vertices visible at node $i$, and the output vertices visible at the receiver. Then, we have the following relations between $U$, $X_i$ and $Y$:
\small{
\begin{align}
&\begin{bmatrix}
X_1 \\
\vdots \\
X_v
\end{bmatrix}=
\begin{bmatrix}
H_{1,1} K_1 & \cdots & H_{v,1} K_v \\
\vdots & \ddots & \vdots \\
H_{1,v} K_1 & \cdots & H_{v,v} K_v
\end{bmatrix}
\begin{bmatrix}
X_1 \\
\vdots \\
X_v
\end{bmatrix}
+
\begin{bmatrix}
H_{tx,1} \\
\vdots \\
H_{tx,v}
\end{bmatrix}U \nonumber \\
&Y=\begin{bmatrix} H_{1,rx} K_1 & \cdots &  H_{v,rx} K_v \end{bmatrix} \begin{bmatrix}
X_1 \\ \vdots \\ X_v
\end{bmatrix}+H_{tx,rx}U. \nonumber
\end{align}
}
Simple algebra then gives the theorem. Here, the invertibility of the matrix can be shown as follows: As shown in Lemma~\ref{lem:keylemma}, the rank of $(I-\begin{bmatrix}
H_{1,1} K_1 & \cdots & H_{v,1} K_v  \\
\vdots & \ddots & \vdots \\
H_{1,v} K_1  & \cdots & H_{v,v} K_v
\end{bmatrix})$ is the largest rank over all $K_i(z)$. Furthermore, by putting $K_i(z)=0$, the matrix becomes invertible.
\end{proof}

Therefore, from end-to-end perspective, the point-to-point LTI network $\mathcal{N}(z)$ can be thought as a MIMO (multiple-input multiple-output) channel whose channel matrix is $G(z,K)$. It is well-known that the capacity of MIMO channels is closely related to the rank of the channel matrix~\cite{Tse}.
\begin{definition}[Degree of Freedom Capacity]
For a given LTI network $\mathcal{N}(z)$, we say that the degree of freedom (d.o.f.) capacity of the network $\mathcal{N}(z)$ is $k$ if its transfer matrix $G(z,K)$ is rank $k$, i.e. $\rank(G(z,K))=k$. 
\end{definition}
On the other hand, when we ``cut" the nodes into two disjoint sets $V=\{ tx, i_1, \cdots , i_k\}$ and $V^c=\{ rx, i_{k+1}, \cdots, i_v \}$, the channel matrix between these two is defined as
\begin{align}
H_{V,V^c}=
\begin{bmatrix}
H_{tx,rx} & H_{i_1,rx} & \cdots & H_{i_k,rx} \\
H_{tx,i_{k+1}} & H_{i_1,i_{k+1}} & \cdots & H_{i_k,i_{k+1}} \\
\vdots & \vdots & \ddots & \vdots \\
H_{tx,i_v} & H_{i_1,i_v} & \cdots & H_{i_k,i_v} \\
\end{bmatrix}. \nonumber
\end{align}
\begin{definition}[Degree of Freedom Mincut]
For a given LTI network $\mathcal{N}(z)$, we say that the degree of freedom (d.o.f.) mincut of the network $\mathcal{N}(z)$ is $k$ if the minimum rank of cuts is equal to $k$, i.e. $\min_{V:V \subseteq \{0,\cdots,{v+1}\}, V \ni {tx}, V \not\ni {rx} } \rank H_{V,V^c}(z)=k$.
\end{definition}

One key fact about LTI networks is that the well-known mincut-maxflow
theorem~\cite{Ford_Maxflow,Shannon_Maxflow} can be extended to
them. This is one of the main theorem of the paper.
\begin{theorem}[Algebraic Mincut-Maxflow Theorem]
With the above definitions,
\begin{align}
&\rank G(z,K) \nonumber\\
&= \min_{V:V \subseteq \{0,\cdots,{v+1}\}, V \ni {tx}, V \not\ni {rx} } \rank H_{V,V^c}(z). \nonumber
\end{align}\label{thm:mincut}
\end{theorem}
\begin{proof}
See Section~\ref{sec:networklin}.
\end{proof}
In this theorem, $K_i$ are considered as dummy variables which are independent from $z$ and each other. However, what this theorem really implies is the existence of mincut-achieving coding schemes, i.e. there exists $z$-transforms that we can plug in for $K_i$ without changing the equality of Theorem~\ref{thm:mincut}. In Section~\ref{sec:linvsun}, we will discuss this point in further detail.

The above notations for LTI point-to-point networks can be naturally generalized to those for LTI networks with multiple sources and destinations.
\begin{figure}[top]
\begin{center}
\includegraphics[width = 3in]{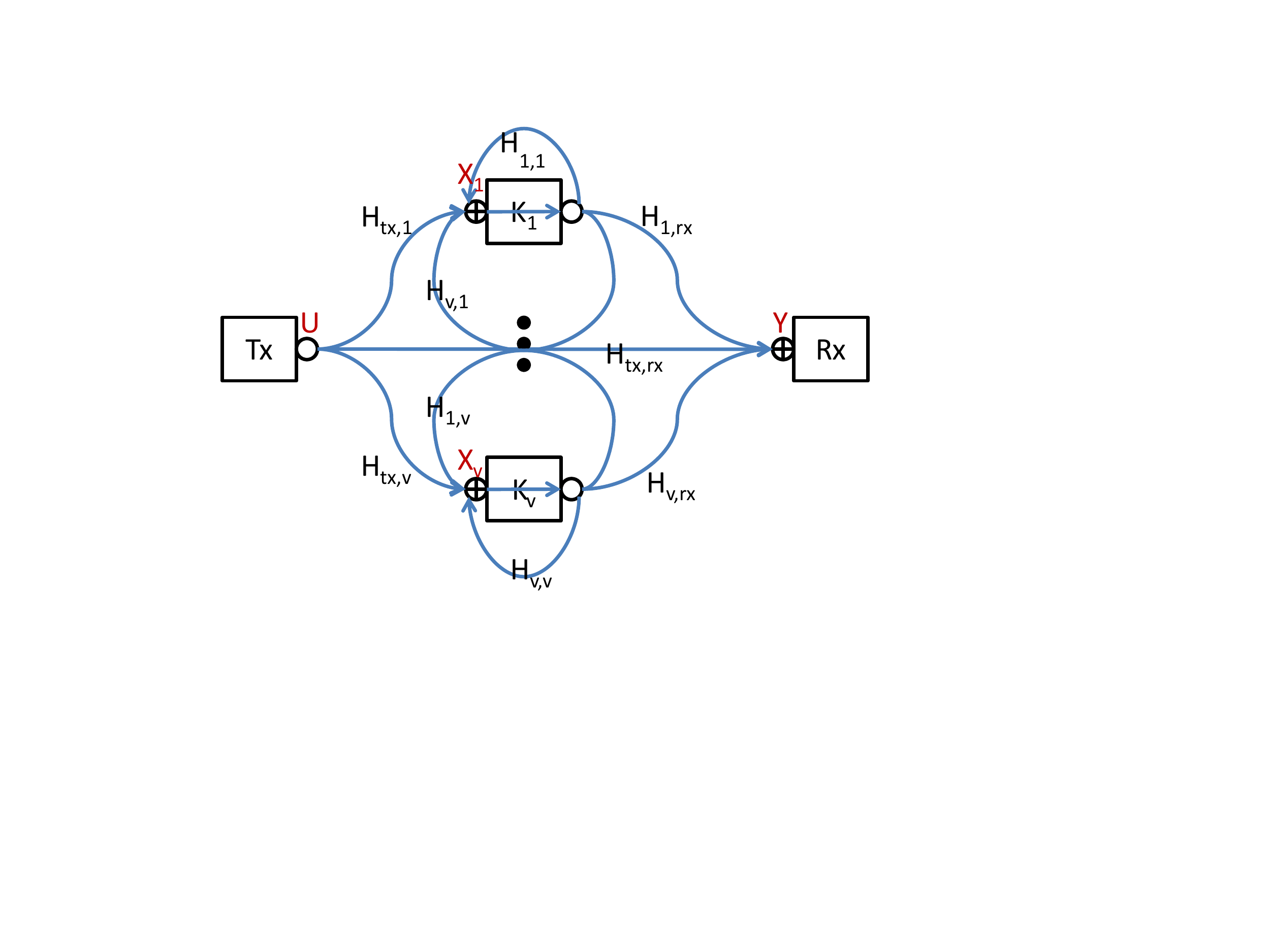}
\caption{point-to-point LTI network $\mathcal{N}(Z)$}
\label{fig:LN_ptop}
\end{center}
\end{figure}
\begin{figure}[top]
\begin{center}
\includegraphics[width = 2in]{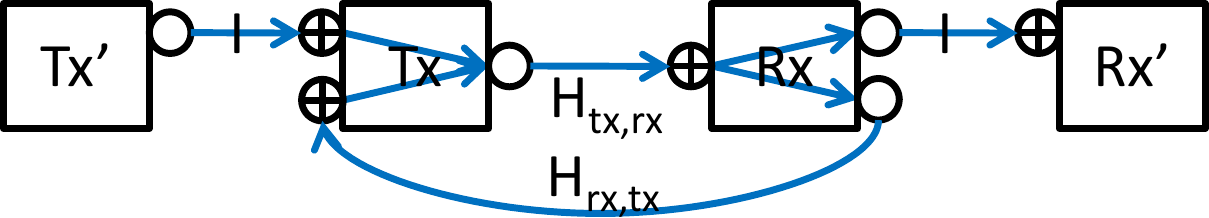}
\caption{We can model feedback by introducing an outer transmitter $Tx'$ and receiver $Rx'$}
\label{fig:feedback}
\end{center}
\end{figure}

One may think the LTI networks above do not cover channels with feedback since we did not include any channel from the receiver to the transmitter. However, as shown in Fig.~\ref{fig:feedback} the channel with feedback can be modeled by introducing an outer transmitter and receiver. In a similar way, we can also include cooperation between transmitters and receivers in cases with multiple sources and destinations.

\subsection{State-Space Representation and Network Linearization}
\label{sec:networklin}
In this section, we prove Theorem~\ref{thm:mincut} using the idea of network linearization. Network linearization is the counterpart of the following fact of linear system theory: Every causal linear time-invariant system with an input $u[n]$ and an output $y[n]$ can be written in state-space form~\cite{Chen}, \emph{i.e.} can be realized as a linear system equation:
\begin{align}
&x[n+1] = A x[n] + B u[n]\\
&y[n] = C x[n] + D u[n]
\end{align}
by introducing proper internal states $x[n]$. Similarly, network linearization tells us that every LTI network with an arbitrary topology can be converted to an acyclic single-hop relay network by introducing proper internal states.

First, we illustrate two key ideas for network linearization.

\begin{figure}[top]
\includegraphics[width = 3in]{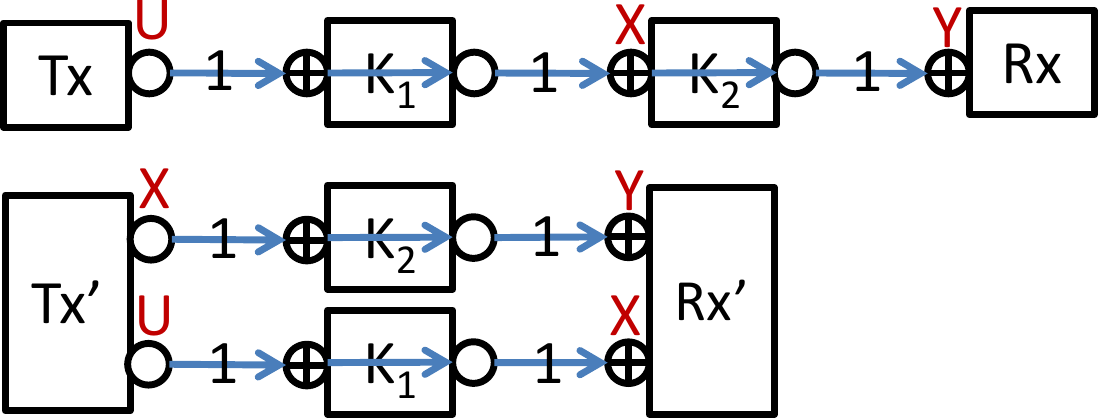}
\caption{LTI network example and its equivalent network with linearized transfer function}
\label{fig:internalexample}
\end{figure}

\begin{figure}[top]
\begin{center}
\includegraphics[width = 3in]{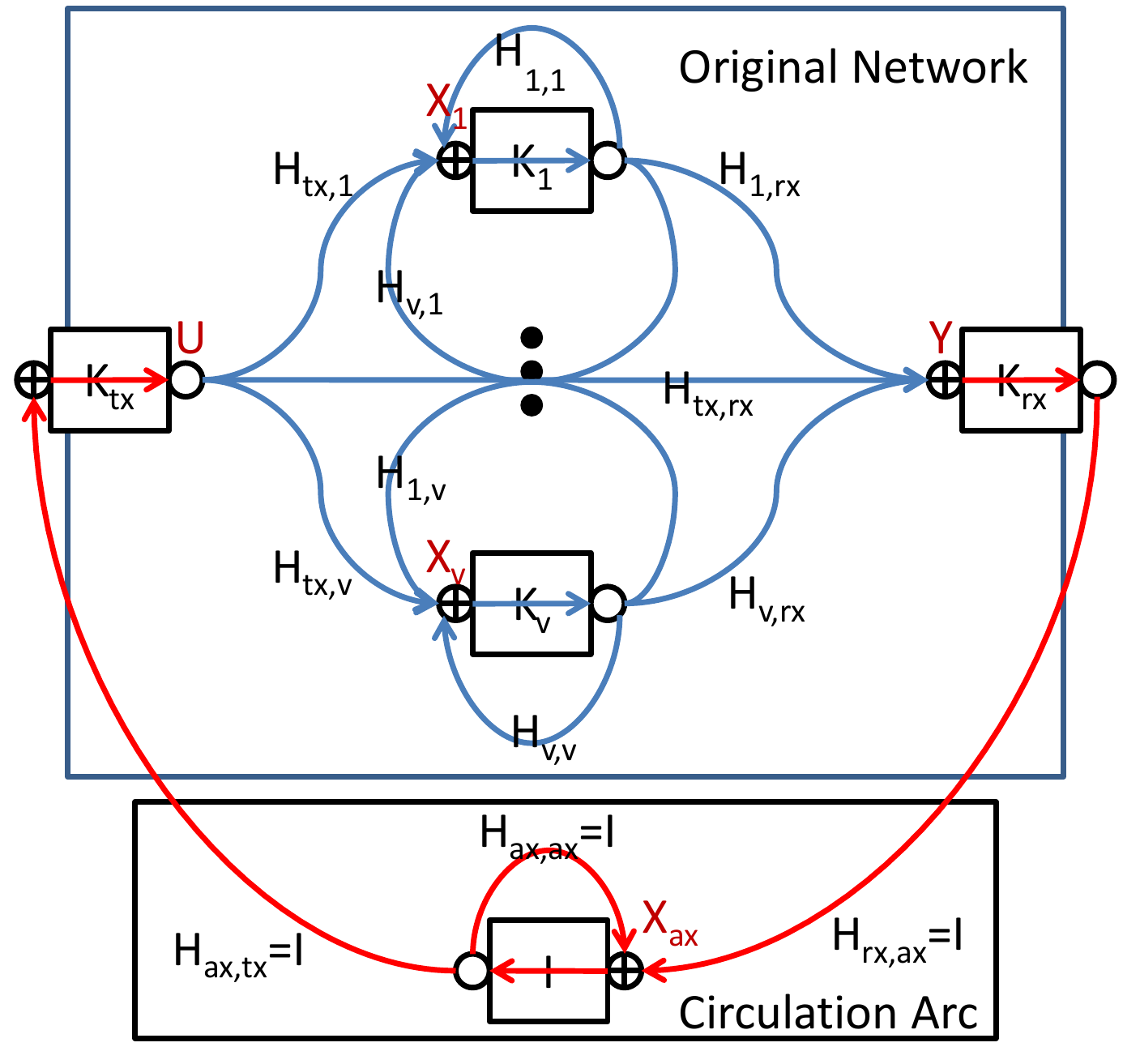}
\caption{LTI network $\mathcal{N}(Z)$ with circulation arc added in}
\label{fig:LN_ptop}
\end{center}
\end{figure}

(1) Internal States:
Consider the two-hop relay network shown in the top figure of Fig.~\ref{fig:internalexample}. The transfer function from $U$ to $Y$ is $k_2 k_1$, which is not linear in $k_1, k_2$. To write the transfer function in a linear form, we introduce an internal state $X$ at the output of the second node. Then, the transfer function matrix from $X,U$ to $Y,X$ is
$\begin{bmatrix}
Y \\ X
\end{bmatrix}=
\begin{bmatrix} k_2 & 0 \\ 0 & k_1\end{bmatrix}
\begin{bmatrix} X \\ U \end{bmatrix}
$, which is linear in $k_1, k_2$. Moreover, since
\begin{align}
\begin{bmatrix}
k_2 & 0 \\
0 & k_1
\end{bmatrix}
=
\begin{bmatrix}
0 \\ 1
\end{bmatrix}
k_1
\begin{bmatrix}
0 & 1
\end{bmatrix}
+
\begin{bmatrix}
1 \\ 0
\end{bmatrix}
k_2
\begin{bmatrix}
1 & 0
\end{bmatrix},
\end{align}
it corresponds to the transfer function of the acyclic single-hop relay network shown in the bottom figure of Fig.~\ref{fig:internalexample}.

(2) Circulation Arc: Even if the transfer function can be written in a linear matrix form by introducing internal states, there has to be a relationship between the rank of the original transfer function and the rank of the linearized transfer function. 

After all, in general the rank of  the linearized transfer function matrix will be bigger as the above example illustrates. So we need a way to relate the ranks of the transfer function matrices.

To make this connection, we borrow the circulation arc idea from the integer programming context~\cite[p.86]{hochbaum}. The problem that they had was that when they tried to write the maxflow problem in linear programming form, the flow conservation law did not hold at the source and the destination. The flow at the source is negative and the flow at the destination is positive. To patch this, they introduced a circulation arc with infinite capacity from the destination to the source. Since the amount of the negative flow at the source is the same as the amount of the positive flow at the destination, the flow conservative law can be recovered as a universal. Moreover, the flow across the network can be easily measured by measuring the flow in the circulation arc.

To apply this idea to LTI networks, we use an underdetermined system. Let's consider $x=x+K_{rx}G(z,K)K_{tx}x$ with unknown vector $x$. Here, $K_{rx}G(z,K)K_{tx}$ is a transfer function with a preprocessing matrix $K_{tx}$ and a postprocessing matrix $K_{rx}$. If the rank of $K_{rx}G(z,K)K_{tx}$ is smaller than the dimension of $x$, the equation is underdetermined. Otherwise, it is not. Thus, we can see that the rank of the transfer function can be measured by the underdeterminedness of the system.\\


Now, we will combine these ideas for network linearization. We first formally introduce the circulation arc. As shown in Fig.~\ref{fig:LN_ptop}, an auxiliary node $N_{ax}$  with $d_{ax}$ input ports and $d_{ax}$ output ports is added to the original network. We also introduce $d_{ax}$ input vertices at the receiver node and $d_{ax}$ output vertices at the transmitter node. Let $H_{rx,ax}=H_{ax,tx}=H_{ax,ax}=K_{ax}=I_{d_{ax}}$. 
As discussed in Section~\ref{sec:prelim}, to reflect the design freedom of the transmitter and receiver, let 
$K_{tx} \in \mathbb{F}[K]^{d_{tx} \times d_{ax}}$ and $K_{rx} \in \mathbb{F}[K]^{d_{ax} \times d_{rx}}$, and each element of $K_{tx}$, $K_{rx}$ is the form of $k_i \in K$ and they are all distinct and also distinct from the elements in $K_1, \cdots, K_v$ inside the relays.

Now, we introduce labels for the internal states. As shown in Fig.~\ref{fig:LN_ptop}, let $X_{ax}$, $X_i$, and $Y$ be the vectors of the signals of the output vertices seen at the auxiliary node, the node $i$, and the receiver respectively.


From the system diagram, Fig.~\ref{fig:LN_ptop}, we can see the following relation has to hold.
\begin{align}
&\begin{bmatrix}
X_{ax} \\
Y \\
X_1 \\
\vdots \\
X_v
\end{bmatrix}
=
\begin{bmatrix}
I_{d_{ax}} & K_{rx} &  0 & \cdots & 0 \\
H_{tx,rx}K_{tx} & 0 & H_{1,rx} K_1 & \cdots & H_{v,rx}K_v \\
H_{tx,1} K_{tx} & 0 & H_{1,1} K_1 & \cdots & H_{v,1} K_v \\
\vdots & \vdots & \vdots & \ddots & \vdots \\
H_{tx,v}K_{tx} & 0 & H_{1,v}K_1 & \cdots & H_{v,v}K_v
\end{bmatrix}
\begin{bmatrix}
X_{ax} \\
Y \\
X_1 \\
\vdots \\
X_v
\end{bmatrix} \nonumber \\
&
(\Leftrightarrow)
\underbrace{\begin{bmatrix}
0 & -K_{rx} & 0 & \cdots & 0 \\
-H_{tx,rx}K_{tx} & I_{d_{rx}} & -H_{1,rx} K_1 & \cdots & -H_{v,rx}K_v \\
-H_{tx,1} K_{tx} & 0 & I_{d_{1,out}}-H_{1,1} K_1 & \cdots & -H_{v,1} K_v \\
\vdots & \vdots & \vdots & \ddots & \vdots \\
-H_{tx,v}K_{tx} & 0 & -H_{1,v}K_1 & \cdots & I_{d_{v,out}}-H_{v,v}K_v
\end{bmatrix}}
_{:=G_{lin}(z,K)}
\begin{bmatrix}
X_{ax} \\
Y \\
X_1 \\
\vdots \\
X_v
\end{bmatrix}
=
\begin{bmatrix}
0 \\ 0 \\ 0 \\ \vdots \\  0
\end{bmatrix} 
\label{eqn:transfer}
\end{align}
The matrix $G_{lin}(z,K)$ here is filled with entries linear in $K_i$. Thus, $G_{lin}(z,K)$ can be rewritten as 
\begin{align}
G_{lin}(z,K)&=
\underbrace{\begin{bmatrix}
0 & 0 & 0 & \cdots & 0 \\
0 & I & 0 & \cdots & 0 \\
0 & 0 & I & \cdots & 0 \\
\vdots & \vdots & \vdots & \ddots & \vdots \\
0 & 0 & 0 & \cdots & I \\
\end{bmatrix}}_{:=A}
+
\underbrace{
\begin{bmatrix}
0 \\ H_{tx,rx} \\ H_{tx,1} \\ \vdots \\ H_{tx,v}
\end{bmatrix}}_{:=B_{tx}}
K_{tx}
\underbrace{
\begin{bmatrix}
-I_{d_{ax}} & 0 & 0 & \cdots & 0
\end{bmatrix}}_{:=C_{tx}}
+
\underbrace{
\begin{bmatrix}
0 \\ H_{1,rx} \\ H_{1,1} \\ \vdots \\ H_{1,v}
\end{bmatrix}}_{:=B_1}
K_1
\underbrace{
\begin{bmatrix}
0 & 0 & -I_{d_{1,out}} & \cdots & 0
\end{bmatrix}}_{:=C_1}+\cdots\\
&+
\underbrace{
\begin{bmatrix}
0 \\ H_{v,rx} \\ H_{v,1} \\ \vdots \\ H_{v,v}
\end{bmatrix}}_{:=B_v}
K_v
\underbrace{
\begin{bmatrix}
0 & 0 & 0 & \cdots & -I_{d_{v,out}}
\end{bmatrix}}_{:=C_v}
+
\underbrace{
\begin{bmatrix}
I_{d_{ax}} \\ 0 \\ 0 \\ \vdots \\ 0
\end{bmatrix}}_{:=B_{rx}}
K_{rx}
\underbrace{
\begin{bmatrix}
0 & -I_{d_{rx}} & 0 & \cdots & 0
\end{bmatrix}}_{:=C_{rx}}. \label{eqn:lintran}
\end{align}
$A, B_{tx}, C_{tx}, B_i, C_i, B_{rx}, C_{rx}$ are defined as above.

Because $G_{lin}(z,K)$ looks like a transfer function matrix, we can formally ask what is the LTI network whose transfer fumtion matrix is $G_{lin}(z,K)$. Then, we can easily see that $G_{lin}(z,K)$ corresponds to the transfer function of the linearized LTI network $\mathcal{N}_{lin}(z)$ of Fig.~\ref{fig:LN_linearized}. The linearized network $\mathcal{N}_{lin}(z)$ has a new transmitter $tx'$ and receiver $rx'$, and is an acyclic single-hop relay network with a direct link between $tx'$ and $rx'$. We also use the subscript $``tx'"$ and $-1$ alternatively, and likewise $``rx'"$ and $v+2$ alternatively. 

Let $d:= \dim \begin{bmatrix} Y \\ X_1 \\ \vdots \\ X_v \end{bmatrix}=d_{rx}+\sum_{1 \leq i \leq v} d_{i,out}$ where $Y, X_1, \cdots, X_v$ are given as \eqref{eqn:transfer}. 
Then, we will prove that the maxflow of $\mathcal{N}_{lin}(z)$ is the same as the maxflow of $\mathcal{N}(z)$ by an offset $d$.

Furthermore, for sets (ordered sets) $V=\{v_1,\cdots, v_i \}$ and $W=\{w_1,\cdots,w_j \}$ we denote
\begin{align}
&B_{V}:=\begin{bmatrix} B_{v_1} & \cdots & B_{v_{i}} \end{bmatrix}\\
&C_{V}:=\begin{bmatrix} C_{v_1} \\ \vdots \\ C_{v_{i}} \end{bmatrix}\\
&D_{V,W}:=\begin{bmatrix} D_{v_1 w_1} & \cdots & D_{v_1 w_{j}} \\ \vdots & \ddots & \vdots \\ D_{v_{i} w_1} & \cdots & D_{v_{i} w_{j}} \end{bmatrix}
\end{align}
whenever this shorthand does not cause confusion.

We also denote the channel matrices from the node $i$ to the node $j$ in the linearized LTI network $\mathcal{N}_{lin}(z)$ as $H^{lin}_{i,j}$. Then, we can easily see that the channel matrix for the cut $V \subseteq \{ 0, \cdots , v+1 \}$ is 
\begin{align}
H^{lin}_{V \cup \{tx'\},V^c \cup \{rx'\}}=\begin{bmatrix}
A & B_V \\ C_{V^c} & 0
\end{bmatrix}. \label{eqn:cutset}
\end{align}
\begin{figure}[top]
\includegraphics[width = 3in]{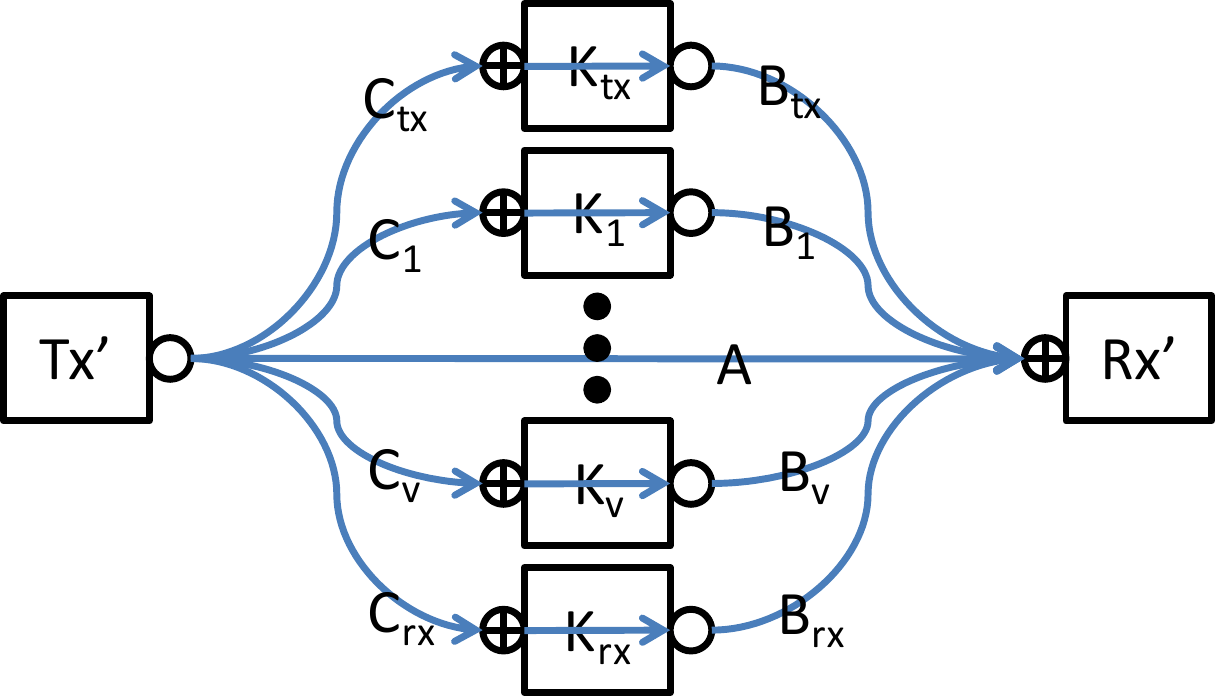}
\caption{Linearized LTI network $\mathcal{N}_{lin}(z)$}
\label{fig:LN_linearized}
\end{figure}
We will prove the essential equivalence between the original network $\mathcal{N}(z)$ and the linearized network $\mathcal{N}_{lin}(z)$. First, we prove a lemma on matrix rank.
\begin{lemma}
For a field $\mathbb{F}$ and $n_1,n_2 \in \mathbb{Z}^+$, let $A \in \mathbb{F}^{n_1 \times n_1}, B \in \mathbb{F}^{n_2 \times n_1}, C \in \mathbb{F}^{n_1 \times n_2}, D \in \mathbb{F}^{n_2 \times n_2}$. If $D$ is invertible, the following rank equality holds.
\begin{align}
\rank \begin{bmatrix} A & B \\ C & D \end{bmatrix} = \rank D + \rank(A-BD^{-1}C) \nonumber
\end{align}
\label{lem:LIN:rank}
\end{lemma}
\begin{proof}
\begin{align}
&\rank \begin{bmatrix} A & B \\ C & D \end{bmatrix} = \rank \left( \begin{bmatrix}I_{n_1} & -BD^{-1} \\ 0 & I_{n_2} \end{bmatrix} \begin{bmatrix} A & B \\ C & D \end{bmatrix} \right) \nonumber \\
&= \rank \begin{bmatrix} A-BD^{-1}C & 0 \\ C & D \end{bmatrix} \nonumber \\
&= \rank D + \rank (A-BD^{-1}C) \nonumber
\end{align}
where the first equality comes from the fact that $\begin{bmatrix}I_{n_1} & -BD^{-1} \\ 0 & I_{n_2} \end{bmatrix} $ is invertible, and the last equality is a consequence of $D$ being invertible.
\end{proof}

Now, we prove that the maxflow of the two networks $\mathcal{N}(z)$ and $\mathcal{N}_{lin}(z)$ are equivalent with an offset $d$.
\begin{lemma}[Maxflow Equivalence Lemma]
Given the above notations,
\begin{align}
\rank(K_{rx} G(z,K) K_{tx})+d=\rank G_{lin}(z,K).
\end{align}
\label{lem:LIN:maxflow}
\end{lemma}
\begin{proof}
\begin{align}
&\rank G_{lin}(z, K)\\
&\overset{(A)}{=}
\rank 
\begin{bmatrix}
I & -H_{1,rx}K_1 & \cdots & -H_{v,rx}K_v \\
0 & I-H_{1,1}K_1 & \cdots & -H_{v,1}K_v \\
\vdots & \vdots & \ddots & \vdots\\
0 & -H_{1,v}K_1 & \cdots & I-H_{v,v}K_v \\
\end{bmatrix}\\
&+\rank\left( - \begin{bmatrix} -K_{rx} & 0 & \cdots & 0 \end{bmatrix}
\begin{bmatrix}
I & -H_{1,rx}K_1 & \cdots & -H_{v,rx}K_v \\
0 & I-H_{1,1}K_1 & \cdots & -H_{v,1}K_v \\
\vdots & \vdots & \ddots & \vdots\\
0 & -H_{1,v}K_1 & \cdots & I-H_{v,v}K_v \\
\end{bmatrix}^{-1}
\begin{bmatrix}
-H_{tx,rx}K_{tx} \\
-H_{tx,1}K_{tx} \\
\vdots \\
-H_{tx,v}K_{tx}
\end{bmatrix}
\right)\\
&\overset{(B)}{=} d+\rank\left( - \begin{bmatrix} -K_{rx} & 0 & \cdots & 0 \end{bmatrix}
\begin{bmatrix}
I & -H_{1,rx}K_1 & \cdots & -H_{v,rx}K_v \\
0 & I-H_{1,1}K_1 & \cdots & -H_{v,1}K_v \\
\vdots & \vdots & \ddots & \vdots\\
0 & -H_{1,v}K_1 & \cdots & I-H_{v,v}K_v \\
\end{bmatrix}^{-1}
\begin{bmatrix}
-H_{tx,rx}K_{tx} \\
-H_{tx,1}K_{tx} \\
\vdots \\
-H_{tx,v}K_{tx}
\end{bmatrix}
\right)\\
&\overset{(C)}{=} d+\rank\left(
K_{rx}
\begin{bmatrix}
I &
\begin{bmatrix}
H_{1,rx}K_1 & \cdots & H_{v,rx}K_v
\end{bmatrix}
\begin{bmatrix}
I-H_{1,1}K_1 & \cdots & -H_{v,1}K_v \\
\vdots & \ddots & \vdots\\
-H_{1,v}K_1 & \cdots & I-H_{v,v}K_v \\
\end{bmatrix}^{-1}
\end{bmatrix}
\begin{bmatrix}
-H_{tx,rx}K_{tx} \\
-H_{tx,1}K_{tx} \\
\vdots \\
-H_{tx,v}K_{tx}
\end{bmatrix}
\right) \\
&\overset{(D)}{=} d+\rank\left( K_{rx}\left( H_{tx,rx} +
\begin{bmatrix}
H_{1,rx}K_1 & \cdots & H_{v,rx}K_v
\end{bmatrix}
\begin{bmatrix}
I-H_{1,1}K_1 & \cdots & -H_{v,1}K_v \\
\vdots & \ddots & \vdots\\
-H_{1,v}K_1 & \cdots & I-H_{v,v}K_v \\
\end{bmatrix}^{-1}
\begin{bmatrix}
H_{tx,1}\\
\vdots \\
H_{tx,v}
\end{bmatrix}
\right) K_{tx}\right)\\
&\overset{(E)}{=} d+\rank(K_{rx}G(z, K)K_{tx})
\end{align}
(A): This comes from Lemma~\ref{lem:LIN:rank} by considering $0_{d_rx}$ as $A$, $\begin{bmatrix} -K_{rx} & 0 & \cdots & 0 \end{bmatrix}$ as $B$, $\begin{bmatrix}
-H_{tx,rx}K_{tx} \\
-H_{tx,1}K_{tx} \\
\vdots \\
-H_{tx,v}K_{tx}
\end{bmatrix}$ as $C$, and $\begin{bmatrix}
I_{d_{rx}} & -H_{1,rx}K_1 & \cdots & -H_{v,rx}K_v \\
0 & I_{d_{1,out}} - H_{1,1}K_1 & \cdots & -H_{v,1}K_v \\
\vdots & \vdots & \ddots & \vdots\\
0 & -H_{1,v}K_1 & \cdots & I_{d_{v,out}} - H_{v,v}K_v \\
\end{bmatrix}$ as $D$. Here, $D$ is invertible, since by Lemma~\ref{lem:keylemma} the rank of $D$ is the maximum rank over all $K_i(z)$ and by putting $K_i(z)=0$ the matrix $D$ becomes full rank.\\
(B): Since each element of $K_i$ is a dummy variable, $\rank \begin{bmatrix}
I_{d_{rx}} & -H_{1,rx}K_1 & \cdots & -H_{v,rx}K_v \\
0 & I_{d_{1,out}} - H_{1,1}K_1 & \cdots & -H_{v,1}K_v \\
\vdots & \vdots & \ddots & \vdots\\
0 & -H_{1,v}K_1 & \cdots & I_{d_{v,out}} - H_{v,v}K_v \\
\end{bmatrix}  \geq \rank \begin{bmatrix} I_{d_{ax}} & 0 & \cdots & 0 \\
0 & I_{d_{1,out}} & \cdots & 0\\
\vdots & \vdots & \ddots & \vdots \\
0 & 0 & \cdots & I_{d_{v,out}} 
\end{bmatrix}  = d $. Moreover, because the dimension of the matrix is $d \times d$, the rank is also upper bounded by $d$.\\
(C): We can easily show $\begin{bmatrix} I & B \\ 0 & D \end{bmatrix}^{-1}=\begin{bmatrix}I & -B D^{-1} \\ 0 & D^{-1} \end{bmatrix}$. Thus, by considering $\begin{bmatrix} - H_{1,rx}K_{rx} & \cdots & - H_{v,rx}K_v \end{bmatrix}$ as $B$ and $\begin{bmatrix}
I_{d_{rx}} & -H_{1,rx}K_1 & \cdots & -H_{v,rx}K_v \\
0 & I_{d_{1,out}} - H_{1,1}K_1 & \cdots & -H_{v,1}K_v \\
\vdots & \vdots & \ddots & \vdots\\
0 & -H_{1,v}K_1 & \cdots & I_{d_{v,out}} - H_{v,v}K_v \\
\end{bmatrix}$ as $D$, and multiplying with the matrix $\begin{bmatrix} -K_{rx} & 0 & \cdots & 0 \end{bmatrix}$, we can prove this step.\\
(D): This comes from direct computation.\\
(E): This comes from the definition of $G(z, K)$ shown in Theorem~\ref{thm:transfer}.
\end{proof}

The mincut of $\mathcal{N}_{lin}(z)$ is also the same as the mincut of $\mathcal{N}(z)$, except for an offset $d$.
\begin{lemma}[Mincut Equivalence Lemma]
Given the above notation, 
\begin{align}
\min\{ \rank K_{tx} , \rank K_{rx} , \min_{W \subseteq \{0,\cdots,v+1 \}, W \ni tx, W \not\ni rx} \rank H_{W,W^c}\}+d=\min_{V \subseteq \{-1,\cdots,v+2 \}, V \ni tx', V \not\ni rx'} \rank H_{V,V^c}^{lin}.
\label{eqn:mincutequi1}
\end{align}
\label{lem:LIN:mincut}
\end{lemma}
\begin{proof}
As we can see in the R.H.S. of \eqref{eqn:mincutequi1}, $V$ is a cut of $\mathcal{N}_{lin}(z)$. We will divide $V$ into three cases: (i) When $tx \in V^c$, (ii) When $rx \in V$, and (iii) When $tx \in V$ and $rx \in V^c$.

For cases (i) and (ii), we will show that the rank of channel matrices is at least $\dim X_{ax} +d$. For case (iii), we will show a one-to-one mapping between the cut $V$ for $\mathcal{N}_{lin}(z)$ and the cut $W$ for $\mathcal{N}(z)$ --- essentially $V$ is a cut of the original network $\mathcal{N}(z)$.

(i) When $tx \in V^c$,

Notice that by definition, we have
\begin{align}
\rank \begin{bmatrix} A \\ C_{tx} \end{bmatrix}=
\begin{bmatrix}
0 & 0 & 0 & \cdots & 0 \\
0 & I_{d_{rx}} & 0 & \cdots & 0\\
0 & 0 & I_{d_{1,out}} & \cdots & 0 \\
\vdots & \vdots & \vdots & \ddots & \vdots \\
0 & 0 & 0 & \cdots & I_{d_{v,out}} \\
-I_{d_{ax}} & 0 & 0 & \cdots & 0
\end{bmatrix}
=\dim X_{ax}+d.
\end{align}

Moreover, whenever $tx \in V^c$, the channel matrix for the cut $H^{lin}_{V,V^c}$ contains $\begin{bmatrix} A \\ C_{tx} \end{bmatrix}$ and so $\rank H^{lin}_{V,V^c} \geq \dim X_{ax}+d$. Thus, we have
\begin{align}
\min_{V \subseteq \{-1,\cdots,v+2 \}, V \ni tx', V \not\ni rx',V^c \ni {tx}} \rank H_{V,V^c}^{lin} \geq \dim X_{ax}+d. \label{eqn:mincutequi2}
\end{align}

Furthermore, by choosing $V=\{ tx'\}$, we have
\begin{align}
\rank H_{tx',\{tx, 1, \cdots, v, rx, rx'\}}^{lin} = \rank \begin{bmatrix} A \\ C_{tx} \\ C_1 \\ \vdots  \\ C_v \\ C_{rx} \end{bmatrix} = \dim X_{ax} +d. \label{eqn:mincutequi3}
\end{align}

Therefore, by \eqref{eqn:mincutequi2} and \eqref{eqn:mincutequi3} we can conclude
\begin{align}
\min_{V \subseteq \{-1,\cdots,v+2 \}, V \ni tx', V \not\ni rx',V^c \ni {tx}} \rank H_{V,V^c}^{lin} = \dim X_{ax}+d.
\end{align}

(ii) When $rx \in V$,

Notice that by definition, we have
\begin{align}
\rank \begin{bmatrix} A & B_{rx} \end{bmatrix}
=
\begin{bmatrix}
0 & 0 & 0 & \cdots & 0 & I_{d_{ax}} \\
0 & I_{d_{rx}} & 0 & \cdots & 0 & 0 \\
0 & 0 & I_{d_{1,out}} & \cdots & 0 & 0 \\
\vdots & \vdots & \vdots & \ddots & \vdots & \vdots \\
0 & 0 & 0 & \cdots & I_{d_{v,out}} & 0 
\end{bmatrix}
=\dim X_{ax}+d.
\end{align}

Moreover, whenever $rx \in V$, the channel matrix for the cut $H^{lin}_{V, V^c}$ contains $\begin{bmatrix} A & B_{rx}\end{bmatrix}$ and so $\rank H^{lin}_{V, V^c} \geq \dim X_{ax} + d$. Thus, we have 
\begin{align}
\min_{V \subseteq \{-1,\cdots,v+2 \}, V \ni tx', V \not\ni rx',V \ni rx} \rank H_{V,V^c}^{lin} \geq \dim X_{ax}+d. \label{eqn:mincutequi100}
\end{align}

Furthermore, by choosing $V = \{tx', tx, 1, \cdots, v, rx\}$, we have
\begin{align}
&\rank H_{\{ tx', tx, 1, \cdots, v, rx \}, rx'}^{lin} =\rank \begin{bmatrix} A & B_{tx} & B_1 & \cdots & B_v & B_{rx} \end{bmatrix} =\dim X_{ax} + d. \label{eqn:mincutequi101}
\end{align}

Therefore, by \eqref{eqn:mincutequi100} and \eqref{eqn:mincutequi101} we can conclude
\begin{align}
\min_{V \subseteq \{-1,\cdots,v+2 \}, V \ni tx', V \not\ni rx',V \ni rx} \rank H_{V,V^c}^{lin}=\dim X_{ax}+d.
\end{align}

(iii) When $tx \in V$ and $rx \in V^c$,

In this case, we will find a one-to-one mapping between the cutset $V$ for $\mathcal{N}^{lin}(z)$ and a cutset $W$ for $\mathcal{N}(z)$, and show that their mincut is the same with an offset of $d$.

Let $W := V\setminus \{ tx' \}$ and $W':=V^c \setminus \{ rx' \}=\{tx', tx, 1, \cdots, v, rx, rx'\} \setminus V \setminus \{ rx'\}$. Now, we will show
\begin{align}
\rank H_{V,V^c}^{lin}=\rank H_{W,W'}+d.\label{eqn:network:1}
\end{align}
However, since the proof of \eqref{eqn:network:1} is not difficult but would be notationally complicated if written out fully, we replace the proof by a representative example.  Let $v=3$ and and $V=\{0,1\}$.
\begin{align}
&\rank H_{V,V^c}^{lin}=\rank \begin{bmatrix} A & B_0 & B_1 \\ C_4 & 0 & 0 \\ C_2 & 0 & 0 \\ C_3 & 0 & 0   \end{bmatrix} \\
&\overset{(A)}{=}\rank
\begin{bmatrix}
0 & 0 & 0 & 0 & 0 & 0 & 0 \\
0 & I_{d_{rx}} & 0 & 0 & 0 & H_{tx,rx} & H_{1,rx} \\
0 & 0 & I_{d_{1,out}} & 0 & 0 & H_{tx,1} & H_{1,1} \\
0 & 0 & 0 & I_{d_{2,out}} & 0 & H_{tx,2} & H_{1,2} \\
0 & 0 & 0 & 0 & I_{d_{3,out}} & H_{tx,3} & H_{1,3} \\
0 & -I_{d_{rx}}& 0 & 0 & 0 & 0 & 0 \\
0 & 0& 0 & -I_{d_{2,out}} & 0 & 0 & 0 \\
0 & 0& 0 & 0 & -I_{d_{3,out}} & 0 & 0 \\
\end{bmatrix}\\
&\overset{(B)}{=}\rank
\begin{bmatrix}
0 & 0 & 0 & 0 & 0 & 0 & 0 \\
0 & 0 & 0 & 0 & 0 & H_{tx,rx} & H_{1,rx} \\
0 & 0 & I_{d_{1,out}} & 0 & 0 & H_{tx,1} & H_{1,1} \\
0 & 0 & 0 & 0 & 0 & H_{tx,2} & H_{1,2} \\
0 & 0 & 0 & 0 & 0 & H_{tx,3} & H_{1,3} \\
0 & -I_{d_{rx}}& 0 & 0 & 0 & 0 & 0 \\
0 & 0& 0 &  -I_{d_{2,out}} & 0 & 0 & 0 \\
0 & 0& 0 & 0 & -I_{d_{3,out}} & 0 & 0 \\
\end{bmatrix}\\
&\overset{(C)}{=}\rank
\begin{bmatrix}
0 & 0 & 0 & 0 & 0 & 0 & 0 \\
0 & 0 & 0 & 0 & 0 & H_{tx,rx} & H_{1,rx} \\
0 & 0 & I_{d_{1,out}} & 0 & 0 & 0 & 0 \\
0 & 0 & 0 & 0 & 0 & H_{tx,2} & H_{1,2} \\
0 & 0 & 0 & 0 & 0 & H_{tx,3} & H_{1,3} \\
0 &-I_{d_{rx}}& 0 & 0 & 0 & 0 & 0 \\
0 & 0& 0 & -I_{d_{2,out}} & 0 & 0 & 0 \\
0 & 0& 0 & 0 & -I_{d_{3,out}} & 0 & 0 \\
\end{bmatrix}\\
&\overset{(D)}{=}\rank
\begin{bmatrix}
0 & 0 & 0 & 0 & 0 & 0 & 0 \\
0 & 0 & 0 & 0 & 0 & H_{tx,rx} & H_{1,rx} \\
0 & 0 & 0 & 0 & 0 & H_{tx,2} & H_{1,2} \\
0 & 0 & 0 & 0 & 0 & H_{tx,3} & H_{1,3} \\
0 & -I_{d_{rx}}& 0 & 0 & 0 & 0 & 0 \\
0 & 0 &  I_{d_{1,out}} & 0 & 0 & 0 & 0 \\
0 & 0& 0 & -I_{d_{2,out}} & 0 & 0 & 0 \\
0 & 0& 0 & 0 &-I_{d_{3,out}} & 0 & 0 \\
\end{bmatrix}\\
&\overset{(E)}{=}
\begin{bmatrix}
H_{tx,rx} & H_{1,rx} \\
H_{tx,2} & H_{1,2} \\
H_{tx,3} & H_{1,3} \\
\end{bmatrix}+d\\
&=\rank H_{W, W'}+d
\end{align}
(A): By the definitions of $A$, $B_i$, $C_i$ shown in \eqref{eqn:lintran}.\\
(B): This comes from elementary row operations to eliminate the $I$'s in the $A$ by using the rows in $C_i$'s. In general, this kind of step will make the $A$ part only have $I$'s at the location corresponding to the set $V$.\\
(C): This comes from elementary column operations to eliminate the $B_i$'s by using the $I$'s in the $A$. In general, this kind of step will make the $B$ part to have $0$'s at the location corresponding to the set $V$.\\
(D): By reordering of the rows so that the $I$'s in the $A$ can be grouped with the $C_i$'s. In general, this kind of step will make the $B$ part to be full-rank.\\
(E): Since we know $\rank \begin{bmatrix} 0 & A \\ B & 0\end{bmatrix}= \rank A  + \rank B$ and by the definitions, $d=d_{rx}+d_{1,out}+d_{2,out}+d_{3,out}$ for this example.

As we can see, we only used elementary row and column operations which hold for general matrices. Thus, we can easily prove that \eqref{eqn:network:1} holds in general by exactly the above argument.

Finally, using (i),(ii) and (iii) we can prove the lemma.
\begin{align}
&\min_{V \subseteq \{-1,\cdots,v+2 \}, V \ni tx', V \not\ni rx'} \rank H_{V,V^c}^{lin}=\min\{ \dim X_{ax}, \min_{W \subseteq \{0,\cdots,v+1 \}, W \ni tx, W \not\ni rx} \rank H_{W, W^c} \}+d\\
&= \min\{ \dim U , \dim Y ,\dim X_{ax}, \min_{W \subseteq \{0,\cdots,v+1 \}, W \ni tx, W \not\ni rx} \rank H_{W, W^c} \}+d \\
&= \min\{ \rank K_{tx}, \rank K_{rx}, \min_{W \subseteq \{0,\cdots,v+1 \}, W \ni tx, W \not\ni rx} \rank H_{W, W^c}  \}+d
\end{align}
Here, the second equality follows from the fact that the mincut of $\mathcal{N}(z)$ is not greater than $\min\{\dim U, \dim Y \}$. The third equality follows from $\rank K_{tx} = \min\{ \dim U , \dim X_{ax} \}$ and $\rank K_{rx} = \min\{ \dim Y , \dim X_{ax} \}$.
\end{proof}

The main advantage of linearized networks is that it is known that the algebraic mincut-maxflow theorem holds for $\mathcal{N}_{lin}(z, K)$~\cite[Theorem 4.1]{Anderson_Algebraic}. Here, we present the theorem with a simpler, self-contained and different proof for completeness.\footnote{The proof of \cite[Theorem 4.1]{Anderson_Algebraic} only uses linear algebraic fact and relates the rank of the matrices with the rank of bigger matrices. However, here by the use of induction we make each step easier to understand.}
\begin{theorem}[Algebraic Mincut-Maxflow Theorem for Linearized Network~\cite{Anderson_Algebraic}]
Given the above notations,
\begin{align}
\rank G_{lin}(z, K) =\min_{V \subseteq \{-1,\cdots,v+2 \}, V \ni tx', V \not\ni rx'} \rank H_{V,V^c}^{lin}  \nonumber
\end{align}
\label{thm:LIN:thm1}
\end{theorem}
\begin{proof}
We saw that the transfer functions and channel matrices of $\mathcal{N}_{lin}(z)$ are given in terms of $A, B_i, C_i$ in \eqref{eqn:lintran} and \eqref{eqn:cutset} respectively. Thus, it is enough to prove that
\begin{align}
\rank(A+\sum_{0 \leq i \leq v+1} B_i K_i C_i)=\min_{V \subseteq \{ 0, \cdots, v+1 \}} \rank \begin{bmatrix} A & B_V  \\ C_{V^c} & 0 \end{bmatrix}. \label{eqn:mincut1}
\end{align}

This is a fact of linear algebra and can be proved in three steps. First, we prove the theorem for networks with a single relay with a scalar input and output, i.e. $v=-1$ and $B_0, C_0$ are vectors (Case (i)). Then, we extend the claim for a single relay with a vector input and output, i.e. $v=-1$ and $B_0, C_0$ are matrices (Case (ii)). Finally, we generalize to multiple relays when $v = 0, 1, 2, \cdots$ (Case (iii)). 

(i) First, consider the case when $v=-1$ and $B_0, C_0$ are vectors i.e. $B_0 \in \mathbb{F}[z]^{m \times 1}$ and $C_0 \in \mathbb{F}[z]^{1 \times m}$. Then, \eqref{eqn:mincut1} reduces to
\begin{align}
&\rank(A+ B_0 K_0 C_0) = \min (\rank \begin{bmatrix} A & B_0 \end{bmatrix}, \rank \begin{bmatrix} A \\ C_0 \end{bmatrix}). \label{eqn:mincut2}
\end{align}
Moreover, since $B_0$ and $C_0$ are vectors, $\min (\rank \begin{bmatrix} A & B_0 \end{bmatrix}, \rank \begin{bmatrix} A \\ C_0 \end{bmatrix})$ is either $\rank (A)$ or $\rank (A) +1$. 

(i-i) When $\min (\rank \begin{bmatrix} A & B_0 \end{bmatrix}, \rank \begin{bmatrix} A \\ C_0 \end{bmatrix})=\rank (A)$. 

In this case, either $\rank \begin{bmatrix}  A & B_0 \end{bmatrix}$ or $\rank \begin{bmatrix} A \\ C \end{bmatrix}$ is equal to $\rank (A)$. Let $\rank \begin{bmatrix} A & B_0 \end{bmatrix}= \rank (A)$. Then, obviously, $\rank(A + B_0 K_0 C_0) \geq \rank (A)$. Moreover, the column space spanned by $B_0$ belongs to the column space spanned by $A$ . Thus, $B_0 K_0 C_0$ cannot increase the rank of the column space and $\rank(A + B_0 K_0 C_0) = \rank (A)$.

When $\rank \begin{bmatrix} A \\ C \end{bmatrix}= \rank (A)$, the proof follows similarly.

(i-ii) When $\min (\rank \begin{bmatrix} A & B_0 \end{bmatrix}, \rank \begin{bmatrix} A \\ C_0 \end{bmatrix})=\rank (A)+1$. 

In this case, $\rank \begin{bmatrix} A & B_0 \end{bmatrix} = \rank \begin{bmatrix} A \\ C_0 \end{bmatrix} = \rank(A) +1$. Moreover, since $B_0$ is a column vector, $\rank (A+B_0 K_0 C_0) \leq \rank (A) +1$. Thus, we only have to prove $\rank (A+ B_0 K_0 C_0) \geq \rank (A) +1$, which is implied by $\rank (A+ B_0  C_0) = \rank (A) +1$. The following claim proves the last statement.

\begin{claim}
Let $A \in \mathbb{F}[z]^{m \times m}$, $b \in \mathbb{F}[z]^{m \times 1}$, and $c \in \mathbb{F}[z]^{1 \times m}$. If
\begin{align}
\rank(A) + 1= \rank \begin{bmatrix} A & b \end{bmatrix}= \rank \begin{bmatrix} A \\ c \end{bmatrix}
\end{align}
then
\begin{align}
\rank(A) +1 = \rank(A+bc).
\end{align}
\label{claim:mincut:1}
\end{claim}
\begin{proof}
Let $\rank(A)=r$. Then, there exist invertible matrices $U$ and $V$ such that
\begin{align}
U A V = \begin{bmatrix} I_r & 0 \\ 0 & 0 \end{bmatrix}.
\end{align}
Denote $\begin{bmatrix} b_1 \\ b_2 \end{bmatrix}:=Ub$ and $\begin{bmatrix} c_1 & c_2 \end{bmatrix}:=cV$ where $b_1$ and $c_1$ are $r \times 1$ column and $1 \times r$ row vectors respectively.

Moreover, since $U$ and $\begin{bmatrix} V & 0 \\ 0 & 1 \end{bmatrix}$ are invertible, we have 
\begin{align}\rank \begin{bmatrix} A & b \end{bmatrix}=\rank(U \begin{bmatrix} A & b \end{bmatrix} \begin{bmatrix} V & 0 \\ 0 & 1 \end{bmatrix})=\rank \begin{bmatrix} UAV & Ub \end{bmatrix}
=\rank \begin{bmatrix} I_r & 0 & b_1 \\ 0 & 0 & b_2 \end{bmatrix}=r+\rank(b_2).
\end{align}
Thus, for $\rank \begin{bmatrix} A & b \end{bmatrix}=\rank (A) +1$ to hold, $b_2$ has to be a non-zero vector. Likewise, $c_2$ also has to be a non-zero vector. 

Finally, we can conclude
\begin{align}
&\rank(A+bc)=\rank(U(A+bc)V)=\rank(\begin{bmatrix}I_r & 0 \\ 0 & 0\end{bmatrix}+\begin{bmatrix}b_1 \\ b_2\end{bmatrix}cV)\\
&=\rank(\begin{bmatrix}I_r & 0 \\ 0 & 0\end{bmatrix}+\begin{bmatrix} 0 \\ b_2 \end{bmatrix} cV ) \label{eqn:lin:claim1}\\
&=\rank(\begin{bmatrix}I_r & 0 \\ 0 & b_2 c_2\end{bmatrix})\label{eqn:lin:claim2}\\
&=\rank(A)+1. \label{eqn:lin:claim3}
\end{align}
\eqref{eqn:lin:claim1}: elementary row operation and $b_2$ is non-zero.\\
\eqref{eqn:lin:claim2}: elementary row operation.\\
\eqref{eqn:lin:claim3}: $b_2$ and $c_2$ are non-zero.
\end{proof}

(ii) Consider the case when $v=-1$ and $B_0, C_0$ are general matrices. 

Like (i), \eqref{eqn:mincut1} reduces to \eqref{eqn:mincut2}. The only difference is now $B_0, C_0 $ can be matrices, and the following claim shows \eqref{eqn:mincut2} still holds. 

\begin{claim}
Let $A \in \mathbb{F}[z]^{m \times m}$, $B_0 \in \mathbb{F}[z]^{m \times r}$, $C_0 \in \mathbb{F}[z]^{q \times m}$, and $K_0 \in \mathbb{F}[K]^{r \times q}$ where each element of $K_0$ is of the form $k_i \in K$ and distinct. Then,
\begin{align}
\rank(A+B_0 K_0 C_0)=\min\{\rank\begin{bmatrix} A & B_0 \end{bmatrix},\rank\begin{bmatrix} A \\ C_0 \end{bmatrix}\}
\end{align}
\label{claim:mincut:2}
\end{claim}
\begin{proof}
Let $x:=\rank \begin{bmatrix} A & B_0 \end{bmatrix}- \rank(A)$ and $y:=\rank \begin{bmatrix} A \\ C_0 \end{bmatrix}-\rank(A)$. 
Then, we can find at least $x$ linearly independent column vectors of $B_0$ which are independent from the columns of $A$, and at least $y$ linearly independent row vectors of $C_0$ which are independent from the rows of $A$. Formally, let $b_1,\cdots,b_x$ and $c_1,\cdots,c_y$ be such vectors, i.e. $b_i$ and $c_j$ are columns and rows of $B_0$ and $C_0$ respectively and $\rank \begin{bmatrix} A & b_1 & \cdots & b_x  \end{bmatrix}= \rank \begin{bmatrix} A & B_0 \end{bmatrix}$, $\rank \begin{bmatrix} A \\ c_1 \\ \vdots \\ c_y \end{bmatrix}=\rank \begin{bmatrix} A \\ C_0 \end{bmatrix}$. Then, we have
\begin{align}
&\rank(A+B_0 K_0 C_0) \geq \rank ( A+ \sum_{1 \leq i \leq \min\{x,y \}}b_i c_i) \label{eqn:lin:claim4}\\
&=\min\{ \rank \begin{bmatrix} A & B_0 \end{bmatrix} , \rank \begin{bmatrix} A \\ C_0 \end{bmatrix} \}. \label{eqn:lin:claim5}
\end{align}
\eqref{eqn:lin:claim4}: We can find a $r \times q$ matrix $K_0'$ such that all the elements of the matrix are $0$ or $1$, and $A+B_0 K_0' C_0=A+\sum_{1 \leq i \leq \min\{x,y \}} b_i c_i$. Moreover, $\rank(A+B_0 K_0 C_0) \geq \rank(A+B_0 K_0' C_0)$ by Lemma~\ref{lem:keylemma}.\\
\eqref{eqn:lin:claim5}: $b_i$ and $c_i$ are independent from the column and row space spanned by $A$ respectively. Furthermore, $b_i$ and $c_i$ are also independent from $b_1, \cdots, b_{i-1}$ and $c_1, \cdots, c_{i-1}$ respectively. Therefore, we can repeatedly apply Claim~\ref{claim:mincut:1} and get the desired result.

Moreover,
\begin{align}
&\rank(A+B_0 K_0 C_0)= \rank ( \begin{bmatrix} A & B_0 \end{bmatrix}\begin{bmatrix} I \\ K_0 C_0 \end{bmatrix} ) \leq \rank \begin{bmatrix} A & B_0 \end{bmatrix} \label{eqn:mincut3}\\
&\rank(A+B_0 K_0 C_0)= \rank ( \begin{bmatrix} I & B_0 K_0 \end{bmatrix}\begin{bmatrix} A \\ C_0 \end{bmatrix} ) \leq \rank \begin{bmatrix} A \\ C_0 \end{bmatrix}. \label{eqn:mincut4}
\end{align}
Therefore, by \eqref{eqn:lin:claim5}, \eqref{eqn:mincut3}, \eqref{eqn:mincut4} the claim is true.
\end{proof}

(iii) The case with multiple relays, i.e. $v = 0, 1, 2, \cdots$ and $B_i, C_i$ are general matrices.

Now, we will prove \eqref{eqn:mincut1} for a general $v$. The proof is an induction in $v = -1, 0, 1, 2, \cdots$. Claim~\ref{claim:mincut:2} shows \eqref{eqn:mincut1} is true for $v=-1$. 
To prove that the theorem also holds for $v= 0, 1, 2, \cdots$, we will assume that the theorem holds for $v=w$ as the induction hypothesis and prove that the theorem holds for $v=w+1$.

First, by applying Claim~\ref{claim:mincut:2} we have
\begin{align}
&\rank(A+\sum_{0 \leq i \leq w+1} B_i K_i C_i)= \rank(A+\sum_{0 \leq i \leq w} B_i K_i C_i + B_{w+1} K_{w+1} C_{w+1})\\
&=\min \{ \rank\begin{bmatrix} A+\sum_{0 \leq i \leq w} B_i K_i C_i & B_{w+1}  \end{bmatrix},
\rank\begin{bmatrix} A+\sum_{0 \leq i \leq w} B_i K_i C_i \\ C_{w+1}  \end{bmatrix} \} \label{eqn:lin:claim6}
\end{align}
Consider the two terms one at a time.
\begin{align}
&\rank\begin{bmatrix} A+\sum_{0 \leq i \leq w} B_i K_i C_i & B_{w+1}  \end{bmatrix}\\
&=\rank ( \begin{bmatrix} A & B_{w+1} \end{bmatrix} + \sum_{0 \leq i \leq w} B_i K_i \begin{bmatrix} C_i & 0 \end{bmatrix} ) \\
&=\min_{W \subseteq \{ 0, \cdots, w\}} \rank \begin{bmatrix} A & B_{w+1} & B_W \\ C_{W^c} & 0 & 0 \end{bmatrix} \label{eqn:lin:claim100}\\
&=\min_{W \subseteq \{ 0, \cdots, w+1 \}, W \ni w+1} \rank \begin{bmatrix} A & B_W \\ C_{W^c} &  0 \end{bmatrix}. \label{eqn:lin:claim7}
\end{align}
where \eqref{eqn:lin:claim100} comes from \eqref{eqn:mincut1} for $v=w$ by replacing $A$ by $\begin{bmatrix} A & B_{v+1}\end{bmatrix}$, $B_i$ by $B_i$, and $C_i$ by $\begin{bmatrix}C_i & 0 \end{bmatrix}$.

Likewise, we can also prove
\begin{align}
&\rank\begin{bmatrix} A+\sum_{0 \leq i \leq w} B_i K_i C_i \\ C_{w+1}  \end{bmatrix}\\
&=\rank\left( \begin{bmatrix} A \\ C_{w+1} \end{bmatrix}+ \sum_{0 \leq i \leq w} \begin{bmatrix}B_i \\ 0 \end{bmatrix} K_i C_i \right) \\
&=\min_{W \subseteq \{ 0, \cdots, w+1 \}, W \not\ni w+1} \rank \begin{bmatrix} A & B_W  \\ C_{W^c} & 0 \end{bmatrix} \label{eqn:lin:claim8}
\end{align}
By plugging \eqref{eqn:lin:claim7} and \eqref{eqn:lin:claim8} to \eqref{eqn:lin:claim6}, we have
\begin{align}
\rank(A+\sum_{0 \leq i \leq w+1} B_i K_i C_i)=\min_{W \subseteq \{ 0, \cdots, w+1 \}} \rank \begin{bmatrix} A & B_W  \\ C_{W^c} & 0 \end{bmatrix}.
\end{align}
Therefore, by induction the theorem is true.
\end{proof}

So far, we discussed how to convert general topology networks into standardized networks --- linearized networks (networks shown in Fig.~\eqref{fig:LN_ptop} to linearized networks shown in Fig.~\ref{fig:LN_linearized}). Moreover, we discovered that the mincuts and maxflows of two networks are equivalent with an offset (Lemma~\ref{lem:LIN:maxflow} and Lemma~\ref{lem:LIN:mincut}). Thus, using the mincut-maxflow theorem for linearized networks (Theorem~\ref{thm:LIN:thm1}), we can prove the algebraic mincut-maxflow theorem for general LTI networks.

\begin{proof}[Proof of Theorem~\ref{thm:mincut}]
Since we can arbitrarily choose $d_{ax}$, let $d_{ax} \geq max\{d_{tx},d_{rx}\}$. Then,
\begin{align}
\rank G(z,K)&=\rank(K_{rx} G(z,K) K_{tx}) \label{eqn:lin:thm0}\\
&=\rank G_{lin}(z,K)-d  \label{eqn:lin:thm1}\\
&=\min_{V \subseteq \{-1,\cdots,v+2 \}, V \ni tx', V \not\ni rx'} \rank H_{V,V^c}^{lin} -d \label{eqn:lin:thm2}\\
&=\min\{  \rank K_{tx} , \rank K_{rx}, \min_{V \subseteq \{0,\cdots,v+1 \}, V \ni tx, V \not\ni rx} \rank H_{V,V^c}  \} \label{eqn:lin:thm3}\\
&=\min_{V \subseteq \{0,\cdots,v+1 \}, V \ni tx, V \not\ni rx} \rank H_{V,V^c}. \label{eqn:lin:thm4}
\end{align}
\eqref{eqn:lin:thm0} is due to the following fact: Select $K_{rx}(z)$ as a $0-1$ matrix that chooses
$\rank G(z,K)$ independent rows of $G(z,K)$ and $K_{tx}(z)$ as a $0-1$ matrix that chooses $\rank G(z, K)$ independent columns of $K_{rx}G(z, K)$. Then, the rank of the resulting matrix $K_{rx}(z) G(z,K) K_{tx}(z)$ is $\rank G(z,K)$. Therefore, \eqref{eqn:lin:thm0} follows from Lemma~\ref{lem:keylemma}.\\
\eqref{eqn:lin:thm1}, \eqref{eqn:lin:thm2} and \eqref{eqn:lin:thm3} follow from Lemma~\ref{lem:LIN:maxflow}, Theorem~\ref{thm:LIN:thm1} and Lemma~\ref{lem:LIN:mincut} respectively.\\
\eqref{eqn:lin:thm4} follows from the fact that the mincut of $\mathcal{N}(z)$ is not greater than $\min \{d_{tx},d_{rx} \}$, $\rank K_{tx}=d_{tx}$ and $\rank K_{rx}=d_{rx}$.
\end{proof}

Remark: Part of Theorem~\ref{thm:mincut} was already known in \cite{Koetter_Algebraic} and \cite{Kim_Algebraic}. In fact, the main insight of the theorem is indebted to Koetter and Medard's algebraic framework of network coding~\cite{Koetter_Algebraic}. However, the scope of the paper~\cite{Kim_Algebraic} is traditional networks with orthogonal links, and the proof of the theorem is a corollary from Ford-Fulkerson algorithm~\cite{Ford_Maxflow}. Later, Kim and Medard~\cite{Kim_Algebraic} extended the algebraic framework to the deterministic model~\cite{Salman_Wireless} using hypergraph ideas, and proved the theorem using Ford-Fulkerson algorithm on hypergraphs~\cite{Medard_personal}. Their idea provides an interesting alternative view to the theorem, and is worth for a formal and rigorous study given that the details in \cite{Kim_Algebraic} were omitted due to space limits. However, the model in \cite{Kim_Algebraic} is still not general enough for LTI networks since it only covers the case when the channel gains are $0$ or $1$ and field sizes are finite. Moreover, sometimes it is not clear how to convert general LTI networks to equivalent graphs (or hypergraphs).

\subsection{Network Linearization vs. Network Unfolding}
\label{sec:linvsun}

\begin{figure}[top]
\includegraphics[width = 3in]{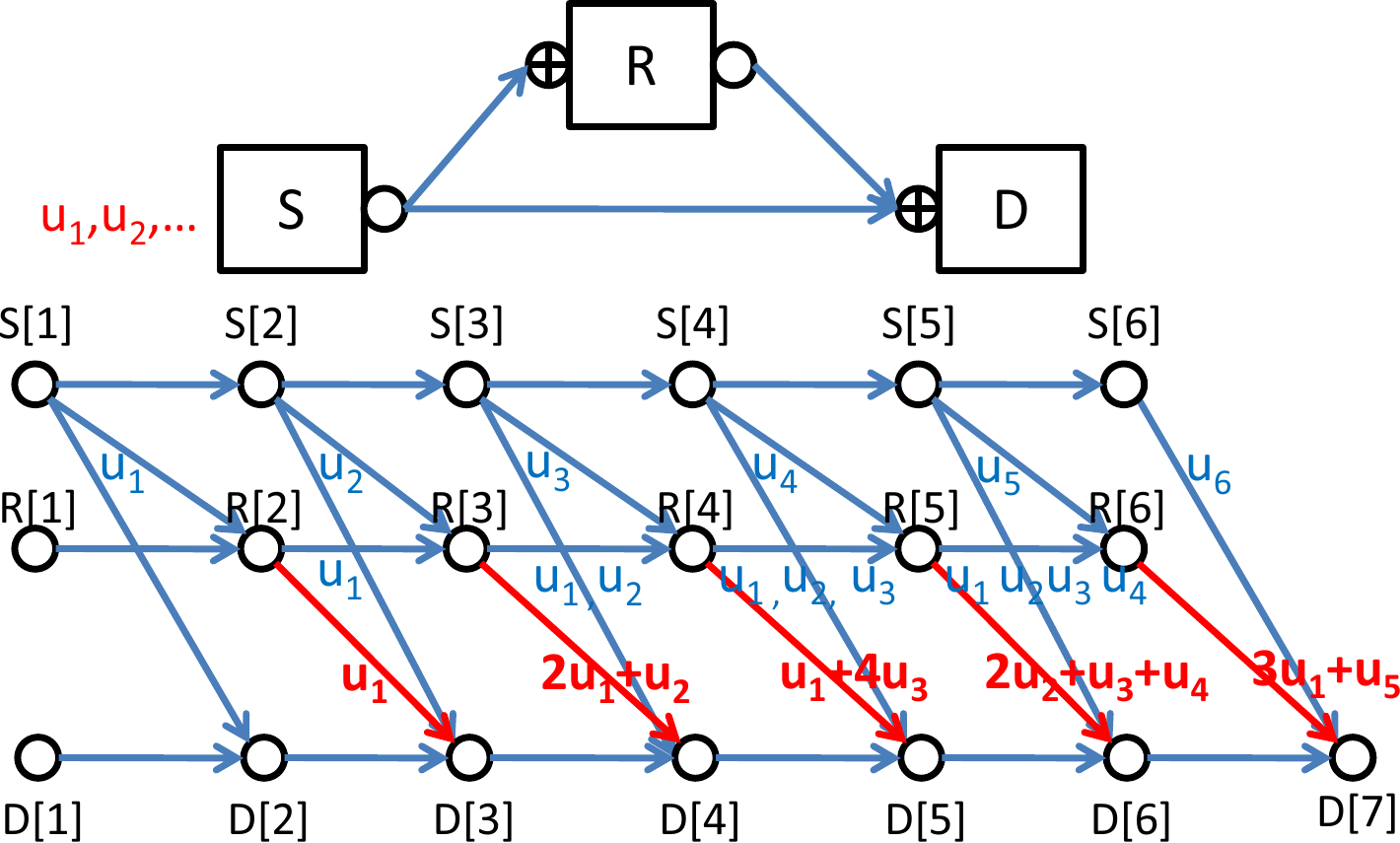}
\caption{Simple Relay Network and the corresponding unfolded network with a mincut-achieving linear scheme}
\label{fig:networkunfolding}
\end{figure}

\begin{figure}[top]
\includegraphics[width = 3in]{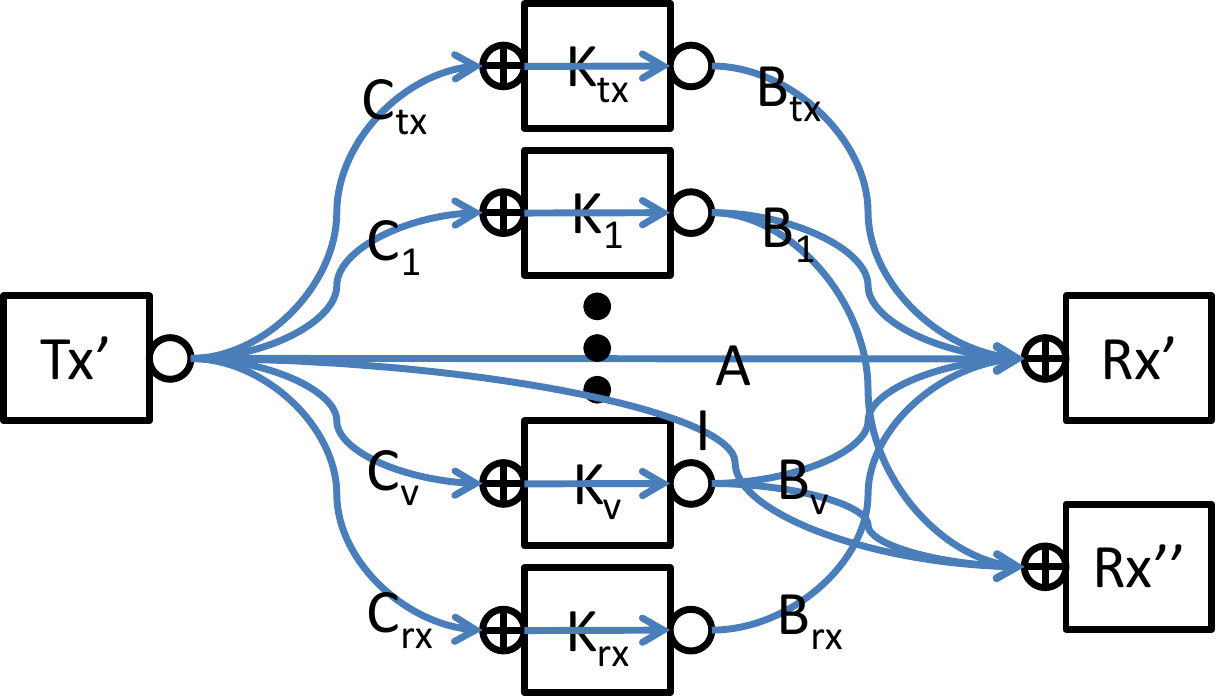}
\caption{Linearized LTI network $\mathcal{N}_{lin}'(Z)$ with an additional destination $Rx''$}
\label{fig:networkunfoldinglin2}
\end{figure}

\begin{figure}[top]
\includegraphics[width = 2in]{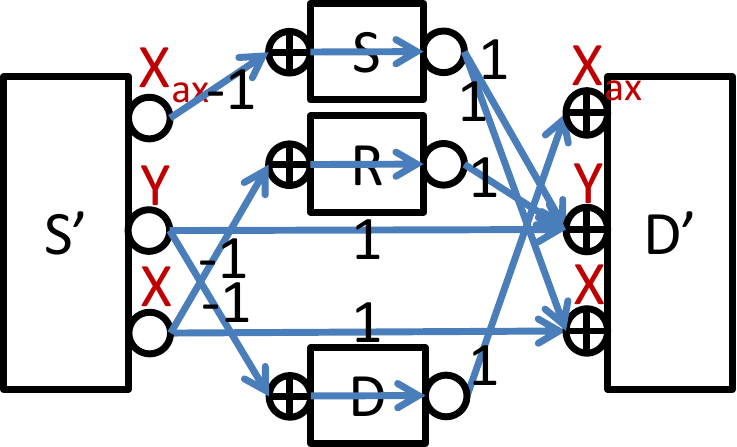}
\caption{Linearized Network of the example in Fig.~\ref{fig:networkunfolding}}
\label{fig:networkunfoldinglin}
\end{figure}

We proposed network linearization as a way of ``converting" an arbitrary relay network to an equivalent acyclic single-hop relay network. In this section, we will compare network linearization with the previously known idea, network unfolding.

Network unfolding is proposed in \cite{Ahlswede_Network} to convert arbitrary networks to layered networks in which the only existing edges are from one layer to the next layer. As we can see in Fig.~\ref{fig:networkunfolding}, by introducing duplicated nodes over the time, any arbitrary network can be thought as a layered network. Moreover, the capacity of the layered network approaches the capacity of the original network as the time expansion gets large. Since layered networks have a quite attractive and simple topology, a series of works~\cite{Goemans_Algorithmic,Amaudruz_Combinatorial,Savari_maxflow} have exclusively focused on them and developed algorithms that find deterministic linear schemes for layered networks.

However, what these papers are overlooking is that when we fold the unfolded network back into its physical topology the time-invariant scheme might become a time-varying scheme. The example shown in Fig.~\ref{fig:networkunfolding} shows that a network-coding design based on an unfolded network can cause significant problems even in the simple network with one source, one relay and one destination. The source transmits $u_1,\cdots,u_6$ to the destination. The letters on the arrows of the unfolded network represent the flows of information. We can easily check that the network-coding scheme shown in the figure is mincut achieving.

However, when we fold it back, we can see problems for implementation. First of all, the scheme is time-varying at the relay. Thus, for the scheme to work every node in the network has to be synchronized to a common clock. Moreover, the transmitted signal at a given time step may depend on all of its previously received signals, which may require a large memory.

On the other hand, from the algebraic mincut-maxflow theorem (Theorem~\ref{thm:mincut}) we can conclude that there exists a mincut achieving LTI scheme by using the same argument used in \cite{Koetter_Algebraic}. By Lemma~1 of \cite{Koetter_Algebraic} when the field size is large enough there exist $K_i$ that achieve the mincut of the network. Moreover, when the underlying field $\mathbb{F}$ are the reals $\mathbb{R}$ or complex $\mathbb{C}$, these fields already have infinite number of elements and there exist channel gain matrices which achieve the mincut of the network. 
When the fields are finite, by extending $\mathbb{F}$ to $\mathbb{F}^m$ we can guarantee a large-enough field size. Furthermore, we even do not have to extend the field when $K_{tx}, K_i, K_{rx}$ are allowed to have memory. $\mathbb{F}[z]$, the field of rational functions in $z$ with coefficient from $\mathbb{F}$, is already an infinite field. Like Lemma~1 of \cite{Koetter_Algebraic} we can prove that there exist mincut-achieving casual\footnote{Notice that even if we put the causal restriction on the design of $K_i$, the dimensions of the algebraic varieties remain the same. Thus, the proof argument for Lemma~1 of \cite{Koetter_Algebraic} still holds.} LTI filters, $K_{tx}, K_i, K_{rx}$, whose elements are from $\mathbb{F}[z]$, i.e. having memory is equivalent to extending a field size.

However, we have to be careful to use the network linearization idea for the actual design of the gain matrices $K_i$, i.e. when we are choosing the elements of $K_i$ from $\mathbb{F}[z]$ and plugging them. 
The reason is we also have to guarantee the existence of the transfer function, which is the invertibility of $\begin{bmatrix} I-H_{1,1}K_1 & \cdots & -H_{v,1}K_v \\ \vdots & \ddots & \vdots \\ -H_{1,v}K_1 & \cdots & I - H_{v,v} K_v \end{bmatrix}$ as shown in Theorem~\ref{thm:transfer}.

Fortunately, this condition can be also posed as a part of the LTI communication network problem.
We can easily see that the condition is equivalent to the invertibility of
$\begin{bmatrix} I_{d_{ax}} & 0 & 0 & \cdots & 0 \\ 0 & I & -H_{1,rx}K_1 & \cdots & -H_{v,rx}K_v \\
0 & 0 & I-H_{1,1}K_1 & \cdots & -H_{v,1}K_v \\ \vdots & \vdots & \vdots & \ddots & \vdots \\
0 & 0 & -H_{1,v}K_1 & \cdots & I-H_{v,v}K_v\end{bmatrix}$.
This matrix further equals to $I+B_1 K_1 C_1 + \cdots + B_v K_v C_v$ by the definitions in \eqref{eqn:lintran}. We can see the maximum rank (and the dimension) of $I+B_1 K_1 C_1 + \cdots + B_v K_v C_v$ over all $K_i$ is $d_{ax}+d$. Therefore, the invertibility of the matrix can be thought as the mincut achieving condition from $Tx'$ to $Rx''$ in Figure~\ref{fig:networkunfoldinglin2}. Finally, we can notice that by choosing $d_{ax}$ as the d.o.f. mincut of $\mathcal{N}(z)$, the maxflow from $Tx'$ to both $Rx'$ and $Rx''$ becomes $d+d_{ax}$.

\begin{theorem}
Given the above definitions of $\mathcal{N}(z)$ and $\mathcal{N}_{lin}'(z)$, let's choose $d_{ax}$ as the d.o.f. mincut of $\mathcal{N}(z)$. Then, all the multicast network gains $K_i(z) \in \mathbb{C}^{d_{i,in} \times d_{i,out}}$ which achieve the mincut of $\mathcal{N}_{lin}'(z)$ to both receivers $Rx'$ and $Rx''$ can also achieve the mincut of $\mathcal{N}(z)$.
\end{theorem}
\begin{proof}
The proof follows essentially the same as Lemma~\ref{lem:LIN:maxflow} only by replacing $K_i$ with $K_i(z)$. The existence of the transfer function comes from the mincut achievability of $Rx''$ as discussed above.
\end{proof}

Therefore, we can find a mincut-achieving LTI network coding scheme of $\mathcal{N}(z)$ as follows: (i) Select $d_{ax}$ of \eqref{eqn:transfer} as the d.o.f. mincut of $\mathcal{N}(z)$. (ii) Find a mincut-achieving multicast network coding scheme for the linearized network $\mathcal{N}_{lin}'(z)$ of Figure~\ref{fig:networkunfoldinglin2} with two receivers. (iii) Apply $K_i$ obtained in the previous steps to the original network.

Furthermore, it is well known that when the network is acyclic, the transfer function always exists~\cite[Lemma~2]{Koetter_Algebraic}. Therefore, when the network $\mathcal{N}(z)$ is acyclic, the receiver $Rx''$ in $\mathcal{N}_{lin}'(z)$ which was introduced to guarantee the existence of the transfer function is redundant.



Fig.~\ref{fig:networkunfoldinglin} shows the linearized network of the example in Fig.~\ref{fig:networkunfolding}. By the above argument, any LTI scheme of $S$, $R$, $D$ that makes the d.o.f. capacity from $S'$ to $D'$ be $3$ achieves the mincut of the original network. For instance, $S=1$, $R=1$, $D=1$ achieve the mincut of both networks of Fig.~\ref{fig:networkunfolding} and Fig.~\ref{fig:networkunfoldinglin}.

Network linearization can also be extended to general information flows, multicast, broadcast, and unicast. Multicast problems will also posed as a multicast problems even after network linearization. However, broadcast and unicast problems will be posed as secrecy problems where eavesdroppers reflect unintended messages in the original problems. We refere Appendix~\ref{app:networklin} for further discussions.

\section{Preliminaries on Decentralized Control}
\label{sec:prelim2}
In the previous section, we introduced network linearization based on the internal states and circulation arcs. As we mentioned, the internal states idea came from linear system theory. Moreover, once we introduce the circulation arc as Fig.~\ref{fig:LN_ptop}, the whole system becomes a closed-loop system, and such closed-loop systems are the main interest of control theory.
Therefore, we can consider control theory from the communication(network coding) perspective.
First, we review several known facts on decentralized linear system theory --- when the system is stabilizable --- and introduce a few concepts to LTI communication networks.

\subsection{Decentralized Linear System}
\label{sec:decentrallinear}
Decentralized linear systems have multiple controllers, each of which
has access to its own observations and generates its own control inputs. Formally, the
decentralized linear system, $\mathcal{L}(A,B_i,C_i)$, is defined as
follows:\footnote{In this paper, we consider discrete-time systems
  since they are conceptually easier to connect to communication theory. We believe
that the underlying phenomena discussed here also exist in
continuous-time. Furthermore, we assume the matrices here are complex
since we will use the Jordan form which can be complex. However, if
the system were real we could prove corresponding results restricting the
controller design to be real without changing the stabilizability
condition.}
\begin{align}
&x[n+1]=A x[n]+B_1 u_1[n] + \cdots+ B_v u_v[n] \\
&y_1[n]= C_1 x[n]  \nonumber \\
&\vdots  \nonumber \\
&y_v[n]= C_v x[n] \nonumber
\end{align}
where $A \in \mathbb{C}^{m \times m}, B_i \in \mathbb{C}^{m \times
  q_i}$  and $C_i \in \mathbb{C}^{r_i \times m}$. Then, an interesting
question is under what conditions such systems are
stabilizable using only LTI controllers:
\begin{definition}[Stabilizability]
A decentralized linear system is called LTI-stabilizable if there exist linear time-invariant (LTI) controllers $\mathcal{K}_i$ (possibly with internal memories) that connect $y_i$ to $u_i$ whose resulting closed-loop system has only stable poles.
\end{definition}

The stabilizability condition for a decentralized linear system is given in \cite{Wang_Stabilization} using the concept of fixed modes.
\begin{definition}\cite[Definition 2]{Wang_Stabilization}
$\lambda$ is called a fixed mode of $\mathcal{L}(A,B_i,C_i)$ if $\lambda \in \bigcap_{K_i \in \mathbb{C}^{q_i \times r_i}} \sigma(A+\sum_{1 \leq i \leq v}B_i K_i C_i)$ where $\sigma(\cdot)$ is the set of eigenvalues of the matrix.
\label{def:fixed1}
\end{definition}

The intuition behind this definition is that if an eigenvalue is fixed
for all choices of (memoryless) controllers, this eigenvalue is either
unobservable or uncontrollable. Thus, if we have unstable fixed modes,
we cannot stabilize the plant.

\begin{theorem}\cite[Theorem 1]{Wang_Stabilization}
$\mathcal{L}(A,B_i,C_i)$ is stabilizable if and only if all of its
fixed modes are within the unit circle.
\label{thm:stability}
\end{theorem}

Therefore, the stabilizability of linear systems is determined by the existence of unstable fixed modes, and the characterization of stabilizability reduces to the characterization of the fixed modes.

However, the characterization of fixed modes shown in Definition~\ref{def:fixed1} involves an intersection over an infinite number of sets. Therefore, Anderson~\textit{et al.} found the following algebraic characterization of fixed modes~\eqref{eqn:intro:1} which only involves minimization over a finite set~\cite{Anderson_Algebraic}. 

\begin{theorem}
$\lambda$ is a fixed mode of $\mathcal{L}(A,B_i,C_i)$ if and only if
\begin{align}
\min_{V \subseteq \{ 1,2, \cdots, v\}} \rank \begin{bmatrix} A-\lambda I & B_V \\ C_{V^c} & 0 \end{bmatrix} \geq \dim(A).
\end{align}
\label{def:fixed2}
\end{theorem}
   
In other words, two characterization of fixed modes shown in Definition~\ref{def:fixed1} and Theorem~\ref{def:fixed2} are equivalent. In the following discussion, we will see this equivalence turns out to be a special case of the mincut-maxflow theorem for LTI networks.

\subsection{LTI Communication Networks at specific frequencies}
Since the channel gain of LTI networks are given in $z$-transform, write the network as $\mathcal{N}(z)$. We will also consider an LTI network, $\mathcal{N}(z)$, at a specific generalized frequency, $z=\lambda$. To indicate that the LTI network is considered at the generalized frequency $z=\lambda$, we write the network as $\mathcal{N}(\lambda)$. $\mathcal{N}(\lambda)$ implies all $z$ in the LTI network are replaced by $\lambda$. Then, the capacity definition is naturally generalized to $\mathcal{N}(\lambda)$.
\begin{definition}
For a given LTI network $\mathcal{N}(z)$, we say that the degree of freedom (d.o.f.) capacity of the network $\mathcal{N}(z)$ is $k$ at frequency $z=\lambda$ if its transfer matrix $G_{tx,rx}(\lambda,K_i)$ is rank $k$.
\end{definition}
Here we can see that the transfer matrix only makes sense at
$z=\lambda$ when it does not have a pole at $\lambda$. Thus, we
assume that $H_{i,j}$ has no pole at $z=\lambda$. Then, the algebraic mincut-maxflow theorem also
holds for $\mathcal{N}(\lambda)$ as before.
\begin{corollary}
Given the LTI network $\mathcal{N}(\lambda)$ with no poles at $\lambda$ in the $H_{ij}(z)$,
\begin{align}
&\rank(G_{tx,rx}(\lambda,K)) \nonumber \\
&= \min_{V \subseteq \{tx,1,\cdots,v,rx \}, V \ni tx, V \not\ni rx} \rank(H_{V,V^c}(\lambda)).\nonumber
\end{align}
\label{cor:mincut}
\end{corollary}
\begin{proof}
Since the $H_{i,j}(z)$ do not have any pole at $\lambda$, we can apply Theorem~\ref{thm:mincut} with the channel matrices $H_{i,j}(\lambda)$.
\end{proof}
Before we discuss the externalization of implicit communication in
decentralized linear systems, it is helpful to define a standard
network we will repeatedly encounter later.
\begin{figure}
\includegraphics[width = 3in]{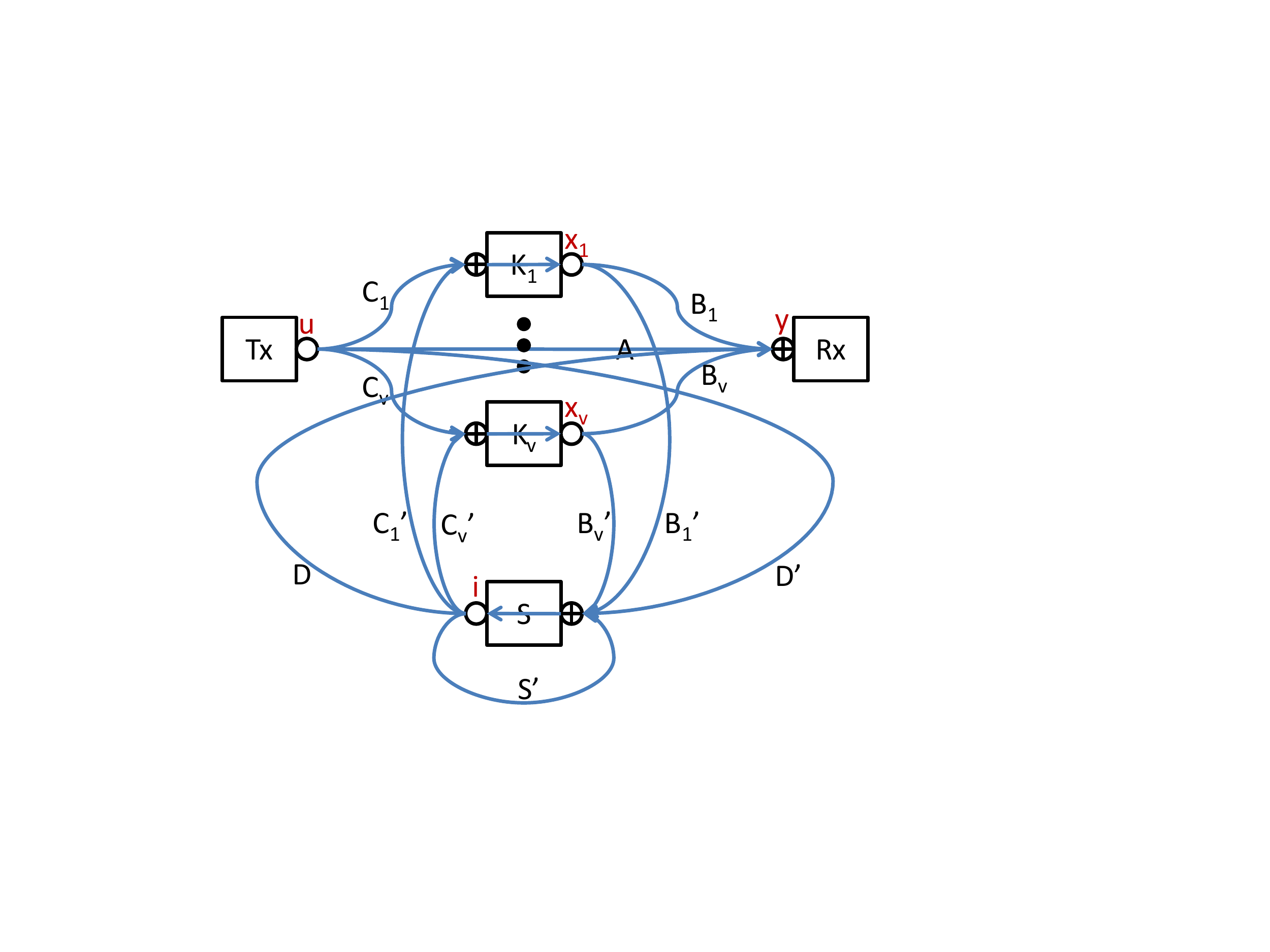}
\caption{Standard LTI Network}
\label{fig:standard}
\end{figure}
\begin{definition}
The LTI network shown in Fig.~\ref{fig:standard} is called a standard
LTI network, $\mathcal{N}_s(A;B_i,B_i';C_i,C_i';D,D';S,S')$.
\end{definition}
The transfer matrix and the channel matrices of the standard network are given as follows.
\begin{lemma}
In the standard network of Fig.~\ref{fig:standard}, the transfer
matrix from the transmitter to the receiver is given as
\begin{align}
G_{tx,rx}=&A+B_1 K_1 C_1 + \cdots + B_v K_v C_v   \nonumber \\
&+(D + B_1 K_1 C_1' + \cdots + B_v K_v C_v') \nonumber \\
&\cdot(S^{-1}-(S'+B_1' K_1 C_1'+\cdots + B_v' K_v C_v' ))^{-1} \nonumber \\
&\cdot(D' + B_1' K_1 C_1 + \cdots + B_v' K_v C_v). \nonumber
\end{align}
The channel matrices $H$ between the transmitter, the relays and the receiver are given for $1\leq i,j \leq v$:
\begin{align}
&H_{tx,rx}=A+D(S^{-1}-S')^{-1}D', \nonumber\\
&H_{tx,i}=C_i+C_i'(S^{-1}-S')^{-1}D', \nonumber\\
&H_{i,rx}=B_i+D(S^{-1}-S')^{-1}B_i', \nonumber\\
&H_{i,j}=C_j'(S^{-1}-S')^{-1}B_i'. \nonumber
\end{align}
Here, we just assume the appropriate inverse matrices exist.
\label{lem:LTI:trans}
\end{lemma}
\begin{proof}
Assign $u$, $x_i$, $i$ and $y$ as we can see in Fig.~\ref{fig:standard}. Then, we can find the following relationships between these:
\begin{align}
&y=B_1 x_1 + \cdots + B_v x_v + Au + Di \label{eqn:std1} \\
&x_1= K_1 C_1 u + K_1 C_1' i \label{eqn:std2}\\
&\vdots \nonumber \\
&x_v= K_v C_v u + K_v C_v' i \nonumber \\
&i=S S' i + S B_1' x_1 + \cdots + S B_v' x_v + S D' u \label{eqn:std3}
\end{align}
By \eqref{eqn:std2} and \eqref{eqn:std3}, we have the following relation:
\begin{align}
&i= S S' i + ( S B_1' K_1 C_1 + \cdots + S B_v' K_v C_v )u +  ( S B_1' K_1 C_1' + \cdots + S B_v' K_v C_v' )i + S D' u \\
&(\Leftrightarrow) S^{-1} i = (S'+B_1' K_1 C_1' + \cdots + B_v' K_v C_v' )i + (D'+B_1' K_1 C_1 + \cdots + B_v' K_v C_v) u \\
&(\Leftrightarrow) i=(S^{-1}-(S+B_1'K_1 C_1'+ \cdots + B_v' K_v C_v'))^{-1} (D'+B_1' K_1 C_1 + \cdots + B_v' K_v C_v) u \label{eqn:std4}
\end{align}
By plugging \eqref{eqn:std2} and \eqref{eqn:std4} into \eqref{eqn:std1}, we get the transfer function from the transmitter to the receiver.\\
One can easily check the channel matrices between nodes.
\end{proof}

\section{Example: Information Flow in a Decentralized Linear System}

\begin{figure}

\begin{center}
\includegraphics[width = 3.5in]{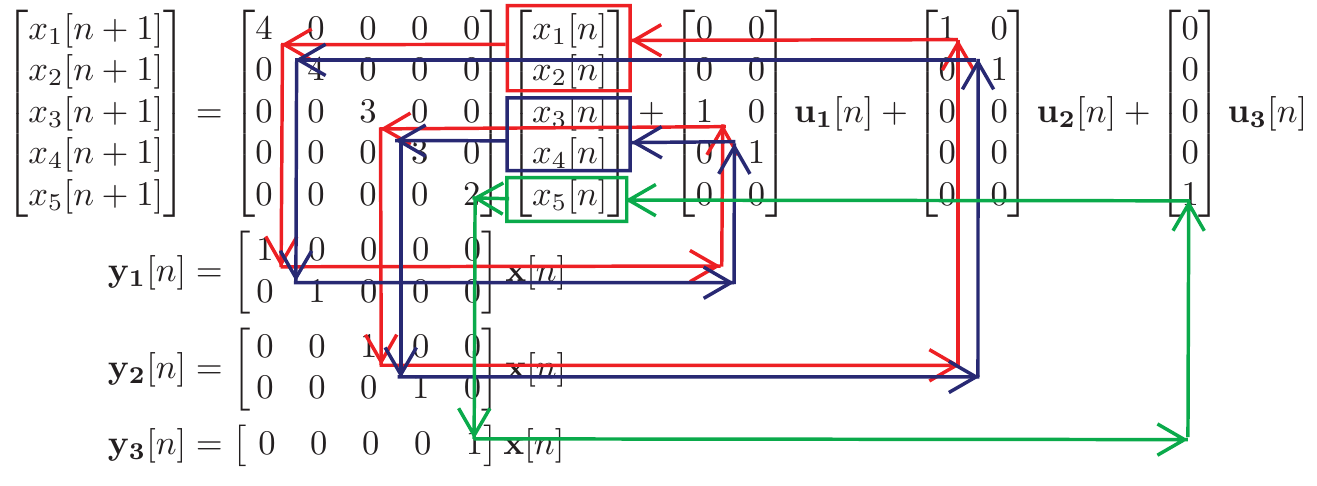}
\caption{An example of an implicit information flow in a decentralized linear system.}
\label{fig:flowexample}
\end{center}
\end{figure}

\begin{figure}
\begin{center}
\includegraphics[width = 2in]{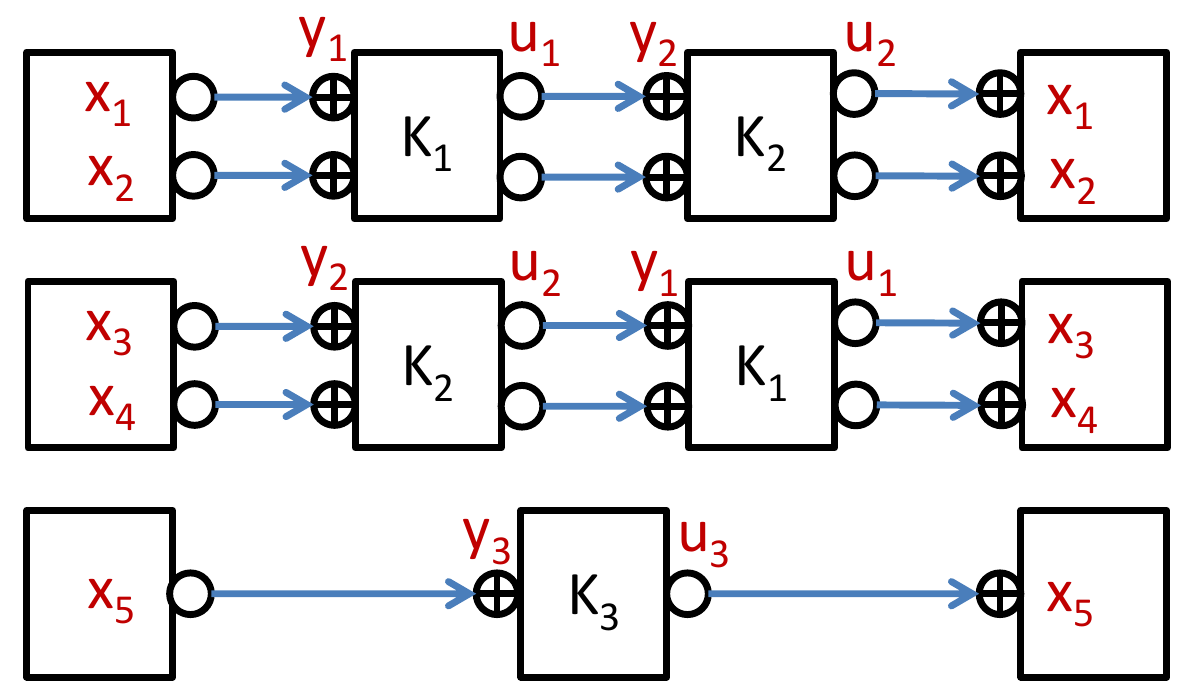}
\caption{Conceptual representations of the information flows within the example of Fig.~\ref{fig:flowexample}}
\label{fig:flowconcept}
\end{center}
\end{figure}
\label{sec:example}
Before we discuss a general algorithm to externalize the implicit
communication between controllers, it will be helpful to see the
information flows that we want to capture in an illustrative
example. By now, we have mounting evidence\footnote{We return to this
  point in the conclusion, but the evidence here has largely come from
  contexts in which the communication is explicitly present. On one
  side, papers like \cite{Kailath_Coding,Elia_Bode, Sahai_Anytime}
  construct feedback communication systems that use unstable states to
  encode desired messages. This provides strong evidence for the
  states acting as information sources. On the other side, papers like
  \cite{Baillieul_Feedback,Tatikonda_Control,Sahai_Anytime} talk about
  networked control systems in which the communication demands on the
  network come from the states. These argue persuasively for the
  states in a control system as being destinations of information
  flows since control and estimation are intimately linked
  together. The perspective on the Kalman filter presented in
  \cite{MitterNewtonKalman} suggests strongly that such information
  flows exist even when there is no explicit communication going on.}
that in linear systems, the unstable states themselves are the sources and, at
the same time, the destinations of information flows. Consider a
linear plant controlled by one controller. The states of the system
will be excited by the disturbance, {\em i.e.}~the states are
generating uncertainties. Then, the states will be observed by the
controller, {\em i.e.}~the uncertain information flows from the state
to the controller. Finally, the controller will compensate for the
disturbance, {\em i.e.}~the information flows back to the states.

When there is more than one controller, the situation becomes more
complicated since the controllers can implicitly communicate with each
other through the
plant~\cite{Witsenhausen_Counterexample,Pulkit_Witsen}. The example
shown in Fig.~\ref{fig:flowexample} (adapted from
\cite{Anderson_Timevarying}) illustrates this phenomenon.
As we can see, the states $x_1[n]$ and $x_2[n]$ are associated with
the eigenvalue $4$. However, the controller $\mathcal{K}_1$ can only
observe $x_1[n], x_2[n]$, the controller $\mathcal{K}_2$ can only
control $x_1[n], x_2[n]$, and the controller $\mathcal{K}_3$ can
neither observe nor control $x_1[n], x_2[n]$. Therefore, to stabilize
$x_1[n], x_2[n]$ the controller $\mathcal{K}_1$ intuitively has to
relay its observations to controller $\mathcal{K}_2$ through the
implicit channel provided by the states $x_3[n], x_4[n]$.

The red arrow of Fig.~\ref{fig:flowexample} shows the information flow
to stabilize $x_1[n], x_2[n]$. First, $x_1[n], x_2[n]$ is observed by
$\mathcal{K}_1$ through $\mathbf{y_1}[n]$. Then, $\mathcal{K}_1$
relays its observations to $\mathcal{K}_2$ by $\mathbf{u_1}[n]$
through the channel $x_3[n], x_4[n]$. $\mathcal{K}_2$ receives the
relayed signals through $\mathbf{y_2}[n]$, and finally controls the
states by $\mathbf{u_2}[n]$. Thus, we expect that the implicit
information flow to stabilize $x_1[n], x_2[n]$ should be roughly
representable as the first LTI network of
Fig.~\ref{fig:flowconcept}. We can see the same kind of information
flow to stabilize the states $x_3[n], x_4[n]$ as indicated by the blue
arrow. Meanwhile the state $x_5[n]$ can be stabilized by the
controller $\mathcal{K}_3$ as indicated by the green
arrow. Conceptually, these information flows can be represented as the
second and third LTI networks of Fig.~\ref{fig:flowconcept}.

Here, we can notice some interesting points. First, we are dividing
the states according to their associated eigenvalues. In this example,
the states are first divided into three sets $\{x_1[n], x_2[n]\}$,
$\{x_3[n], x_4[n]\}$ and $\{x_5[n] \}$, and the information flows for
these sets are considered separately. Moreover, in each information
flow the states associated with the same eigenvalue are considered as
both sources and destinations of the information. The
remaining states are considered as the channels that are available to
implicitly carry this information flow. The controllers themselves are
considered as relays. So in the standard LTI model of
Fig.~\ref{fig:standard}, the blocks ``$tx$'' and ``$rx$'' correspond
to the set of states in consideration and the remaining states are
included in the channel matrices, $A,B_i,\cdots,S'$. The ``$K_i$''
blocks correspond to the controllers.

We can also see the connection between stabilizability and
capacity. The eigenvalue $4$ has two associated states, $x_1[n]$ and
$x_2[n]$. Thus, we can think that this source has $2$ d.o.f.~to
transmit. This information can be successfully transferred since the
channel provided by the states $x_3[n]$ and $x_4[n]$ has
d.o.f.~capacity $2$, and so the eigenvalue $4$ is not a fixed
mode. However, if we remove the state $x_4[n]$ from the system, the
implicit channel's d.o.f.~capacity becomes $1$. Thus, a source with $2$
d.o.f.~cannot be transferred, and the eigenvalue $4$ becomes a fixed
mode.

Table~\ref{tbl:comparison} summarizes the relationship between decentralized control and relay communication problems which we have discussed so far and will make rigorous in following sections.
\begin{table}
\begin{tabular}{|l|l|}
\hline
 LTI Communication Networks & Decentralized Linear Systems \\
\hline\hline
Source & Unstable States associated with eigenvalue $\lambda$ \\
\hline
Destination  & Unstable States associated with eigenvalue $\lambda$ \\
\hline
Relays & Controllers  \\
\hline
Channels & Remaining States and $B_i$, $C_i$\\
\hline
Message & Unstable Subspace associated with eigenvalue $\lambda$ \\
\hline
Rate of Message & Number of Jordan blocks associated with eigenvalue $\lambda$\\
\hline
Capacity & Stabilizability (Enough implicit communication for unstable subspace) \\
\hline
\end{tabular}
\caption{The comparison between decentralized linear systems and LTI communication networks}
\label{tbl:comparison}
\end{table}

\section{Externalization of Implicit Communication}
\label{sec:external}
In this section, we discuss how to externalize the implicit
communication in decentralized linear systems. The main idea can be
considered as the reverse of the algebraic approach to network
coding. In \cite{Koetter_Algebraic}, Koetter and Medard considered
network coding as an algebraic problem. In other words, they found
that what is important about networks (graphical objects) in network
coding is their transfer functions (algebraic objects). What we do is
the opposite. First, we will find transfer functions which are closely
connected to the implicit information flows needed to stabilize linear systems. Then, we will find the LTI networks whose
transfer functions these are.

\subsection{Canonical-Form Externalization}
It turns out that what is important in externalization is the right
choice of transfer function. In section~\ref{sec:example} we saw that
the source and the destination of the information flows are the
states. Thus, the straightforward choice is the transfer function from
the states $x[n]$ to themselves. For that purpose, we introduce an
auxiliary input $u[n]$ and auxiliary output $y[n]$ to the closed loop
system in the following way.
\begin{align}
&x[n+1]=(A+B_1 K_1 C_1 +\cdots+ B_v K_v C_v ) x[n]+ u[n],  \nonumber \\
&y[n]= x[n]. \nonumber
\end{align}
It is clear that all the states ${x}[n]$ are directly controllable by
${u}[n]$ and observable by ${y}[n]$. Since the fixed modes show up as
poles in the transfer function, checking whether $\lambda$ is a fixed
mode involves checking whether the transfer function from $u[n]$ to
$y[n]$ has a fixed pole. However, checking poles is mathematically
troublesome since it results in division by zero. Thus, instead we
inspect the zeros of the formal transfer function from $y[n]$ to
$u[n]$.

Under the assumption that ${x}[0]={0}$, the formal transfer function
from ${y}[n]$ to ${u}[n]$ is given as
\begin{align}
{u}(z)=\underbrace{(z {I} - {A} - {B_1} {K_1} {C_1} - \cdots - {B_v} {K_v} {C_v})}_{:=G_{cn}(z,K)}{y}(z). \nonumber
\end{align}
Here, $G_{cn}(z,K)$ is a rational function whose dummy variables are not
only $z$ but also the elements of the $K_i$s.

By Lemma~\ref{lem:LTI:trans}, the standard network,
$\mathcal{N}_s(zI-A;-B_i,0;C_i,0;0,0;0,0)$, has $G_{cn}(z,K)$ as its
transfer function. Denote this standard network as
$\mathcal{N}_{cn}(z)$. The graphical representation of
$\mathcal{N}_{cn}(z)$ at the generalized frequency $z=\lambda$ is
shown in Fig.~\ref{fig:canonexternal}.

Then, we can easily derive the following theorem connecting the d.o.f.~capacity of the LTI network $\mathcal{N}_{cn}(z)$ with the stabilizability of the decentralized linear system $\mathcal{L}(A,B_i,C_i)$.
\begin{figure}
\includegraphics[width = 3in]{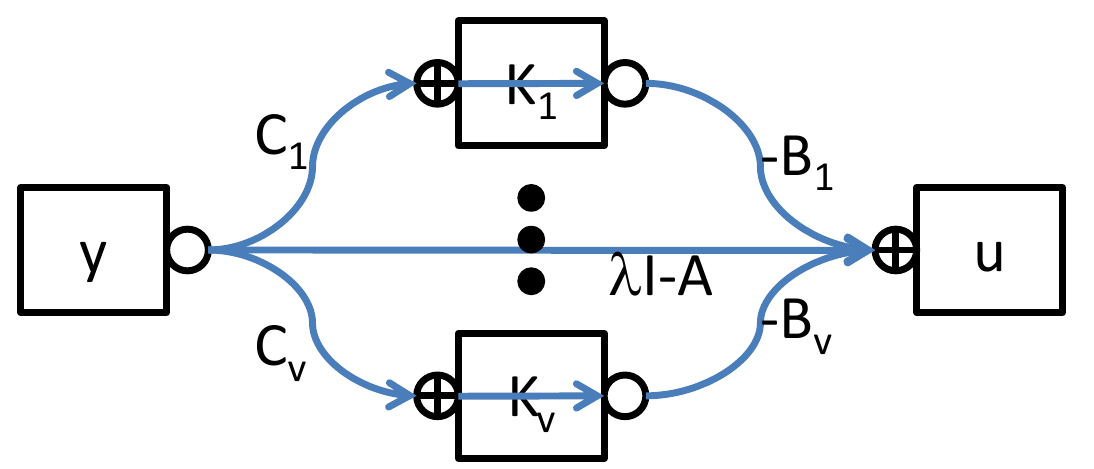}
\caption{The graphical representation of $\mathcal{N}_{cn}(\lambda)$}
\label{fig:canonexternal}
\end{figure}
\begin{theorem}[Capacity-Stabilizability Equivalence]
Given the above definitions, the following statements are equivalent.\\
(1) $\lambda$ is a fixed mode of the decentralized linear system $\mathcal{L}({A},{B_i},{C_i})$. \\
(2) $\rank\left({G_{cn}}\left(\lambda,K \right)\right) < \dim({A})$. \\
(3) (transfer matrix rank of LTI network $\mathcal{N}_{cn}(\lambda)$) $< \dim({A})$.\\
(4) (mincut rank of the LTI network $\mathcal{N}_{cn}(\lambda)$) $< \dim({A})$. \\
(5) $\min_{V \subseteq \{1,\cdots,v \}} \rank \begin{bmatrix} \lambda I-A & -B_V \\ C_{V^c} & 0 \end{bmatrix} < \dim(A)$.
\label{thm:equivalence1}
\end{theorem}
\begin{proof}
By the definition of fixed modes, (1) is equivalent to $\det( A + \sum_{1 \leq i \leq v} B_i K_i C_i - \lambda I)=0$ for all $K_i \in \mathbb{C}^{q_i \times r_i}$. By Lemma~\ref{lem:keylemma}, this is equivalent to $\det( A + \sum_{1 \leq i \leq v} B_i K_i C_i - \lambda I) = 0$ where each element of $K_i$ is considered as distinct dummy variables. Since $\det( A + \sum_{1 \leq i \leq v} B_i K_i C_i - \lambda I)=0$ means not full rank, this is again equivalent to $\rank( \lambda I - A - \sum_{1 \leq i \leq v} B_i K_i C_i ) < \dim(A)$, which is the statement (2). (2) and (3) are equivalent by the definitions of $G_{cn}(z,K)$ and $\mathcal{N}_{cn}(z)$. (3) and (4) are equivalent by the algebraic
mincut-maxflow theorem, Corollary~\ref{cor:mincut}. The equivalence of (4) and (5) follows from the definitions of the cutset matrices of $\mathcal{N}_{cn}(z)$.
\end{proof}
Remark 1: $y(z)$ is the signal assigned to the transmitter of
$\mathcal{N}_{cn}(z)$, and $u(z)$ is the signal assigned to the
receiver of $\mathcal{N}_{cn}(z)$. Thus, the LTI network connects the
states $x[n]$ to themselves, which complies with our discussion of
section~\ref{sec:example}.\\
Remark 2: The statement (1) of the theorem is directly connected to
stabilizability by Theorem~\ref{thm:stability}, and the statement
(3) of the theorem is about the d.o.f.~capacity of the network at the
frequency $z=\lambda$. Thus, this theorem reveals a fundamental
equivalence between stabilizability and capacity.\\
Remark 3: This externalization seems naive, and as we can see in
Fig.~\ref{fig:canonexternal} it gives only networks with a simple
topology that does not have any links between the relays. We call this
externalization as the canonical-form externalization because of its
simple topology. In the next section, we show another way of
externalizing the implicit communication which a different network
topology. The fact that different externalizations are possible is what allowed
to us discover that, in fact, any arbitrary network can be converted to
the canonical network of Fig.~\ref{fig:canonexternal}, which is the insight for network linearization as discussed in Section~\ref{sec:networklin}.\\
Remark 4: In fact, statement (5) is the algebraic characterization of
fixed modes shown in \cite{Anderson_Algebraic}. So in hindsight,
we can say that Anderson and Clements found the algebraic
mincut-maxflow theorem for the special network of
Fig.~\ref{fig:canonexternal}.\\
Remark 5: It is known that the rank of the channel matrix for a cut is
a submodular function~\cite{Savari_maxflow}. The complexity of
submodular function minimization is polynomial
time~\cite{Schrijver_Combinatorial}. Therefore, we can efficiently
check for fixed modes.

\begin{figure}
\begin{center}
\includegraphics[width = 2in]{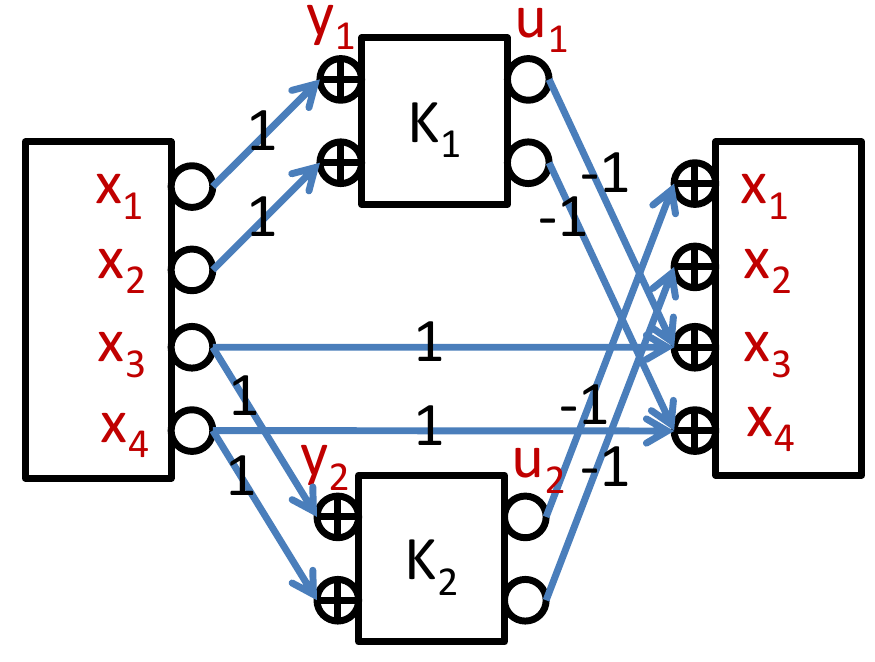}
\end{center}
\caption{Canonical-form externalization of the system of Fig.~\ref{fig:flowexample} for $\lambda=4$}
\label{fig:externalexcn}
\end{figure}

Now, we can try to externalize the implicit communication of the
example shown in
Fig.~\ref{fig:flowexample}. Fig.~\ref{fig:externalexcn} shows the
canonical-form externalization for eigenvalue $4$. If we look at the
figure, this externalization is not what we expected in
Fig.~\ref{fig:flowconcept}. Since the links between the relays are
missing, we cannot see any relaying behavior between two
controllers. Also, we cannot clearly see the fact that there are 2
degrees-of-freedom that must be communicated. This motivates us to
seek a more compact externalization where the eigenvalues are emphasized by using Jordan forms.

\subsection{Jordan-Form Externalization}
\label{sec:jordanex}
As we see in the above section, externalization is done for each
eigenvalue of $A$. For a general matrix $A$, there is no
clear correspondence between eigenvalues and particular states in the linear
system. Thus, we cannot but choose the transfer function from all the
states $x[n]$ to themselves. However, if ${A}$ is given in
Jordan normal form~\cite{Chen}, we can find a natural correspondence
between eigenvalues and states, and use this to reduce the dimension
of the transfer function. Moreover, by a similarity transform an
arbitrary linear system $\mathcal{L}(A,B_i,C_i)$ can be converted to
an equivalent linear system $\mathcal{L}(\widetilde{A},\widetilde{B_i},\widetilde{C_i})$ with the
matrix $A'$ in Jordan form~\cite{Chen}. Thus, without loss
of generality, assume that ${A}$ is in Jordan form. (This corresponds to examining the system in its natural coordinate system.)

For a Jordan-form ${A}$ matrix, there is no (internal) interaction between states
belonging to different Jordan blocks. Thus, as discussed in
section~\ref{sec:example}, to check if $\lambda$ is a fixed mode, it
is enough to examine the transfer matrix from the states associated
with Jordan blocks corresponding to the eigenvalue $\lambda$ to
themselves. For externalization, we can simply repeat the steps of
the above section.

To understand the core ideas, we first consider a diagonal $A$ matrix,
{\em i.e.} $A=\begin{bmatrix} \lambda I_{m_\lambda} & 0 \\
  0 & A' \end{bmatrix}$ where $A'$ is a diagonal matrix whose diagonal elements are not equal to $\lambda$. Because the matrix is diagonal, each Jordan block is just a $1 \times 1$ matrix and so $m_\lambda$ can
be thought of as the number of Jordan blocks associated with
$\lambda$. We will introduce auxiliary inputs and outputs that control
and observe the states corresponding to the eigenvalue $\lambda$. For
this, we define $B_\lambda$ and $C_\lambda$ as follows:
\begin{align}
C_{\lambda}=\begin{bmatrix} I_{m_\lambda} & 0 \end{bmatrix}, B_{\lambda}=\begin{bmatrix} I_{m_\lambda} & 0 \end{bmatrix}^T.
\label{eqn:jordan:def1}
\end{align}
Then, the closed loop system is given as
\begin{align}
&x[n+1]=(A+ \sum_{1 \leq i \leq v} B_i K_i C_i)x[n]+B_\lambda u_\lambda[n] \nonumber \\
&y_\lambda[n]= C_\lambda x[n]\nonumber
\end{align}
where $u_{\lambda}[n]$ and $y_{\lambda}[n]$ are $m_\lambda \times 1$ vectors. Let's set
\begin{align}
&\left( z{I}-{A} \right)=
\begin{bmatrix}
{A_{\lambda,1,1}}(z) & {A_{\lambda,1,2}}(z) \\
{A_{\lambda,2,1}}(z) & {A_{\lambda,2,2}}(z) \\
\end{bmatrix} \label{eqn:jordan:def2} \\
&{C_i} =
\begin{bmatrix}
{C_{i,\lambda,1}} & {C_{i,\lambda,2}}
\end{bmatrix}, 
 {B_i}=
\begin{bmatrix}
{B_{i,\lambda,1}} \\
{B_{i,\lambda,2}} \\
\end{bmatrix} \nonumber
\end{align}
where ${A_{\lambda,1,1}}(z)$ is a $m_\lambda \times m_\lambda$ matrix,
$B_{i,\lambda,1}$ is a $m_\lambda \times q_i$ matrix,
$C_{i,\lambda,1}$ is a $r_i \times m_\lambda$ matrix, and the others
are the proper implied dimensions. Here, by construction, we can see
$A_{\lambda,1,1}(\lambda)=0$, $A_{\lambda,1,2}(\lambda)=0$,
$A_{\lambda,2,1}(\lambda)=0$, and $A_{\lambda,2,2}(\lambda)$ is
invertible.

Then, we can see that the transfer function from $u_\lambda(z)$ to
$y_\lambda(z)$ is given as follows:
\begin{align}
&y_{\lambda}(z)= \begin{bmatrix} I & 0 \end{bmatrix}
\bigg(
\begin{bmatrix}
A_{\lambda,1,1}(z) & A_{\lambda,1,2}(z) \\
A_{\lambda,2,1}(z) & A_{\lambda,2,2}(z) \\
\end{bmatrix} \nonumber \\
&-
\sum_{1 \leq i \leq v}
\begin{bmatrix}
B_{i,\lambda,1} K_i C_{i,\lambda,1} & B_{i,\lambda,1} K_i C_{i,\lambda,2} \\
B_{i,\lambda,2} K_i C_{i,\lambda,1} & B_{i,\lambda,2} K_i C_{i,\lambda,2} \\
\end{bmatrix}
\bigg)^{-1}
\begin{bmatrix} I \\ 0 \end{bmatrix} u_\lambda(z) \label{eqn:jordan:1}
\end{align}
We need the following lemma to obtain the transfer function from $y_\lambda(z)$ to $u_\lambda(z)$.
\begin{lemma}
For a field $\mathbb{F}$ and $n_1, n_2 \in \mathbb{Z}^+$, let $y \in \mathbb{F}^{n_1 \times 1}$, $u \in \mathbb{F}^{n_1 \times 1}$, $A \in \mathbb{F}^{n_1 \times n_1}$, $B \in \mathbb{F}^{n_1 \times n_2}$, $C \in \mathbb{F}^{n_2 \times n_1}$, and $D \in \mathbb{F}^{n_2 \times n_2}$. Assume $D$ is invertible. Then, $\begin{bmatrix} A & B \\ C & D \end{bmatrix}$ is invertible iff $(A-BD^{-1}C)$ is invertible.\\
Moreover, if we assume $D$ and $\begin{bmatrix} A & B \\ C & D \end{bmatrix}$ are invertible,
\begin{align}
y=\begin{bmatrix} I_{n_1} & 0 \end{bmatrix} \begin{bmatrix} A & B \\ C & D \end{bmatrix}^{-1} \begin{bmatrix} I_{n_1} \\ 0 \end{bmatrix} u \label{eqn:pattern1}
\end{align}
implies
\begin{align}
u=(A-B D^{-1} C)y. \nonumber
\end{align}
\label{lem:ext:matrix}
\end{lemma}
\begin{proof}
By Lemma~\ref{lem:LIN:rank},
\begin{align}
\rank \begin{bmatrix} A & B \\ C & D \end{bmatrix} = n_2  + \rank ( A-BD^{-1}C ).
\end{align}
Therefore, the first statement of the lemma is true. For the second,
\begin{align}
y&=
\begin{bmatrix} I_{n_1} & 0 \end{bmatrix} \begin{bmatrix} A & B \\ C & D \end{bmatrix}^{-1} \begin{bmatrix} I_{n_1} \\ 0 \end{bmatrix} u \\
&=\begin{bmatrix}
I_{n_1} & 0
\end{bmatrix}\left(
\begin{bmatrix}
I_{n_1} & BD^{-1} \\
0 & I_{n_2}
\end{bmatrix}
\begin{bmatrix}
A-B D^{-1} C & 0 \\
C & D
\end{bmatrix}
\right)^{-1}
\begin{bmatrix}
I_{n_1} \\
0
\end{bmatrix} u \\
&=\begin{bmatrix}
I_{n_1} & 0
\end{bmatrix}
\begin{bmatrix}
(A-B D^{-1} C)^{-1} & 0 \\
-D^{-1} C (A-B D^{-1} C)^{-1} & D^{-1}
\end{bmatrix}
\begin{bmatrix}
I_{n_1} & -B D^{-1} \\
0 & I_{n_2}
\end{bmatrix}
\begin{bmatrix}
I_{n_1} \\
0
\end{bmatrix} u \nonumber \\
&=( A-B D^{-1}C )^{-1} u \nonumber
\end{align}
Here, the matrix inverses exist because of the assumption that $D$ is invertible, and the first statement of the lemma. Therefore, $u=(A-BD^{-1}C)y$.
\end{proof}

By Lemma~\ref{lem:ext:matrix} and matching \eqref{eqn:jordan:1} to the pattern given by \eqref{eqn:pattern1}, the transfer function from $y_\lambda(z)$ to $u_\lambda(z)$, $G_{jd,\lambda}(z,K)$, is given as
\begin{align}
&G_{jd,\lambda}(z,K)
=
(A_{\lambda,1,1}(z)- \sum_{1 \leq i \leq v} B_{i,\lambda,1}K_i C_{i,\lambda,1} ) \nonumber \\
&\quad+( A_{\lambda,1,2}(z)-\sum_{1 \leq i \leq v} B_{i,\lambda,1} K_i C_{i,\lambda,2} )\nonumber \\
&\quad\cdot( I - (I - A_{\lambda,2,2}(z) + \sum_{1 \leq i \leq v} B_{i,\lambda,2} K_i C_{i,\lambda,2}) )^{-1} \nonumber\\
&\quad\cdot( -A_{\lambda,2,1}(z) + \sum_{1 \leq i \leq v} B_{i,\lambda,2} K_i C_{i,\lambda,1}). \label{eqn:jordan:3}
\end{align}
By Lemma~\ref{lem:LTI:trans}, $G_{jd,\lambda}(z,K)$ corresponds to the transfer matrix of the standard LTI network,
\begin{align}
&\mathcal{N}_s(A_{\lambda,1,1}(z);-B_{i,\lambda,1},B_{i,\lambda,2};C_{i,\lambda,1},C_{i,\lambda,2}\nonumber\\
&;A_{\lambda,1,2}(z),-A_{\lambda,2,1}(z);I,I-A_{\lambda,2,2}(z)). \label{eqn:jordan:4}
\end{align}
Call this network $\mathcal{N}_{jd,\lambda}(z)$. When it is evaluated
at the generalized frequency $z=\lambda$,
$\mathcal{N}_{jd,\lambda}(z)$ can be simplified further as
$\mathcal{N}_s(0;-B_{i,\lambda,1},B_{i,\lambda,2};C_{i,\lambda,1},C_{i,\lambda,2};0,0;I,I-A_{\lambda,2,2}(\lambda))$.
\begin{figure}
\includegraphics[width = 3in]{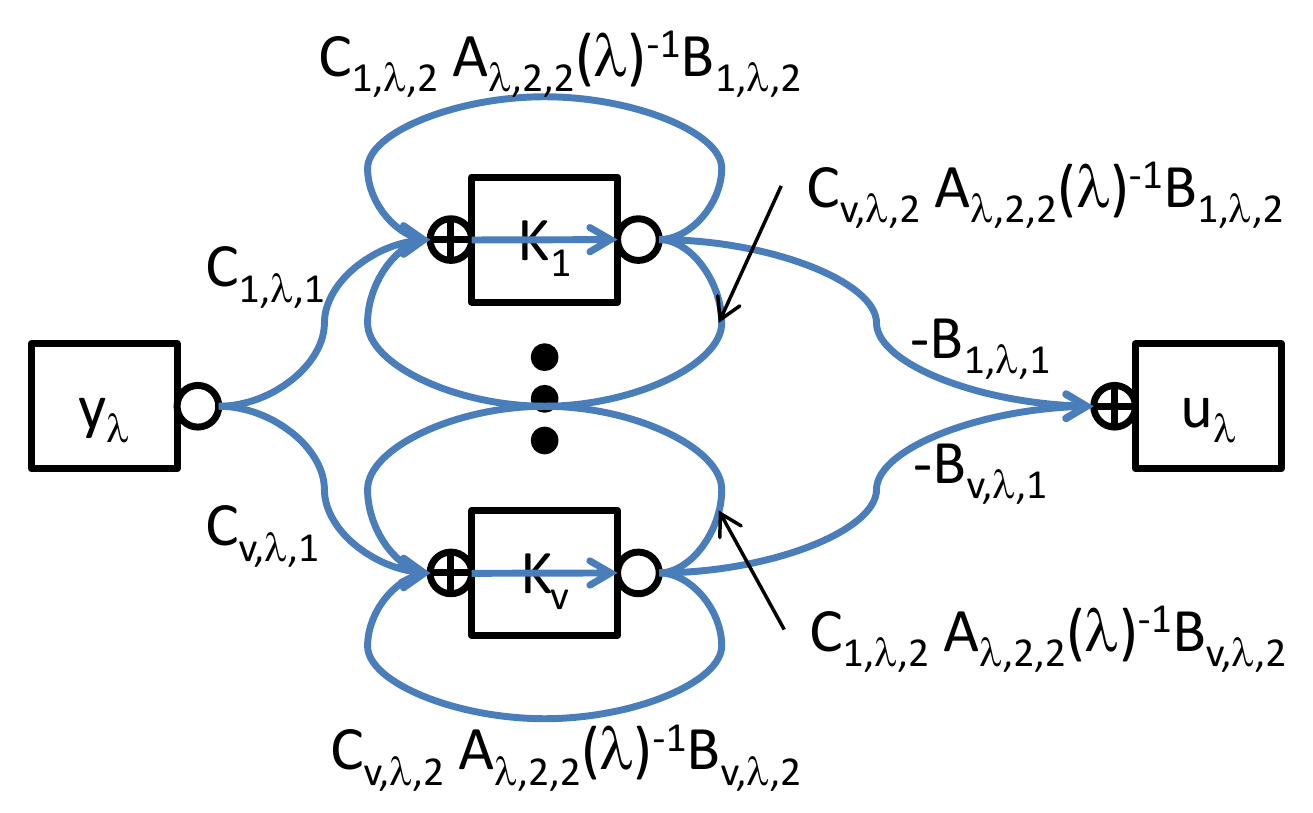}
\caption{The graphical representation of $\mathcal{N}_{jd,\lambda}(\lambda)$}
\label{fig:jordanexternal}
\end{figure}
Fig.~\ref{fig:jordanexternal} shows this network,
$\mathcal{N}_{jd,\lambda}(\lambda)$, and by Lemma~\ref{lem:LTI:trans}
the channel matrices are given as follows:
\begin{align}
&H_{tx,rx}(\lambda)=0, \nonumber \\
&H_{tx,i}(\lambda)=C_{i,\lambda,1}, \nonumber \\
&H_{i,rx}(\lambda)=-B_{i,\lambda,1}, \nonumber \\
&H_{i,j}(\lambda)=C_{j,\lambda,2}A_{\lambda,2,2}(\lambda)^{-1}B_{i,\lambda,2}. \label{eqn:jordan:5}
\end{align}
Now, we state a parallel proposition to Theorem~\ref{thm:equivalence1}.
\begin{proposition}
Given the above definitions, the following statements are equivalent.\\
(1) $\lambda$ is a fixed mode of the decentralized linear system $\mathcal{L}(A,B_i,C_i)$\\
(2) $\rank(G_{jd,\lambda}(\lambda,K)) < m_\lambda$\\
(3) (transfer matrix rank of LTI network $\mathcal{N}_{jd,\lambda}(\lambda)$) $< m_\lambda$\\
(4) (mincut rank of the LTI network $\mathcal{N}_{jd,\lambda}(\lambda)$) $< m_\lambda$\\
(5) \small{$\min_{V \subseteq \{1,\cdots,v \}} \rank
\begin{bmatrix}
0 & -B_{V,\lambda,1} \\
C_{V^c,\lambda,1} & C_{V^c,\lambda,2} A_{\lambda,2,2}(\lambda)^{-1} B_{V,\lambda,2}
\end{bmatrix}
 < m_\lambda$}
\label{prop:equivalence2}
\end{proposition}
\begin{proof}
By Theorem~\ref{thm:equivalence1} (2) and the fact that the dimension of $G_{cn}(\lambda,K)$ is $\dim(A)$, we know that the statement (1) is equivalent to $G_{cn}(\lambda,K)$ is rank deficient. Furthermore, in Lemma~\ref{lem:ext:matrix} by considering $(A_{\lambda,1,1}(\lambda)- \sum_{1 \leq i \leq v} B_{i,\lambda,1}K_i C_{i,\lambda,1} )$ as $A$, $(A_{\lambda,1,2}(\lambda)-\sum_{1 \leq i \leq v} B_{i,\lambda,1} K_i C_{i,\lambda,2})$ as $B$, $(A_{\lambda,2,1}(\lambda) - \sum_{1 \leq i \leq v} B_{i,\lambda,2} K_i C_{i,\lambda,1})$ as $C$, and $(A_{\lambda,2,2}(\lambda) - \sum_{1 \leq i \leq v} B_{i,\lambda,2} K_i C_{i,\lambda,2})$ as $D$, we can conclude that $G_{jd,\lambda}(\lambda,K)$ is full rank if and only if
\begin{align}
&\begin{bmatrix}
A_{\lambda,1,1}(z) & A_{\lambda,1,2}(z) \\
A_{\lambda,2,1}(z) & A_{\lambda,2,2}(z) \\
\end{bmatrix} -
\sum_{1 \leq i \leq v}
\begin{bmatrix}
B_{i,\lambda,1} K_i C_{i,\lambda,1} & B_{i,\lambda,1} K_i C_{i,\lambda,2} \\
B_{i,\lambda,2} K_i C_{i,\lambda,1} & B_{i,\lambda,2} K_i C_{i,\lambda,2} \\
\end{bmatrix}\\
&= \lambda I - A - \sum_{1 \leq i \leq v} B_i K_i C_i \\
&= G_{cn} (\lambda, K)
\end{align}
is full rank. Thus, $G_{cn}(\lambda,K)$ is rank deficient if and only if $G_{jd,\lambda}(\lambda,K)$ is rank deficient. Since the dimension of $G_{jd,\lambda}(\lambda,K)$ is $m_{\lambda}$, the statement (1) is equivalent to the statement (2).

The statement (2) and (3) are equivalent, since $G_{jd, \lambda}(\lambda, K)$ is the transfer function of $\mathcal{N}_{jd,\lambda}(\lambda)$. 

The statement (3) and (4) are equivalent by the mincut-maxflow theorem of Corollary~\ref{cor:mincut}. 

The equivalence of the statement (4) and (5) comes from the definitions of the channel matrices of $\mathcal{N}_{jd,\lambda}(\lambda)$ shown in \eqref{eqn:jordan:5}.
%
%
\end{proof}

This theorem can be generalized to arbitrary Jordan forms $A$ by
introducing auxiliary inputs and outputs from the states associated with
$\lambda$ to themselves. However, we can further reduce the dimension
of the transfer matrix by inspecting the information flow inside
nontrivial Jordan blocks.
\begin{figure}
\begin{center}
\includegraphics[width = 1.5in]{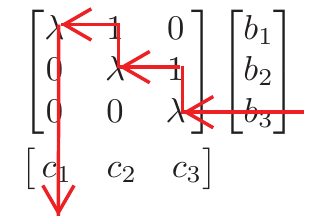}
\end{center}
\caption{Critical Information Flow and Transfer Function in Jordan block}
\label{fig:flowjordan}
\end{figure}

Let's consider the stabilizability condition for a single Jordan block $A$ matrix, 
${A}=\begin{bmatrix} \lambda & 1 & 0 \\ 0 & \lambda & 1 \\ 0 & 0 &
  \lambda \end{bmatrix}$, ${B}=\begin{bmatrix} b_1 \\ b_2 \\
  b_3 \end{bmatrix}$, ${C}= \begin{bmatrix} c_1 & c_2 &
  c_3 \end{bmatrix}$. It is well-known~\cite{Chen} that the observability
condition for this example is $c_1 \neq 0$ and the controllability
condition is $b_3\neq 0$. In other words,
as shown in Fig.~\ref{fig:flowjordan}, we can think of the critical information flow to stabilize a single Jordan block which flows from the right-bottom element  to the left-top element. To check whether a single Jordan block has a fixed mode or not, it is enough to consider the transfer function corresponding to this information flow. 

This observation for a single Jordan block can be generalized to multiple Jordan blocks. To decide whether $\lambda$
is a fixed mode or not, it is enough to examine the transfer function matrix from the right-bottom elements of the multiple Jordan blocks (corresponding to the
eigenvalue $\lambda$) to their left-top elements.

We will make this observation rigorous by introducing the following definitions. Since the definitions are notationally heavy, we recommend to see Appendix~\ref{app:jordanex} for a descriptive example. In Appendix~\ref{app:jordanex}, we consider the case when $A
=\begin{bmatrix} \lambda & 1 & 0 & 0 & 0& 0 \\ 0 & \lambda & 1 & 0 & 0& 0 \\ 0 & 0 & \lambda & 0 & 0& 0 \\ 0 & 0 & 0 & \lambda & 1 & 0 \\ 0 & 0 & 0 & 0 & \lambda & 0 \\ 0 & 0 & 0 & 0 & 0& \lambda' \end{bmatrix}$. Then, we can see that the 3rd and 5th rows and the 1st and 4th column in $\lambda I - A$ are all zeros. To reduce the system to the system considered in Proposition~\ref{prop:equivalence2}, we move these all zero columns and rows to left top side of the matrix by multiplying permutation matrices to $\lambda I - A$. To this end, we will define the permutation matrices $P_{L,\lambda}$, $P_{R,\lambda}$. 

Let $a_{i,j}$ be the $(i,j)$ element of ${A}\in \mathbb{C}^{m \times m}$. 
Since the locations of all zero columns and rows are related to the locations of Jordan blocks, we have to define the indexes which indicates the location of each Jordan block. The sequences $\kappa_{L,\lambda}$ and $\kappa_{R,\lambda}$ count the number of Jordan blocks associated with $\lambda$. The difference between two sequences is that $\kappa_{L,\lambda}$ increases at the right-bottom element of the Jordan block, while $\kappa_{R,\lambda}$ increases at the left-top.
\begin{align}
&\kappa_{L,\lambda}(0)=0 \\
&\mbox{For } 1 \leq i < m, \\
&\kappa_{L,\lambda}(i)=
\left\{
\begin{array}{ll}
\kappa_{L,\lambda}(i-1)+1 & \mbox{if } a_{i,i}=\lambda \mbox{ and } a_{i,i+1}=0 \\
\kappa_{L,\lambda}(i-1) & \mbox{otherwise} \\
\end{array}
\right. \\
&\kappa_{L,\lambda}(m)=
\left\{
\begin{array}{ll}
\kappa_{L,\lambda}(m-1)+1 & \mbox{if } a_{m,m}=\lambda  \\
\kappa_{L,\lambda}(m-1) & \mbox{otherwise} \\
\end{array}
\right.
\end{align}
\begin{align}
&\kappa_{R,\lambda}(0)=0 \\
&\kappa_{R,\lambda}(1)=
\left\{
\begin{array}{ll}
\kappa_{R,\lambda}(0)+1 & \mbox{if } a_{1,1}=\lambda \\
\kappa_{R,\lambda}(0) & \mbox{otherwise}
\end{array}
\right. \\
&\mbox{For } 1 < i \leq m, \\
&
\kappa_{R,\lambda}(i)=
\left\{
\begin{array}{ll}
\kappa_{R,\lambda}(i-1)+1 & \mbox{if } a_{i,i}=\lambda \mbox{ and } a_{i-1,i}=0 \\
\kappa_{R,\lambda}(i-1) & \mbox{otherwise}
\end{array}
\right.
\end{align}
Notice that these two sequences are just different ways of counting the number of Jordan blocks associated with the eigenvalue $\lambda$. If we denote by $m_{\lambda}$ the number of Jordan blocks associated with the eigenvalue $\lambda$, then $m_\lambda=\kappa_{L,\lambda}(m)=\kappa_{L,\lambda}(m)$. From the sequences $\kappa_{R,\lambda}$ and $\kappa_{L,\lambda}$, we also define $\iota_{R,\lambda}$ that indicates the left-top elements of the Jordan block associated with $\lambda$ and $\iota_{L,\lambda}$ that indicates the right-bottom elements.
\begin{align}
&\iota_{L,\lambda}(0)=0 \\
&\mbox{For }1\leq  i \leq m_{\lambda}\\
&\iota_{L,\lambda}(i)=\min\{ k \in \mathbb{N} :k > \iota_{L,\lambda}(i-1), \kappa_{L,\lambda}(k) > \kappa_{L,\lambda}(k-1) \}
\end{align}
Likewise,
\begin{align}
&\iota_{R,\lambda}(0)=0 \\
&\mbox{For }1\leq  i \leq m_{\lambda}\\
&\iota_{R,\lambda}(i)=\min\{ k \in \mathbb{N} :k > \iota_{R,\lambda}(i-1), \kappa_{R,\lambda}(k) > \kappa_{R,\lambda}(k-1) \}
\end{align}
We also define permutation maps and matrices for $ \lambda I-A$. The role of these permutation maps and matrices is to collect all zero rows and columns in $\lambda I - A$. 
The permutation maps $\pi_{L,\lambda}(i)$ and $\pi_{R,\lambda}(i)$ that map the set $\{1,\cdots,m \}$ to itself are defined as follows:
\begin{align}
&\pi_{L,\lambda}(i)=
\left\{
\begin{array}{ll}
\kappa_{L,\lambda}(i) & \mbox{if } \kappa_{L,\lambda}(i) > \kappa_{L,\lambda}(i-1) \\
i+\kappa_{L,\lambda}(m)-\kappa_{L,\lambda}(i) & \mbox{otherwise} 
\end{array}
\right. \\
&\pi_{R,\lambda}(i)=
\left\{
\begin{array}{ll}
\kappa_{R,\lambda}(i) & \mbox{if } \kappa_{R,\lambda}(i) > \kappa_{R,\lambda}(i-1) \\
i+\kappa_{R,\lambda}(m)-\kappa_{R,\lambda}(i) & \mbox{otherwise}
\end{array}
\right.
\end{align}
From the permutation map, we define the permutation matrices.
\begin{align}
P_{L,\lambda}=\begin{bmatrix}
e_{\pi_{L,\lambda}(1)} \\
\vdots \\
e_{\pi_{L,\lambda}(m)} \\
\end{bmatrix},
P_{R,\lambda}=\begin{bmatrix}
e_{\pi_{R,\lambda}(1)} \\
\vdots \\
e_{\pi_{R,\lambda}(m)} \\
\end{bmatrix}
\label{eqn:permutation}
\end{align}
where $e_i$ is the row vector with $1$ in $i$th position and $0$ in every other position.

Let's multiply these permutation matrices to $z I - A$.
\begin{align}
&{P_{L,\lambda}}^T \left( z{I}-{A} \right) {P_{R,\lambda}}=
\begin{bmatrix}
{A_{\lambda,1,1}}(z) & {A_{\lambda,1,2}}(z) \\
{A_{\lambda,2,1}}(z) & {A_{\lambda,2,2}}(z) \\
\end{bmatrix} \label{eqn:jordan:property2} 
\end{align}
where ${A_{\lambda,1,1}}(z)$ is a $m_\lambda \times m_\lambda$ matrix, ${A_{\lambda,1,2}}(z)$ is a $m_\lambda \times (m-m_\lambda)$ matrix, ${A_{\lambda,2,1}}(z)$ is a $(m-m_\lambda) \times m_\lambda$ matrix, ${A_{\lambda,2,2}}(z)$ is a $(m-m_\lambda) \times (m-m_\lambda)$ matrix.

Since the permutation matrices $P_{L,\lambda}$, $P_{R,\lambda}$ moves all zero columns and rows in $\lambda I - A$ to the left-top side of the matrix (see Appendix~\ref{app:jordanex} for an example), we can see $A_{\lambda,1,1}(\lambda)=0$, $A_{\lambda,1,2}(\lambda)=0$, $A_{\lambda,2,1}(\lambda)=0$, and $A_{\lambda,2,2}(\lambda)$ is invertible.

We also multiply the permutation matrices to $B_i$ and $C_i$, and define the following sub-matrices after this permutation.
\begin{align}
&{C_i} {P_{R,\lambda}}=
\begin{bmatrix}
{C_{i,\lambda,1}} & {C_{i,\lambda,2}}
\end{bmatrix}, \label{eqn:jordan:2}
{P_{L,\lambda}}^T  {B_i}=
\begin{bmatrix}
{B_{i,\lambda,1}} \\
{B_{i,\lambda,2}} \\
\end{bmatrix}
\end{align}
where $B_{i,\lambda,1}$ is a $m_\lambda \times q_i$ matrix, $B_{i,\lambda,2}$ is a $(m-m_\lambda) \times q_i$ matrix, $C_{i,\lambda,1}$ is a $r_i \times m_\lambda$ matrix, $C_{i,\lambda,2}$ is a $r_i \times (m-m_\lambda)$ matrix.

Furthermore, we will also define the auxiliary control and observation matrices $B_{\lambda}$, $C_{\lambda}$ as we did in \eqref{eqn:jordan:def1}.

We will introduce an auxiliary input that can control the right-bottom elements of the Jordan blocks and an auxiliary output that can observe the left-top elements of the Jordan blocks. The following matrices $B_{\lambda}$ and $C_{\lambda}$ correspond to the input and output matrices to the system for these auxiliary input and output.
\begin{align}
C_\lambda=
\begin{bmatrix}
e_{\iota_{R,\lambda}(1)} \\
\vdots \\
e_{\iota_{R,\lambda}(m_\lambda)} \\
\end{bmatrix}, B_\lambda=
\begin{bmatrix}
e_{\iota_{L,\lambda}(1)} \\
\vdots \\
e_{\iota_{L,\lambda}(m_\lambda)} \\
\end{bmatrix}^T.
\end{align}
From the construction of the permutation matrices, we can see that when they are applied to $C_{\lambda}$ and $B_{\lambda}$, the resulting matrices have nonzero elements only on the left or top side (just as we saw in \eqref{eqn:jordan:def1}). Formally, 
\begin{align}
C_{\lambda}P_{R,\lambda}=\begin{bmatrix} I_{m_\lambda \times m_\lambda} & 0 \end{bmatrix}, 
P_{L,\lambda}^{T} B_\lambda=\begin{bmatrix} I_{m_\lambda \times m_\lambda} \\ 0 \end{bmatrix}. \label{eqn:jordan:property1}
\end{align}

Finally, we get system equations which exactly parallel with the previous diagonal systems in \eqref{eqn:jordan:def1}, \eqref{eqn:jordan:def2}.

Now, we are ready to externalize the implicit communication based on the Jordan form matrix $A$. Just as the previous diagonal systems, we introduce the auxiliary input $u_\lambda[n] \in \mathbb{C}^{m_\lambda}$ and the auxiliary output $y_\lambda[n] \in \mathbb{C}^{m_\lambda}$. However, unlike the previous section, $u_\lambda[n]$ only controls the right-bottom elements of the Jordan blocks through $B_\lambda$ and $y_\lambda[n]$ only observes the left-top elements of the Jordan blocks through $C_\lambda$.
\begin{align}
&x[n+1]=(A+B_1 K_1 C_1+\cdots+B_v K_v C_v)x[n]+B_\lambda u_\lambda[n] \\
&y_\lambda[n]= C_\lambda x[n]
\end{align}
Then, the transfer function from $u_\lambda(z)$ to $y_\lambda(z)$ is given as follows:
\begin{align}
y_{\lambda}(z)&= {C_{\lambda}}( z {I} - {A} - \sum_{1 \leq i \leq v}{B_i}{K_i}{C_i} )^{-1}{B_{\lambda}} {u_{\lambda}}(z) \\
&= C_{\lambda} \left( {P_{L,\lambda}} {P_{L,\lambda}}^{T} \left(z{I} - {A} - \sum_{1 \leq i \leq v} {B_i}{K_i}{C_i} \right) {P_{R,\lambda}} {P_{R,\lambda}}^{T} \right)^{-1} B_\lambda u_{\lambda}(z) \\
&= C_{\lambda} {P_{R,\lambda}} \left( {P_{L,\lambda}}^T \left( z{I}-{A} \right) {P_{R,\lambda}} - \sum_{1 \leq i \leq v} {P_{L,\lambda}}^T  {B_i} {K_i} {C_i} {P_{R,\lambda}}  \right)^{-1} {P_{L,\lambda}}^{T} {B_{\lambda}} {u_{\lambda}}(z) \\
&= \begin{bmatrix} I & 0 \end{bmatrix}
\left(
\begin{bmatrix}
A_{\lambda,1,1}(z) & A_{\lambda,1,2}(z) \\
A_{\lambda,2,1}(z) & A_{\lambda,2,2}(z) \\
\end{bmatrix}
-
\sum_{1 \leq i \leq v}
\begin{bmatrix}
B_{i,\lambda,1} \\
B_{i,\lambda,2} \\
\end{bmatrix}
K_i
\begin{bmatrix}
C_{i,\lambda,1} & C_{i,\lambda,2}
\end{bmatrix}
\right)^{-1}
\begin{bmatrix} I \\ 0 \end{bmatrix} u_\lambda(z) \label{eqn:jordan:jd1}
\end{align}
where the last line uses \eqref{eqn:jordan:property1}, \eqref{eqn:jordan:property2}, \eqref{eqn:jordan:2}.

Since \eqref{eqn:jordan:1} and \eqref{eqn:jordan:jd1} are the same, \eqref{eqn:jordan:3}, \eqref{eqn:jordan:4}, \eqref{eqn:jordan:5} still hold. Thus, we can state the capacity-stabilizability equivalence theorem based on the Jordan form $A$.
\begin{theorem}(Capacity-Stabilizability Equivalence 2) Given the above definitions, the following statements are equivalent.\\
(1) $\lambda$ is the fixed mode of the decentralized linear system $\mathcal{L}(A,B_i,C_i)$\\
(2) $\rank(G_{jd,\lambda}(\lambda,K)) < m_\lambda$\\
(3) (transfer matrix rank of the LTI network $\mathcal{N}_{jd,\lambda}(\lambda)$) $< m_\lambda$\\
(4) (mincut rank of the LTI network $\mathcal{N}_{jd,\lambda}(\lambda)$) $< m_\lambda$\\
(5) $\min_{V \subset \{1,\cdots,v \}} \rank
\begin{bmatrix}
0 & -B_{V,\lambda,1} \\
C_{V^c,\lambda,1} & C_{V^c,\lambda,2} A_{\lambda,2,2}(\lambda)^{-1} B_{V,\lambda,2}
\end{bmatrix}
 < m_\lambda$
\label{thm:equivalence2}
\end{theorem}
\begin{proof}
The same as Proposition~\ref{prop:equivalence2}.
\end{proof}

Remark 1: Notice that the condition (5) seems to be quite different from the statement (5) of Theorem~\ref{thm:equivalence1} that we saw before. However, by remembering that $A$ has Jordan block structure and using the following lemma, we can directly prove the equivalence between these two statements.
\begin{lemma}
For an invertible square matrix $A$,
\begin{align}
\rank \begin{bmatrix} 0 & 0 & B_0 \\ 0 & A & B_1 \\ C_0 & C_1 & D \end{bmatrix} = \rank A + \rank \begin{bmatrix} 0 & B_0 \\ C_0 & D-C_1 A^{-1} B_1 \end{bmatrix} \nonumber
\end{align}
\end{lemma}
\begin{proof}
\begin{align}
\rank\begin{bmatrix} 0 & 0 & B_0 \\ 0 & A & B_1 \\ C_0 & C_1 & D \end{bmatrix}
=\rank\begin{bmatrix} 0 & B_0 & 0 \\ C_0 & D & C_1 \\ 0 & B_1 & A \end{bmatrix}
=\rank A + \rank\begin{bmatrix} 0 & B_0 \\ C_0 & D-C_1 A^{-1}B_1 \end{bmatrix}
\end{align}
where the first equality is due to the elementary row and column operations and the second equality is due to Lemma~\ref{lem:LIN:rank}.
\end{proof}

Remark 2: This externalization is minimal in the sense that the
dimensions of the transmitter input signal and the receiver output
signals are minimal. In other words, if we introduce an auxiliary input and output whose dimensions are smaller than the ones shown in this characterization, we cannot find the equivalent condition for fixed modes. The minimality of this characterization manifests as the absence of direct link between the transmitter and the receiver in $\mathcal{N}_{jd,\lambda}(\lambda)$.

Remark 3: It has to be mentioned that this theorem for $m_\lambda=1$ is already shown in \cite{lavaei2010time}. For this case, the condition (4) of the theorem reduces whether the mincut of the network is $0$ or not. Thus, it is equivalent to check the existence of the path from the source to the destination.

\begin{figure}
\includegraphics[width = 3in]{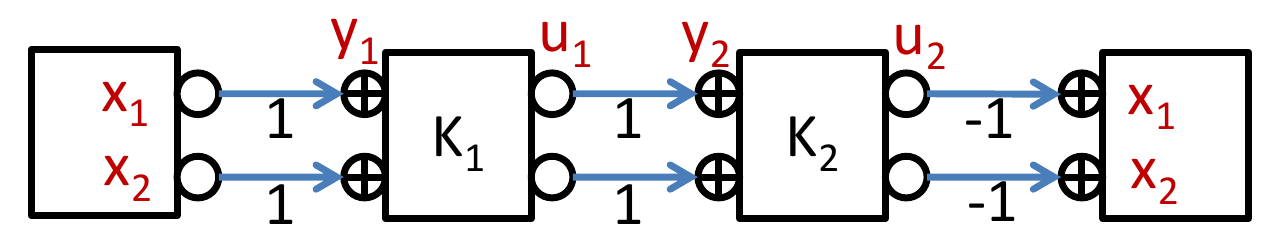}
\caption{Jordan-form externalization of the system of Fig.~\ref{fig:flowexample} for $\lambda=4$}
\label{fig:externalex}
\end{figure}

The LTI network of Fig.~\ref{fig:externalex} shows the Jordan-form
externalization of the Fig.~\ref{fig:flowexample} example for
$\lambda=4$. We can easily see that the LTI network of
Fig.~\ref{fig:externalex} agrees with the first LTI network of
Fig.~\ref{fig:flowconcept}. The information generated at
$x_1[n],x_2[n]$ is first observed by the controller $\mathcal{K}_1$,
then relayed to the controller $\mathcal{K}_2$, and finally returned
to $x_1[n],x_2[n]$. Here, the controller $\mathcal{K}_3$ is correctly
omitted since it does not affect the transfer function of the relevant
LTI network.

Until now, our discussion was limited to strictly proper systems where the impulse response from $u_i[n]$ to $y_j[n]$ is strictly causal. However, the capacity-stabilizability theorem can be easily extended to proper decentralized linear systems $\mathcal{L}(A,B_i,C_i,D_{ij})$ as shown in Appendix~\ref{app:proper}.

Before we close this section, for a sanity check we apply the result of this section to centralized systems which are already well-understood. Moreover, this will be helpful to clarify our mind in later sections.
\begin{corollary}[Stabilizability of Centralized Systems\cite{Chen}]
Let's consider the above system with a single controller, $v=1$. Then, the following conditions are equivalent.\\
(1) The centralized linear system $\mathcal{L}(A,B_1,C_1)$ is stabilizable.\\
(2) $(A,B_1)$ is controllable and $(A,C_1)$ is observable.\\
(3) $\rank(C_{1,\lambda,1}) \geq m_{\lambda}$ and\footnote{Here, the inequalities are actually equialities, $\rank(C_{1,\lambda,1}) = m_{\lambda}$ and $\rank(B_{1,\lambda,1}) = m_{\lambda}$, since the size of $C_{1,\lambda,1}$ and $B_{1,\lambda,1}$ is $m_{\lambda}$.} $\rank(B_{1,\lambda,1}) \geq m_{\lambda}$ for all unstable eigenvalues $\lambda$ of $A$.
\label{cor:observability}
\end{corollary}
\begin{proof}
This is a well-known fact in linear system theory~\cite{Chen}. Especially, the equivalence of (1) and (2) immediately follows from Theorem~\ref{thm:equivalence2}.
\end{proof}

\section{Control over LTI networks}
\label{sec:stablilizationoverLTI}
To clarify the previous discussion and reveal the further connection between network coding and decentralized linear control, we consider a stabilizability problem with an explicit communication network. Following the problem formulations in \cite{Tatikonda_Control,Sahai_Anytime,Sinopoli_Kalman,Pajic_topological}, we propose `control over LTI networks' problems.
The main advantage of these new problems is that the information for control can only flow explicitly through the communication network, while in general decentralized systems the information can also flow implicitly through the plant. Therefore, we can measure the minimum information flow to stabilize the system by simply measuring the capacity (or reliability) of the explicit communication network.

\subsection{Point-to-Point}
\label{sec:stablilizationoverLTI:ptop}
\begin{figure}
\includegraphics[width = 5in]{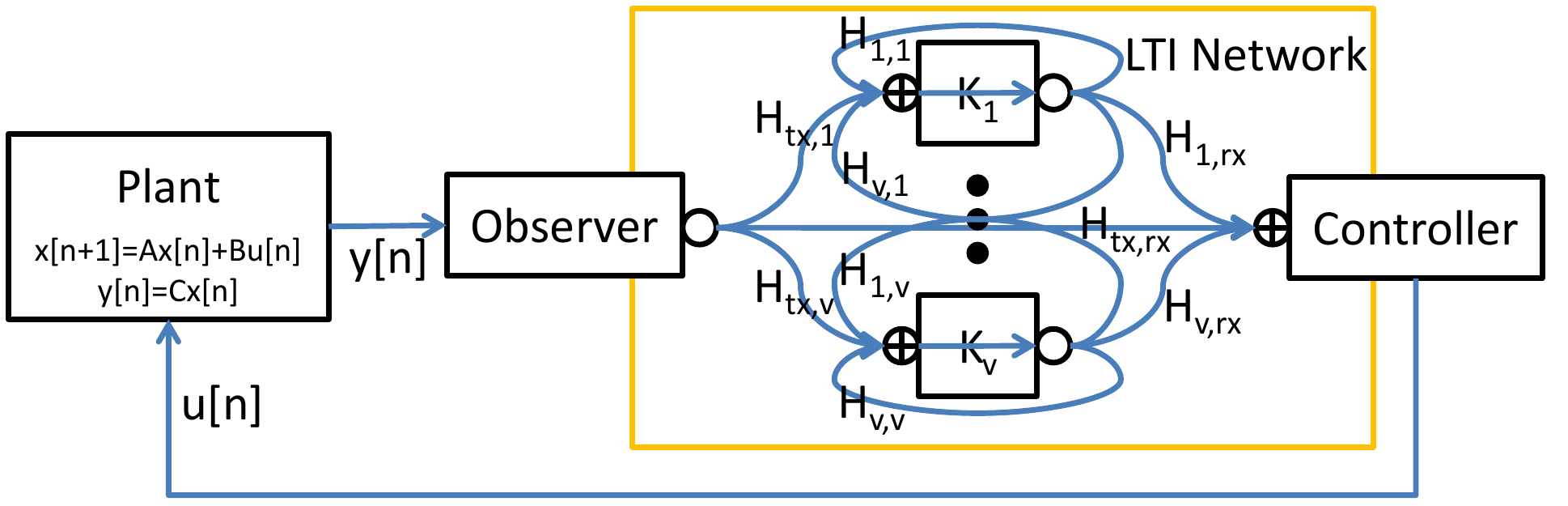}
\caption{Control over LTI Networks: Point-to-Point case}
\label{fig:ptop}
\end{figure}
The problem of control over LTI networks is shown in Fig.~\ref{fig:ptop}. The unstable plant is given as
\begin{align}
&x[n+1]=A x[n] + B u[n] + w[n] \\
&y[n]=Cx[n]
\end{align}
where $A \in \mathbb{C}^{m \times m}$, $B \in \mathbb{C}^{m \times q_{cn}}$ and $C \in \mathbb{C}^{r_{ob} \times m}$.
$x[n]$ is the state, $u[n]$ is the input to the system, $y[n]$ is the output from the system, and $w[n]$ is the disturbance.

The observer can observe the output $y[n]$, but cannot control the plant. On the other hand, the controller can control the plant through the input $u[n]$, but cannot observe the plant. Therefore, to stabilize the plant the observer has to communicate to the controller. The observer and the controller are connected by an LTI communication network, $\mathcal{N}_{ptop}(z)$, where the observer is the transmitter, the controller is the receiver, and the relays are connected by linear time-invariant channels. To make the problem physically meaningful, we assume that the channel matrices $H_{i,j}(z)$ between the relays are stable and causal. Here, we want to find the linear time-invariant observer, controller and relays that stabilize the plant. Therefore, by $z$-transform, every signal can be represented as a vector in $\mathbb{F}[z]$, and the operation of nodes (controller, observer, and relays) can be represented as a matrix in $\mathbb{F}[z]$. Denote the dimension of the input signal to the LTI network at the observer to be $q_{ob}$, and that of the output signal from the LTI network at the controller to be $r_{cn}$. Therefore, the dimensions of the observer and controller gain matrices are $q_{ob} \times r_{ob}$ and $q_{cn} \times r_{cn}$ respectively.
At the relay node $i$, denote the dimension of the input signal to the LTI network to be $q_i$ and that of the output signal from the LTI network to be $r_i$. Then, the dimension of the relay gain matrix, $K_i$, is $q_i \times r_i$. 

The goal of control and communication nodes is to stabilizing the plant.
\begin{definition}[Stabilizability over LTI networks]
Given the above definitions, we say the plant is \textbf{stabilizable over the LTI network} if there exist LTI observer, controller and relays that make $x[n]$, $y[n]$, $u[n]$, and all the inputs and outputs of the LTI network uniformly bounded for all uniformly bounded disturbances $w[n]$. For a given design, we say the plant is \textbf{stable over the LTI network} if $x[n]$, $y[n]$, $u[n]$, and all the inputs and outputs of the LTI network are uniformly bounded for all uniformly bounded disturbance $w[n]$.
\end{definition}
For a given matrix $A$, let $\sigma(A)$ be the set of eigenvalues of $A$. Let $m_{\lambda}$ be the number of Jordan blocks of $A$ associated with the eigenvalue $\lambda$. Then, the stabilizability condition is given as follows.
\begin{theorem}
The plant is stabilizable over the LTI network if and only if
for all $\lambda$ such that $\lambda \in \{ \lambda : |\lambda| \geq 1 \} \cap \sigma(A)$ the following conditions are satisfied:
\begin{align}
(i) &\begin{bmatrix} \lambda I - A \\ C \end{bmatrix} \mbox{ is full rank, i.e. $\lambda$ is observable.}\\
(ii) &\begin{bmatrix} \lambda I - A & B \end{bmatrix} \mbox{ is full rank, i.e. $\lambda$ is controllable.}\\
(iii) &m_\lambda \leq \mbox{(mincut rank of the LTI network $\mathcal{N}_{ptop}(\lambda)$)}
\end{align}
\label{thm:LTI:ptop}
\end{theorem}
\begin{proof}
For the necessity proof, we will use the realization idea. In other words, we will consider control over LTI networks as distributed linear systems and apply the concept of the fixed modes to check the stabilizability. For the sufficiency proof, we will give a constructive proof. We first design the relays in the LTI network so that it can accommodate enough information flow to stabilize the system. Then, we will design the observer and controller to connect the plant with the communication network, and stabilize it.

(1) Necessity Proof: An insightful reader may notice that `control over LTI networks' that we are considering is essentially the same as `decentralized linear systems' of Section~\ref{sec:decentrallinear}. The observer, controller, and relays in Figure~\ref{fig:ptop} can be thought as decentralized controllers. The state $x[n]$ and the internal states of the channels can be combined to one big state $x'[n]$. Then, the minimal realization procedure described in Appendix~\ref{app:realization} can convert `control over LTI networks' problems to the following decentralized linear system $\mathcal{L}_{re}(A_i',B_i',C_i',D_{ij}')$.
\begin{align}
&x'[n+1]=A'x'[n]+ \sum_{i=0}^{v+1} B_i'u_i[n] + \begin{bmatrix} I_m \\ 0 \end{bmatrix} w[n]\\
&y_i[n] = C_i' x'[n]+ \sum_{j=0}^{v+1} D_{ij}'u_j[n] \mbox{ for $0 \leq i \leq v+1$}
\end{align}
Here, the controller $0$ and $v+1$ of $\mathcal{L}_{re}(A_i',B_i',C_i',D_{ij}')$ corresponds to the observer and controller of the original problem respectively. The controllers $1$ to $v$ correspond to the relays in the original problem. The state $x'[n]$ can be written as $\begin{bmatrix}  x[n] \\ x_{ch}[n] \end{bmatrix}$ where $x[n]$ and $x_{ch}[n]$ are respectively the plant and the internal states of the network in the original problem. Then, the state transition matrix $A'$ is a block diagonal matrix $\begin{bmatrix} A & 0 \\ 0 & A_{ch}\end{bmatrix}$.

However, there are minor differences between `control over LTI networks' and `decentralized linear system control' problems. In `control over LTI networks' problems, we only want to stabilize the plant $x[n]$ not all internal states $x'[n]$. And the state disturbance $w[n]$ is also added to only $x[n]$ not to all internal states $x'[n]$. However, since we assume all the channel matrices are stable, the $A_{ch}$ which correspond to $x_{ch}[n]$ have only stable eigenvalues. The only possibly unstable states are $x[n]$. Therefore, by simply repeating the proof shown in \cite{Wang_Stabilization,Davison_Decentralized}, we can justify that the stabilizability of the realized system $\mathcal{L}_{re}(A_i',B_i',C_i',D_{ij}')$ is still a necessary condition for stabilizability over the LTI network.

Now, we can apply the Jordan form externalization\footnote{Exactly speaking, we have to apply the Jordan form externalization for proper systems shown in Appendix~\ref{app:proper_jd}.} of Section~\ref{sec:jordanex} for all unstable eigenvalues $\lambda$ of $A$. Figure~\ref{fig:ptop_real} shows the resulting LTI network from the Jordan form externalization with respect to $\lambda$. By Theorem~\ref{thm:equivalence2}, we know that $\lambda$ is not a fixed mode only if the mincut of the network in Figure~\ref{fig:ptop_real} is greater than $m_{\lambda}$. First, we can think of the cutset that only includes the transmitter $y_{\lambda}$. The channel matrix for this cut is $C_{\lambda,1}$ and so $\rank C_{\lambda,1} \geq m_{\lambda}$ is a necessary condition for stabilizability. By Corollary~\ref{cor:observability}, this is equivalent to the observability of $\lambda$ which is the condition (i) of the theorem. Likewise, we can think of the cutset that only excludes the receiver $u_{\lambda}$. The channel matrix for this cut is $-B_{\lambda,1}$ and so $\rank B_{\lambda,1} \geq m_{\lambda}$ is a necessary condition. This corresponds to the theorem's condition (ii), the controllability of $\lambda$. The remaining cuts have a one-to-one correspondence to the cuts of the LTI network of Figure~\ref{fig:ptop}. The conditions that these cuts are larger than $m_{\lambda}$ corresponds to the mincut condition of the LTI network, which is the condition (iii) of the theorem.

(2) Sufficiency Proof: For sufficiency, we can also apply the realization idea and use the same sufficiency proof for decentralized linear systems shown in \cite{Wang_Stabilization,Davison_Decentralized}. However, to reveal connections we will give a constructive proof based on network coding, and this style of proof will turn out to be useful in the extensions that we will consider later.

The proof consists of three steps: LTI network design, observer design, and controller design. Without loss of generality, we can assume that $A$ is given in a Jordan form. Then, we can use the notations of Section~\ref{sec:decentrallinear}. For for each unstable eigenvalue $\lambda$ of $A$, define the permutation matrices $P_{R,\lambda}$ and $P_{L,\lambda}$ in the same ways as \eqref{eqn:permutation}. Then, we can apply these permutations to the system input and output matrices $B$ and $C$, and denote the following sub-matrices.
\begin{align}
C \cdot P_{R,\lambda} = \begin{bmatrix} C_{\lambda,1} & C_{\lambda,2} \end{bmatrix}, P_{L,\lambda}^T \cdot B = \begin{bmatrix} B_{\lambda,1} \\ B_{\lambda,2} \end{bmatrix} \nonumber
\end{align}
where $B_{\lambda,1}$ is a $ m_{\lambda} \times q_{cn}$ matrix, and $C_{\lambda,1}$ is a $ r_{ob}\times m_{\lambda}$ matrix. 
We will design the controller, observer and relay gain matrices $K_{cn}, K_{ob}, K_i$. Each element in these gain matrices can be interpreted in two ways, ether as a variable in the form of $k_{i,j,k}$, or as constant in $\mathbb{F}[z]$ (a transfer function set in z-transform). Then, designing the controller gains can be understood as a procedure of plugging in constants in $\mathbb{F}[z]$ to variables. To distinguish these two meanings of $K_i$, as mentioned in Section~\ref{sec:prelim} we will write $K_i$ when it is considered as a variable, and just $K_i(z)$ when it is considered as a constant.\\
\\
(2-a) LTI network (relay) design: 
The goal of the relays is flowing enough information to stabilize all unstable eigenvalues $\lambda$. Denote the transfer function of the LTI network as $G_{ptop}(z,K)$. The goal of the relay gain design is finding $K_i(z) \in \mathbb{F}[z]^{q_i \times r_i}$ such that for all unstable eigenvalues $\lambda$, $\rank(G_{ptop}(\lambda,K))=\rank(G_{ptop}(\lambda,K(z)))$ i.e. achieving the maxflow. Here, because of the condition (iii), the maxflow at $z=\lambda$ is always greater or equal to $m_{\lambda}$ which is enough to stabilize.

Since the complex (or real) field is infinite, we can find the memoryless gain $K_i(z) \in \mathbb{C}^{q_i \times r_i}$ which achieves the maxflow. Rigorously speaking, for each $\lambda$, the algebraic variety that makes the rank of $G_{ptop}(\lambda,K)$ be smaller than its maximum rank has a strictly lower dimension than its underlying space. Therefore, there exists an infinite number of solutions that can achieve the maxflow for each $\lambda$~\cite[Lemma~1]{Koetter_Algebraic}. Moreover, even if we have to achieve the maxflow for different eigenvalues simultaneously, the algebraic variety which reduces the ranks of any of transfer function matrices just corresponds to a union. Therefore, the dimension is still strictly less than its underlying space, and an infinite number of solutions exist.

However, when the LTI network has cycles, just guaranteeing the rank condition from the transmitter to the receiver is not enough. Even though all the channel transfer functions are stable, by introducing the relay gains to the nodes, we can shift the stable poles to unstable poles. To prevent such situations, we will adapt the argument introduced by Wang \textit{et al.} in \cite{Wang_Stabilization}. As shown in \cite{Wang_Stabilization}, using Gershgorin's circle theorem~\cite{Richard} we can prove that as long as the relays gains are chosen small enough, the location of the poles does not move far from the original location. Formally, we can find $\epsilon > 0$ such that for all $|K_i(z)| < \epsilon$ such that $K_i(z) \in \mathbb{C}^{q_i \times r_i}$, all the poles of the LTI network are stable. Moreover, even if we restrict $K_i(z)$ to be the ones satisfying $|K_i(z)| < \epsilon$, the dimension of the algebraic variety remains the same. Therefore, The proof of \cite[Lemma~1]{Koetter_Algebraic} still holds, and the same argument above guarantees the existence of the mincut achieving $K_i(z)$ which keeps the whole LTI network stable.\\
\\
(2-b) Observer design: The goal of the observer design is simply connecting all the unstable states of the plant to the LTI network. Mathematically, finding $K_{ob}(z) \in \mathbb{C}^{q_{ob} \times r_{ob}}$ such that for all unstable eigenvalue $\lambda$, $\rank(G_{ptop}(\lambda, K(z)) K_{ob} C_{\lambda,1})=\rank(G_{ptop}(\lambda, K(z)) K_{ob}(z) C_{\lambda,1})$. Here, we can see since the elements of $K_{ob}$ are variables, $\rank(G_{ptop}(\lambda, K(z)) K_{ob} C_{\lambda,1}) = \min (\rank(G_{ptop}(\lambda, K(z))), \rank(C_{\lambda,1}))$. Therefore, by the relay design (2-a) and the condition (i) ---together with Corollary~\ref{cor:observability}--- we can conclude $\rank(G_{ptop}(\lambda, K(z)) K_{ob} C_{\lambda,1}) \geq m_{\lambda}$. Using the same algebraic variety argument as (2-a), we can prove the existence of such $K_{ob}(z)$. (Here, we do not need Gershgorin's circle theorem for stability.)\\
\\
(2-c) Controller design: The goal of the controller is to actually stabilize the plant based on the information it got. Once the design of the observer and the relays are fixed, from the controller's point of view the whole system can be viewed as follows in $z$-transform:
\begin{align}
&z x(z)=A x(z)+B u(z) \\
&y(z)=C'(z)x(z)
\end{align}
where $C'(z)=G_{ptop}(z,K(z))K_{ob}(z) C$. For each unstable eigenvalue $\lambda$ of $A$, let's apply the same permutation matrix $P_{R,\lambda}$ to $C'(z)$ and denote the following sub-matrices as $C'(z) \cdot P_{R,\lambda} = \begin{bmatrix} C'_{\lambda,1}(z) & C'_{\lambda,2}(z) \end{bmatrix}$. Then, we can easily see $C'_{\lambda,1}(z)=G_{ptop}(z,K(z)) K_{ob}(z) C_{\lambda,1}$. Moreover, a simple extension of Corollary~\ref{cor:observability} gives that in this new system, $\lambda$ is observable if and only if $\rank(C'_{\lambda,1}(\lambda)) \geq m_{\lambda}$. We already know this condition holds for all unstable eigenvalues $\lambda$. Moreover, by condition (ii) all unstable eigenvalues are controllable, and we can stabilize the system using a conventional controller design~\cite{Chen}.

This finishes the sufficiency proof.
\begin{figure}
\includegraphics[width = 4in]{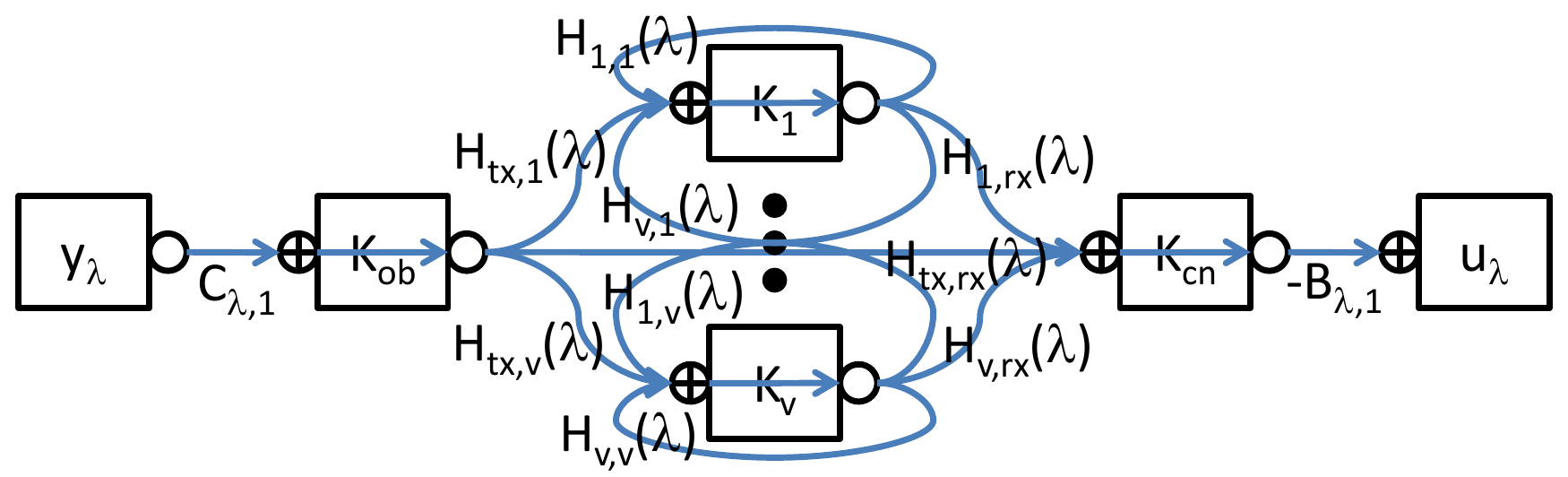}
\caption{Jordan form externalization of $\mathcal{L}_{re}(A_i',B_i',C_i',D_{ij}')$ at $z=\lambda$}
\label{fig:ptop_real}
\end{figure}
\end{proof}
In the proof of the theorem, we saw how the Jordan form externalization of implicit information flows discussed in Section~\ref{sec:jordanex} can be used to understand problems which have both control and communication aspects. Moreover, the connection between network coding and implicit information flows for control leads to a new controller design for stabilizing the plant. 

More importantly, the ideas used in the proof justifies our intuition on information flows in decentralized linear system shown in Section~\ref{sec:example}, especially Table~\ref{tbl:comparison}. We converted `control over LTI networks' problems to decentralized linear systems by considering the relays in LTI networks as controllers of decentralized systems and the channels as a part of the states and input-output matrices $B_i, C_i$. The goal of the observer and the relays was to send enough information about unstable states associated with $\lambda$. Therefore, the unstable states can be considered the source of information flows, and the unstable subspaces can be thought of as the message. The maxflow of the LTI network was compared with $m_{\lambda}$, the number of Jordan blocks associated with $\lambda$. Therefore, $m_{\lambda}$ can be considered the rate of the message. The controller stabilized the plant by controlling the unstable states based on its received information. Therefore, the unstable states can also be thought of as the destination of information flows. Theorem~\ref{thm:LTI:ptop} reveals that we can stabilize the system if and only if the LTI network has enough capacity to afford the information flows for control. Therefore, the capacity of LTI networks is deeply related to stabilizability of control systems. Moreover, the communication scheme that we used for the relays was linear network coding.

Another important point is the relationship between network linearization that we discussed in Section~\ref{sec:networklin} and control over LTI networks. By comparing Figure~\ref{fig:LN_ptop} and Figure~\ref{fig:ptop_real}, we can easily notice the similarity. The transmitter and  receiver in LTI communication networks correspond to the observer and  controller in control over LTI networks. These nodes are connected by relay nodes in both problems. Now we can see that what we did by introducing the circulation arc in network linearization (in Figure~\ref{fig:LN_ptop}) is essentially introducing an unstable plant to be stabilized through the LTI communication network. This insight will be helpful in the later generalization of control over LTI networks, and also the generalization of network linearization in Appendix~\ref{app:networklin}.

%
%
%
%
%

\subsection{Multicast}
\label{sec:stablilizationoverLTI:multicast}
\begin{figure}
\includegraphics[width = 5in]{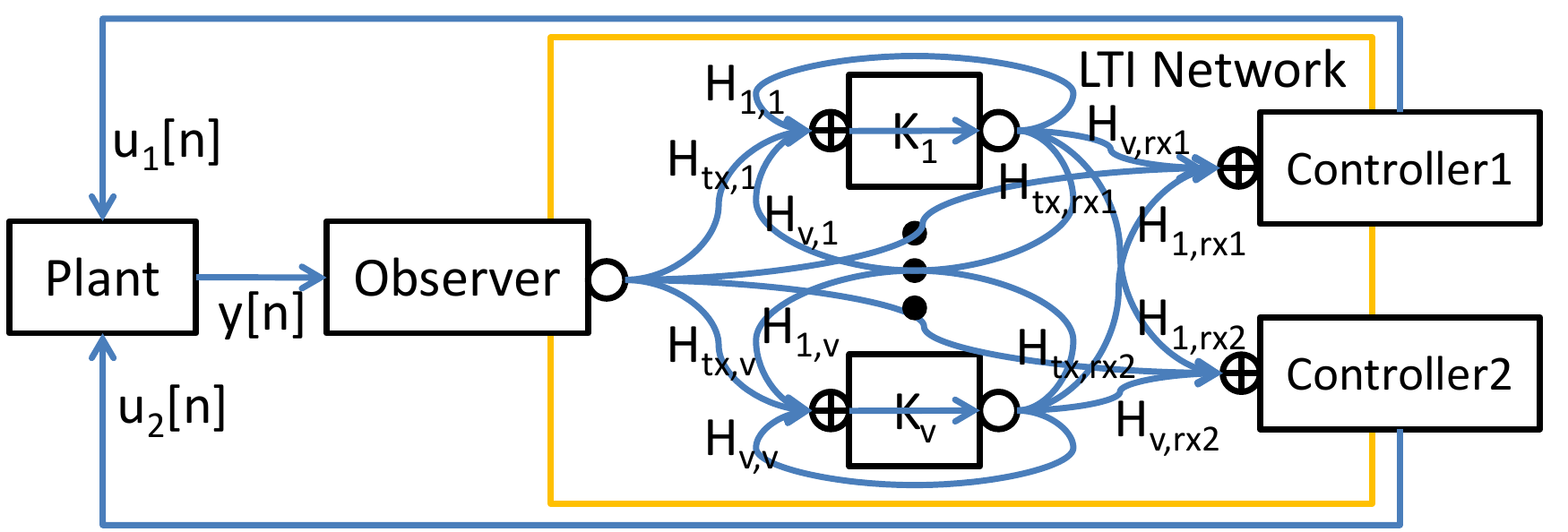}
\caption{Control over LTI Networks with multiple controllers: Multicast case}
\label{fig:multicast}
\end{figure}
Now, we understand that the distributed controllers communicate by network coding. However, it is known in the communication community that network coding is really helpful to improve the performance when the problem involves multiple transmitters and receivers. Therefore, we will extend the previous single-plant single-observer single-controller problems to the problem with multiple plants, observers, and controllers. We will see a close relationship and parallelism between control over LTI networks and network coding.

Arguably, the easiest and most well-understood problem among multi-user network coding problems is the multicast problem. In multicast problems, there are a single transmitter and multiple receivers, and all the receivers want to receive a common message from the transmitter. The worst mincut to all receivers is a trivial lower bound for the message rate in multicast problems. It is shown~\cite{Ahlswede_Network} that we can achieve this lower bound and network coding is necessary for this.

Let's find the counterpart of multicast problems in control over LTI networks. In the sufficiency proof of Thereom~\ref{thm:LTI:ptop}, we saw that the destination of the information flow for control is the controller.\footnote{Even if the ultimate destination of the information flow is the unstable states, in control over LTI network problems, only the controller can control the plant. The controller can be thought as a destination.} Therefore, the controller are the receivers, and so we have to increase the number of controllers to find the counterpart of multicast problems.

The situation that we will consider in this section is following. Consider control over LTI networks problem with two controllers as shown in Figure~\ref{fig:multicast}. Let's say we want to design the system so that the plant becomes stable by either one of the controllers --- but does not have to be stable when both controllers are active. To design such systems, we can introduce the multicast communication scheme for LTI network so that the observer sends enough information to stabilize the plant to both controllers.

For simplicity, let's limit our discussion to two controllers but all the results in this section can be easily generalized to multiple controllers. Figure~\ref{fig:multicast} shows the resulting problem, control over LTI networks with two controllers. Formally, the plant has two control inputs $u_1$ and $u_2$, i.e. the plant is given as
\begin{align}
&x[n+1]=Ax[n]+B_1u_1[n]+B_2u_2[n]+w[n]\\
&y[n]=Cx[n]
\end{align}
where $A \in \mathbb{C}^{m \times m}$, $B_1 \in \mathbb{C}^{m \times q_{cn 1}}$, $B_2 \in \mathbb{C}^{m \times q_{cn 2}}$ and $C \in \mathbb{C}^{r_{ob}\times m}$.
If the observations of the observer is decodable at the both controllers, it is possible to stabilize the plant by either one of two controllers. The following definition captures this idea.
\begin{definition}[Alternative Stabilizability]
Given the above definitions, we say that the plant is \textbf{alternatively stabilizable over the LTI network} if there exist `common' LTI observer and relays, and possibly different controllers that makes both the first plant
\begin{align}
&x[n+1]=Ax[n]+B_1 u_1[n] + w[n] \\
&y[n]=Cx[n]
\end{align}
and the second plant
\begin{align}
&x[n+1]=Ax[n]+B_2 u_2[n] + w[n] \\
&y[n]=Cx[n]
\end{align}
stable over the LTI network.
\end{definition}
The reason why this problem is different from just two separate problems with a single controller is that the same observer and relays have to be used for two different systems.

Let the LTI network that includes the observer, relays and controller $1$ be $\mathcal{N}_{mul1}(z)$. Likewise, the LTI network including the observer, relays and controller $2$ is denoted by $\mathcal{N}_{mul2}(z)$. The other notations and assumptions about the problem are the same as the point-to-point case. Then, the condition for alternative stabilizability is given as follows.
\begin{theorem}
Given the above definitions, the plant is alternatively stabilizable over the LTI network if and only if for all $\lambda$ such that $\lambda \in \{ \lambda : |\lambda| \geq 1 \} \cap \sigma(A)$ the following conditions are satisfied
\begin{align}
(i) &\begin{bmatrix}
\lambda I - A \\
C
\end{bmatrix}\mbox{ is full rank}\\
(ii) &\begin{bmatrix} \lambda I - A  & B_1 \end{bmatrix}\mbox{ and }\begin{bmatrix} \lambda I - A  & B_2 \end{bmatrix}\mbox{ are both full rank}\\
(iii)& m_\lambda \leq \mbox{(mincut rank of the LTI network $\mathcal{N}_{mul1}(\lambda)$)}\\
&m_\lambda \leq \mbox{(mincut rank of the LTI network $\mathcal{N}_{mul2}(\lambda)$)}
\end{align}
\label{thm:LTI:multicast}
\end{theorem}
\begin{proof}
(1) Necessity Proof: Since the plant has to be stabilizable by both the controller $1$ and $2$, the conditions of Theorem~\ref{thm:LTI:ptop} has to be satisfied for both controllers, which corresponds to the condition (i), (ii), (iii) of the theorem.

(2) Sufficiency Proof: Just as the sufficiency proof of Theorem~\ref{thm:LTI:ptop}, we will give a three-step constructive proof. Since the only difference from that of Theorem~\ref{thm:LTI:ptop} is LTI network desing, we use the essentially definitions.\\
(2-a) LTI network design: Since we have to afford enough information flow for both controllers, we choose the relay gain matrices $K_i(z) \in \mathbb{C}^{q_i \times r_i}$ such that for all unstable eigenvalue $\lambda$, $\rank(G_{mul1}(\lambda,K(z))) \geq m_{\lambda}$ and $\rank(G_{mul2}(\lambda,K(z))) \geq m_{\lambda}$. The existence of such gain matrices can be proved in the same way as Theorem~\ref{thm:LTI:ptop} and using the condition (iii). In other words, the set that we cannot choose $K_i(z)$ is the union of two algebraic varieties: one that makes $G_{mul1}(\lambda,K_i)$ lose its rank and the other one that makes $G_{mul2}(\lambda,K_i)$ lose its rank. The dimension of their union is also strictly smaller than that of the underlying space. Therefore, almost all $K_i(z) \in \mathbb{C}^{q_i \times r_i}$ can achieve the maximum rank of both transfer functions.\\
\\
(2-b) Observer Design: For the observer design, we find $K_{ob}(z) \in \mathbb{C}^{r_{ob} \times q_{ob}}$ such that for all unstable eigenvalue $\lambda$, $\rank(G_{mul1}(\lambda,K(z)) K_{ob}(z) C_{\lambda,1})\geq m_{\lambda}$ and $\rank(G_{mul2}(\lambda,K(z)) K_{ob}(z) C_{\lambda,1})\geq m_{\lambda}$. The existence of such $K_{ob}(z)$ follows from the same way as Theorem~\ref{thm:LTI:ptop} and the union of two algebraic variety argument.\\
\\
(2-c) Controller Design: Now, at both controllers the plant is observable. We can simply use conventional controller designs to stabilize the system by both controllers.
\end{proof}
Like the point-to-point problem, memoryless observer and relays are enough for alternative stabilizablity. The generalization of this result to more than two controllers is trivial. We can simply add more controller conditions to the condition (ii) and (iii).

In fact, this theorem can be generalized to arbitrary decentralized linear systems. First, we define strong connectivity of decentralized linear systems~\cite{Cormat_Decentralized}.
\begin{definition}\cite{Cormat_Decentralized}
A proper decentralized linear system $\mathcal{L}(A,B_i,C_i)$ with $v$ decentralized controllers is called strongly connected if for all $V \subset \{1,\cdots, v \}$,
$C_V (zI -A )^{-1}B_{V^c}$ is nonzero.
\end{definition}
The strong connectivity of the decentralized system implies that for any cuts, the transfer function across this cut is not zero. In other word, we can always send some information for any cuts, and thereby every controller is connected with each other.

We generalize the alternative stabilizability definition to a set of decentralized linear systems.
\begin{definition}
Consider a set of $p$ decentralized linear systems with $v$ decentralized controllers,
\begin{align}
\{\mathcal{L}(A^{(1)},B_i^{(1)},C_i^{(1)}),\cdots, \mathcal{L}(A^{(p)},B_i^{(p)},C_i^{(p)}\}.
\end{align}
where for all $2 \leq i \leq v$ the dimensions of $B_i^{(1)}, \cdots, B_i^{(p)}$ are the same, and the dimensions of $C_i^{(1)},\cdots, C_i^{(p)}$ are also the same.\footnote{The dimension of $B_1^{(1)}, \cdots, B_1^{(p)}$ and the dimension of $C_1^{(1)},\cdots, C_1^{(p)}$ can be different.} This set of the decentralized systems is called \textbf{alternatively stabilizable} if there exist common LTI controllers $\mathcal{K}_2, \cdots, \mathcal{K}_v$ and possibly different\footnote{the design of the first controller $\mathcal{K}_{1}^{(i)}$ can be changed depending on which system is chosen.} controllers $\mathcal{K}_1^{(1)}, \cdots, \mathcal{K}_1^{(p)}$ such that for all $1\leq k \leq p$, all systems $\mathcal{L}(A^{(k)},B_i^{(k)},C_i^{(k)})$ with controllers $\mathcal{K}_1^{(k)},\mathcal{K}_2, \cdots, \mathcal{K}_v$ are stable simultaneously.
\end{definition}

The above definition implies that even if the decentralized system is arbitrary chosen from a given set, we can stabilize the system by changing only one controller (the controller $1$). We can relate this problem with the previous control over LTI network problem. We can consider the observer and relays of control over LTI networks as the controller $2$ to $v$ in decentralized systems. We can consider the multiple controllers as the controller $1$ in decentralized systems. Therefore, from the realization idea, we can see the alternative stabilization of decentralized linear systems includes that of control over LTI networks as a special case.

This generalized problem corresponds to robust networking~\cite{Koetter_Algebraic} in a network coding context. In robust networking, the communication network can be adversarially chosen from a given set, and we want to design the relay scheme that achieves the worst case mincut. In \cite{Koetter_Algebraic}, it is shown that robust networking is essentially the same as multicast problems, and the worst case mincut is achievable using the network coding.

Likewise, the alternative stabilizability of decentralized linear systems is essentially the same as that of control over LTI networks. If the systems are strongly connected, the alternative stabilizability condition is given as follows.
\begin{theorem}
Consider a set of decentralized linear systems with $v$ controllers
\begin{align}
\{\mathcal{L}(A^{(1)},B_i^{(1)},C_i^{(1)}),\cdots, \mathcal{L}(A^{(p)},B_i^{(p)},C_i^{(p)})\}
\end{align}
where each decentralized linear system is strongly connected.\footnote{Otherwise, controllers can be isolated from the remaining system.} Then, this set of the decentralized linear systems is alternatively stabilizable if and only if each decentralized linear system does not have unstable fixed modes.
\label{thm:LTI:multicast2}
\end{theorem}
\begin{proof}
The necessity is obvious since each system has to be stabilizable.

Let's prove the sufficiency. By \cite[Corollary 1]{Cormat_Decentralized}, we know that except a certain algebraic variety whose dimension is strictly smaller than that of the underlying space, almost all constant matrices $K_2(z), \cdots, K_v(z)$ make all unstable eigenvalues of $\mathcal{L}(A^{(1)}, B_i^{(1)},C_i^{(1)})$ to be observable and controllable at the controller $1$. Moreover, by Gershgorin's circle theorem~\cite{Richard}, there exists $\epsilon > 0$ such that for all $|K_i(z)| \leq \epsilon$ such that $K_i(z) \in \mathbb{C}^{q_i \times r_i}$, the stable eigenvalues of the system remain stable.

Using the union of algebraic variety argument, we can prove that there exist constant matrices $K_2(z) \in \mathbb{C}^{q_2 \times r_2}, \cdots, K_v(z) \in \mathbb{C}^{q_v \times r_v}$ such that for all systems $\{\mathcal{L}(A^{(1)},B_i^{(1)},C_i^{(1)}),\cdots, \mathcal{L}(A^{(p)},B_i^{(p)},C_i^{(p)})\}$, the unstable eigenvalues are observable and controllable at the controller $1$ and the stable eigenvalues remains stable. Then, knowing which system is chosen, the first controller can stabilize the system using a conventional design~\cite{Chen}.
\end{proof}
Just as the sufficiency of Theorem~\ref{thm:LTI:multicast}, memoryless controllers are enough for the controller $2$ to $v$. The underlying reason why this theorem holds is that the controllers $2$ to $v$ relays enough information for control to the controller $1$ by network coding.

\subsection{Broadcast}
\label{sec:stablilizationoverLTI:broadcast}
\begin{figure}
\includegraphics[width = 5in]{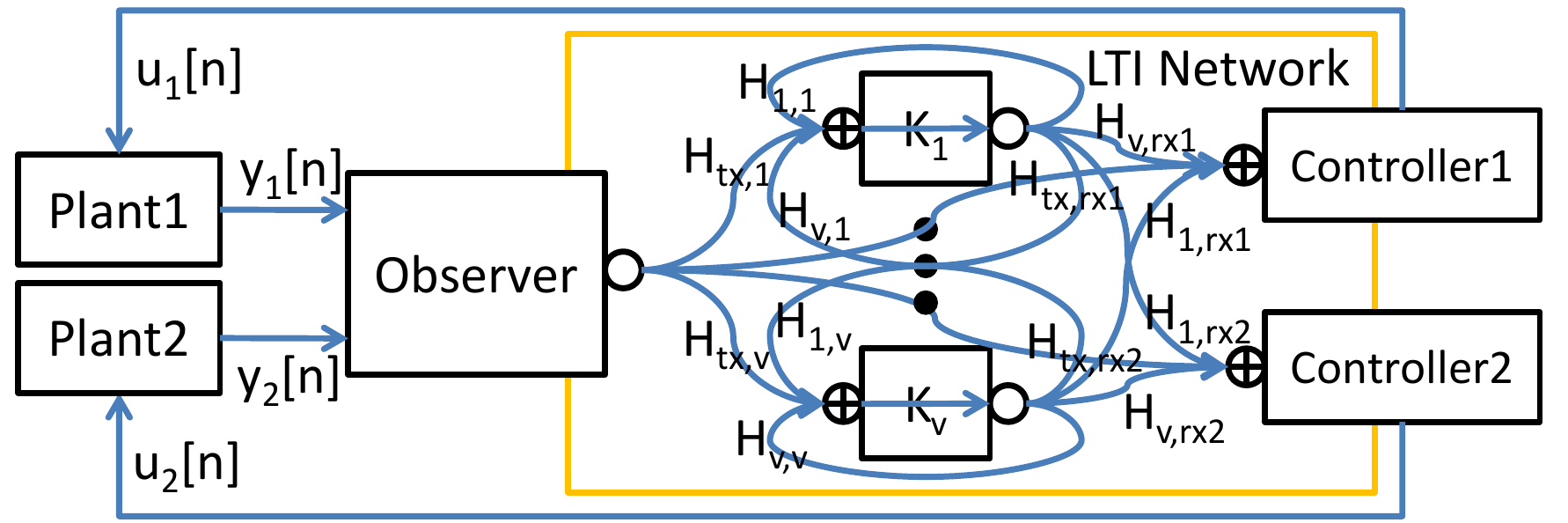}
\caption{Stabilization over LTI Network with multiple plants and multiple controllers: Broadcast case}
\label{fig:broadcast}
\end{figure}
Another well-understood problem in network coding is the broadcast. Like multicast problems, broadcast problems have a single transmitter and multiple receivers. However, unlike multicast problems, each receiver wants to receiver its own message which is independent from the other's. We can find a simple lower bound on the message rate using cutset bounds. The message rate to the receiver $1$ cannot exceed the cutset bound for the receiver $1$, and similar bounds hold for all receivers. We can also think of sum cutsets for augmented receivers. The sum of the message rates to the receiver $1$ and receiver $2$ cannot exceed the cutset bound for the augmented receiver $1$ and $2$. Likewise, we can think of the cutset bounds for the sum of all two messages, three messages, and so on. This cutset bound is also known to be achievable using network coding together with precoding at the transmitter~\cite{Li_Network,Koetter_Algebraic}.

In this section, we will find a counterpart of broadcast problems in control over LTI networks. As we saw in the previous section, multiple receivers in network coding problems correspond to multiple controllers. Now, we have to find the counterpart of multiple messages. In previous discussions, we found that   the unstable states correspond to the messages. Therefore, as a counterpart of independent messages, we introduce multiple plants which have orthogonal unstable states. Each controller can only control its designated plant. 

Consider the control over LTI network problems with two plants and two controllers as shown in Figure~\ref{fig:broadcast}. Obviously, we want to design the system so that both plants becomes stable. However, we will require an additional property of disturbance rejection. In other words, if we add disturbance only to the plant 1, the states of the plant 2 stay zero for all time. Likewise, if we add disturbance only to the plant 2, the states of the plant 1 stay zero for all time. In other words, the disturbance added to the plant 1 does not propagated to the plant 2, and vice versa.

For notational simplicity, we will only consider the two plants and two controllers case, but the results in this section can be easily generalized to multiple plants and multiple controllers. Figure~\ref{fig:broadcast} shows the resulting control over LTI network problem with two plants and two controllers. The plant models are given as follows:
\begin{align}
&x_1[n+1]=A_1 x_1[n] + B_1 u_1[n] + w_1[n]\\
&y_1[n]=C_1 x_1[n] 
\end{align}
\begin{align}
&x_2[n+1]=A_2 x_2[n] + B_2 u_2[n] + w_2[n]\\
&y_2[n]=C_2 x_2[n]
\end{align}
where $A_i \in \mathbb{C}^{m_i \times m_i}$, $B_i \in \mathbb{C}^{m_i \times q_{cn i}}$, and $C_i \in \mathbb{C}^{r_{ob i} \times m_i }$. As shown in Fig.~\ref{fig:broadcast}, the observer has both observations $y_1[n]$ and $y_2[n]$, but both controllers can only control their designated plants via $u_1[n]$ and $u_2[n]$. The basic assumptions and notations for the LTI network are the same as the multicast problem.

If just as broadcast problems the observation $y_1[n]$ (information about $x_1[n]$) is decodable separately from $y_2[n]$ at the controller $1$ and the observation $y_2[n]$ (information about $x_2[n]$) is decodable separately from $y_1[n]$ at the controller $2$, it is possible for the controllers to control their designated plants without causing any interference to the others. In other words, the controller $1$ can control the plant $1$ without causing any additional disturbance to the plant $2$, and likewise the controller $2$ can control the plant $2$ without causing any additional disturbance to the plant $1$. This notion is the following definition of independent stabilizablity.
\begin{definition}[Independent Stabilizability]
Given the above definitions, we say that the plants are \textbf{independently stabilizable over the LTI network} if there exist the LTI observer, controllers and relays that satisfy the following conditions:\\
(i) both of the plants are stable over the LTI network \\
(ii) If $w_1[n]=0$ for all $n$, then $x_1[n]=0$ for all $n$ regardless of $w_2[n]$\\
(iii) If $w_2[n]=0$ for all $n$, then $x_2[n]=0$ for all $n$ regardless of $w_1[n]$
\end{definition}
In Figure~\ref{fig:broadcast}, denote the LTI network including the observer, the relays and the controller $1$ as $\mathcal{N}_{br1}(z)$. Likewise, denote the LTI network that including the observer, the relays and the controller $2$ as $\mathcal{N}_{br2}(z)$. The LTI network that has the controller $1$ and $2$ as the augmented receiver is denoted as $\mathcal{N}_{br1,2}(z)$.

We put $m_{1,\lambda}$ be the number of the Jordan blocks of $A_1$ associated with the eigenvalue $\lambda$, and $m_{2,\lambda}$ be that for $A_2$. We also denote $m_{1,max}:=\max_{\lambda \in \mathbb{C},|\lambda| \geq 1} m_{1,\lambda}$ and $m_{2,max}:=\max_{\lambda \in \mathbb{C},|\lambda| \geq 1} m_{2,\lambda}$.


One may think since we have to prevent the disturbance propagation for independent stabilizability, the existence of separate paths from the observer to each controller is required for independent stabilizability. However, we do not need separate paths to each controller. For example, let the plant $1$ and $2$ be scalar plants. Let the observer have two dimensional input signal $\begin{bmatrix} u_{ob,1}[n] \\ u_{ob,2}[n] \end{bmatrix}$ to the network, the controller $1$ and $2$ have one dimensional $y_{cn1}[n]$ and $y_{cn2}[n]$ respectively, and their relation be given as
\begin{align}
\begin{bmatrix}
y_{cn1}[n] \\
y_{cn2}[n]
\end{bmatrix}
=
\begin{bmatrix}
2 & 1\\
1 & 2
\end{bmatrix}
\begin{bmatrix}
u_{ob,1}[n]\\
u_{ob,2}[n]
\end{bmatrix}
.
\end{align}
We further assume the network have no relays. In this example, one may think that it is impossible to independently stabilize the system since the communication channels to each controller interfere with each other. However, by simply introducing a precoding gain $\begin{bmatrix} 2 & 1 \\ 1 &2 \end{bmatrix}^{-1}$, we can orthogonalize the paths and independently stabilize the system.

This idea can be formalized for general cases. A sufficient condition and a necessary condition for the independent stabilizability are given as follows.
\begin{theorem}
Given the above definitions, the sufficient condition for the plants to be independently stabilizable is that for all $\lambda$ such that $\lambda \in \{\lambda:|\lambda| \geq 1 \} \cap ( \sigma(A_1) \cup \sigma(A_2) )$ the following conditions hold:
\begin{align}
(i) &\begin{bmatrix}
\lambda I - A_1 \\
C_1
\end{bmatrix} \mbox{ and }
\begin{bmatrix}
\lambda I - A_2 \\
C_2
\end{bmatrix}
\mbox{ are both full rank} \\
(ii) &\begin{bmatrix} \lambda I - A_1  & B_1 \end{bmatrix}\mbox{ and }\begin{bmatrix} \lambda I - A_2  & B_2 \end{bmatrix}\mbox{ are both full rank}\\
(iii) &\ m_{1,max}+m_{2,max} \leq \mbox{(mincut rank of the LTI network $\mathcal{N}_{br1,2}(\lambda)$)}\\
&\ m_{1,max} \leq \mbox{(mincut rank of the LTI network $\mathcal{N}_{br1}(\lambda)$)}\\
&\ m_{2,max} \leq \mbox{(mincut rank of the LTI network $\mathcal{N}_{br2}(\lambda)$)}
\end{align}
The necessary condition for the plants to be independently stabilizable is that for all $\lambda$ such that $\lambda \in \{\lambda : |\lambda| \geq 1\} \cap ( \sigma(A_1) \cup \sigma(A_2) )$ the following conditions hold:
\begin{align}
(i) &\begin{bmatrix}
\lambda I - A_1 \\
C_1
\end{bmatrix} \mbox{ and }
\begin{bmatrix}
\lambda I - A_2 \\
C_2
\end{bmatrix}
\mbox{ are both full rank} \\
(ii) &\begin{bmatrix} \lambda I - A_1  & B_1 \end{bmatrix}\mbox{ and }\begin{bmatrix} \lambda I - A_2  & B_2 \end{bmatrix}\mbox{ are both full rank}\\
(iii) &\ m_{1,\lambda}+m_{2,\lambda} \leq \mbox{(mincut rank of the LTI network $\mathcal{N}_{br1,2}(\lambda)$)}\\
&\ m_{1,\lambda} \leq \mbox{(mincut rank of the LTI network $\mathcal{N}_{br1}(\lambda)$)}\\
&\ m_{2,\lambda} \leq \mbox{(mincut rank of the LTI network $\mathcal{N}_{br2}(\lambda)$)}
\end{align}
\label{thm:LTI:broadcast}
\end{theorem}
\begin{proof}
(1) Necessary condition: The plant $1$, the plant $2$, and their augmented plant have to be stabilizable by the controller $1$, the controller $2$, and their augmented controller. Therefore, by apply theorem~\ref{thm:LTI:ptop} to these systems, we get the necessary conditions.

(2) Sufficient condition: 

The proof is similar to that of Thereom~\ref{thm:LTI:multicast}, but here we need an additional step to remove the interference between the information flows to two controllers. For this, we will use an pre-and-post processing idea shown in \cite{Li_Network,Koetter_Algebraic}.

(2-a) LTI Network design: 

Let $G_{br1}(z,K)$ and $G_{br2}(z,K)$ be the transfer function matrices of $\mathcal{N}_{br1}(z)$ and $\mathcal{N}_{br2}(z)$ respectively. Then, we can see $G_{br1,2}(z,K):=\begin{bmatrix}G_{br1}(z,K) \\ G_{br2}(z,K) \end{bmatrix}$ is the transfer function matrics of $\mathcal{N}_{br1,2}(z)$. 
Using the same union of algebraic varieties argument of Theorem~\ref{thm:LTI:multicast}, by the condition (iii) we can prove that there exist $K_i(z) \in \mathbb{C}^{q_i \times r_i}$ such that for all unstable eigenvalue $\lambda$
\begin{align}
&\rank(G_{br1}(\lambda,K(z)) \geq m_{1,max}\\
&\rank(G_{br2}(\lambda,K(z)) \geq m_{2,max}\\
&\rank(G_{br1,2}(\lambda,K(z))) \geq m_{1,max}+m_{2,max} \label{eqn:multicast1}
\end{align}
and keep the stable eigenvalues stable.

(2-b) Pre-and-Post processors at Controller and Observer: Even if we design the relays so that they can flow enough information, information flows from the observer to the controllers can interfere with each other. To remove this interference, we introduce pre-and-post processors at the controllers and observer as shown in  \cite{Li_Network,Koetter_Algebraic}. 

First, let's make $G_{br1,2}(z,K(z))$ a square matrix by introducing pre-and-post processors $K_{cn1}'(z) \in \mathbb{C}^{m_{1,max}\times r_{cn1}}$, $K_{cn2}'(z) \in \mathbb{C}^{m_{2,max}\times r_{cn2}}$, $K_{ob}'(z) \in \mathbb{C}^{q_{ob} \times (m_{1,max}+m_{2,max})}$ as follows:
\begin{align}
G'_{br1,2}(z,K(z)):=\begin{bmatrix}K_{cn1}'(z) & 0 \\ 0 & K_{cn2}'(z) \end{bmatrix} G_{br1,2}(z,K(z)) K_{ob}'(z).
\end{align}
The resulting matrix $G'_{br1,2}(z,K(z))$ is a square matrix with dimension $(m_{1,max}+m_{2,max})$, and using algebraic variety argument and \eqref{eqn:multicast1} we can choose $K_{cn1}'(z)$, $K_{cn2}'(z)$, $K_{ob}'(z)$ so that for all unstable eigenvalue $\lambda$, $G_{br1,2}'(\lambda, K(z)) $ is invertible. 

Now, we can remove the interference by simply multiplying by the matrix inverse. To this end, denote
\begin{align}
K_{ob}''(z):=z^{-d} det(G'_{br1,2}(z,K(z))) G'_{br1,2}(z,K(z))^{-1}
\end{align}
Here, we introduce $z^{-d}$ to make $K_{ob}''(z)$ causal. Therefore, $d \in \mathbb{Z}^+$ has to be chosen large enough so that each element in $K_{ob}''(z)$ is causal. Furthermore, since we multiplied $det(G'_{br1,2}(z,K(z)))$, $K_{ob}''(z)$ does not have any additional pole other than the existing ones in $G'_{br1,2}(z,K(z))$. Thus, $K_{ob}''(z)$ is also stable. Let's multiple this matrix to $G'_{br1,2}(z,K(z))$ and denote
\begin{align}
G''_{br1,2}(z,K(z)) := G'_{br1,2}(z,K(z)) K_{ob}''(z).
\end{align}
In $G''_{br1,2}(z,K(z))$, the only non-zero entries are diagonal entries, and so we have $(m_{1,max}+m_{2,max})$ ``orthogonal" communication channels.

(2-c) Observer design: In the observer, we will use $m_{1,max}$ communication channels to send information about the plant $1$, and the remaining for the plant $2$. 
First, denote $C_{1,\lambda,1}$ and $C_{2,\lambda,1}$ for $C_1$ and $C_2$ in the same way we defined $C_{\lambda,1}$ for $C$ in Theorem~\ref{thm:LTI:multicast}. Using the algebraic variety argument as Theorem~\ref{thm:LTI:multicast} and the condition (i), we can show that there exist $K_{ob}'''(z) \in \mathbb{C}^{m_{1,max}\times q_{cn1}}$ and $K_{ob}''''(z) \in \mathbb{C}^{m_{2,max}\times q_{cn2}}$ such that for all unstable eigenvalue $\lambda$,
\begin{align}
&\rank(K_{ob}'''(z) C_{1,\lambda,1})\geq m_{1,\lambda} \\
&\rank(K_{ob}''''(z) C_{2,\lambda,1}) \geq m_{2,\lambda}.
\end{align}
Then, we will put the observer gain $K_{ob}(z)$ as
\begin{align}
K_{ob}(z)= K_{ob}'(z) K_{ob}''(z) \begin{bmatrix}
K_{ob}'''(z) & 0 \\
0 & K_{ob}''''(z)
\end{bmatrix}.
\end{align}

(2-d) Controller design: Once we fix the relay gain and observer gain matrices as above and introduce the gain matrix $K_{cn1}'(z)$ at the controller $1$, by the construction the controller $1$ will have the following observation about the state.
\begin{align}
&
\begin{bmatrix}
K_{cn1}'(z) & 0
\end{bmatrix}
G_{br1,2}(z,K(z)) K_{ob}'(z) K_{ob}''(z)
\begin{bmatrix}
K_{ob}'''(z) & 0 \\
0 & K_{ob}''''(z)
\end{bmatrix}
\begin{bmatrix}
y_1(z) \\
y_2(z)
\end{bmatrix}\\
&=
\begin{bmatrix}
K_{cn1}'(z) & 0
\end{bmatrix}
G_{br1,2}(z,K(z)) K_{ob}'(z) K_{ob}''(z)
\begin{bmatrix}
K_{ob}'''(z) & 0 \\
0 & K_{ob}''''(z)
\end{bmatrix}
\begin{bmatrix}
C_1 x_1(z) \\
C_2 x_2(z)
\end{bmatrix}\\
&= z^{-d} \det(G_{br1,2}'(z)) \begin{bmatrix} I & 0 \end{bmatrix}
\begin{bmatrix}
K_{ob}'''(z) & 0 \\
0 & K_{ob}''''(z)
\end{bmatrix}
\begin{bmatrix}
C_1 x_1(z) \\
C_2 x_2(z)
\end{bmatrix}\\
&= z^{-d} \det(G_{br1,2}'(z)) K_{ob}'''(z) C_1 x_1(z)
\end{align}
As we can see, the observation is orthogonal to the state of plant $2$. Moreover, since for all unstable eigenvalue $\lambda$, $\det(G_{br1,2}'(\lambda)) \neq 0$ and $K_{ob}'''(z) C_1$ can observe all unstable states of $x_1[n]$, the plant $1$ is observable. Therefore, by a conventional controller design, the controller $1$ can orthogonally stabilize the plant $1$. The same holds for the plant $2$ and controller $2$.
\end{proof}
The result can be easily generalized to multiple plants and multiple observers. Unlike Theorem~\ref{thm:LTI:ptop} and Theorem~\ref{thm:LTI:multicast}, the memories at the observer and the relies are actually helpful. The necessary and the sufficient condition coincide when all the unstable eigenvalues of $A_1$ and $A_2$ are the same, and this corresponds to the broadcast result of network coding. 

However, unlike broadcast problems in network coding, the augmentation idea of nodes and cutset bounds fail to give a tight necessary condition. The reason for this is in this problem we have an additional factor, the frequency $z$. According to the frequency where it is evaluated, the channel behaves significantly differently. Thus, there is no way to orthogonalize the channel simultaneously for all frequencies, and we cannot achieve the necessary condition obtained by the augmentation idea. 

For example, let's consider the plant $A_1=3$, $A_2=2$, $B_1=B_2=1$ and $C_1=C_2=1$. And the LTI network has no relays, the input signal dimension of the observer and the output signal dimension of the controllers are $1$, and $G_{br1,2}(z,K_i)=\begin{bmatrix} 3-6z^{-1} \\ 2-6z^{-1} \end{bmatrix}$. 
Here, since there are two scalar plants and the observer has only one dimensional input signal to the network, it ``seems impossible" to independently stabilize the systems. In fact, this system violates the sufficiency condition of Theorem~\ref{thm:LTI:broadcast} since $m_{1, max}=1$, $m_{2, max}=1$, and the mincut ranks of $\mathcal{N}_{br1}(3)$, $\mathcal{N}_{br2}(2)$ are both $1$. Therefore, Theorem~\ref{thm:LTI:broadcast} fails to guarantee independent stabilizability of the system.

However, the system still satisfies the necessary condition of Theorem~\ref{thm:LTI:broadcast} derived by a simple augmented system idea. We can easily check that the system parameters are $m_{1,3}=1$, $m_{2,3}=0$, (mincut rank of $\mathcal{N}_{br1,2}(3)$)=1, (mincut rank of $\mathcal{N}_{br1,2}(3)$)=1, (mincut rank of $\mathcal{N}_{br1,2}(3)$)=0, $m_{1,2}=0$, $m_{2,2}=1$, (mincut rank of $\mathcal{N}_{br1,2}(2)$)=1, (mincut rank of $\mathcal{N}_{br1,2}(2)$)=0, (mincut rank of $\mathcal{N}_{br1,2}(2)$)=1. These parameters satisfy the necessary condition of the theorem.

Therefore, for even for this simple system, the necessary and sufficient condition of Theorem~\ref{thm:LTI:broadcast} do not match.
Finding the tight characterization for the independent stabilizability will be an interesting further research direction.

\subsection{Multiple-Unicast}
\label{sec:stablilizationoverLTI:multiple}
\begin{figure}
\includegraphics[width = 5in]{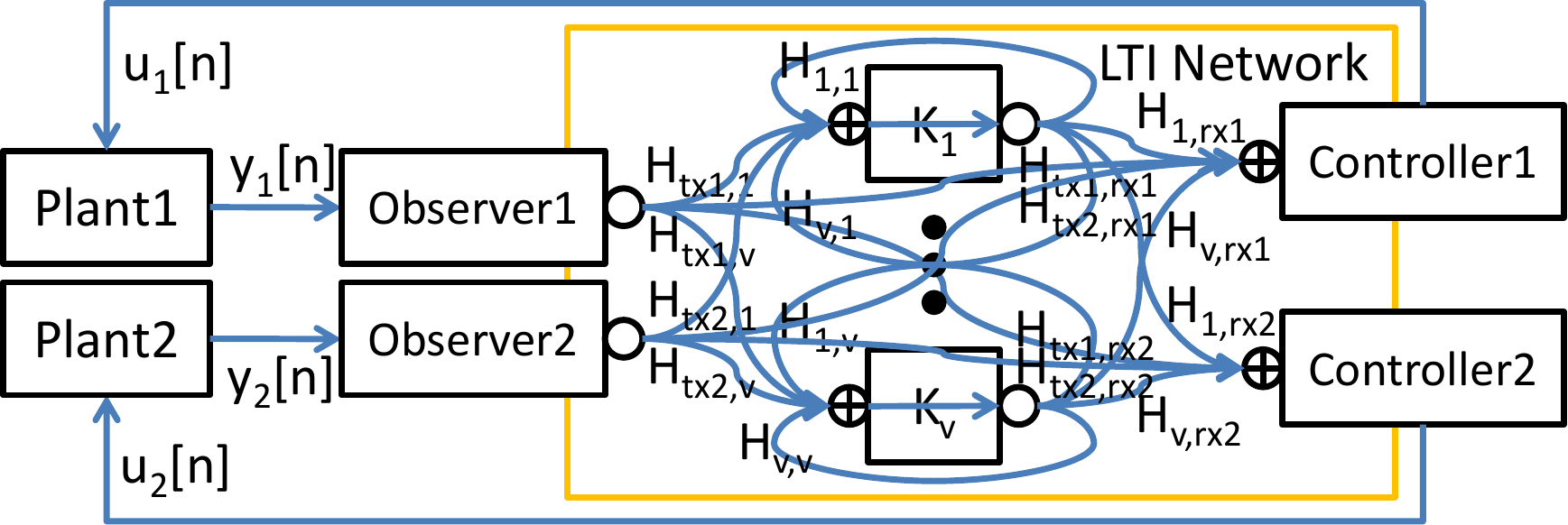}
\caption{Control over LTI Networks with multiple plants, multiple observers, and multiple controllers: Multiple-unicast case}
\label{fig:unicast}
\end{figure}
Multiple-unicast problems in network coding have multiple transmitter-receiver pairs which try to communicate their own individual messages. Unlike the previous problems, each transmitter only knows its own messages, and it is well-known that the cutset bound is not tight and the capacity region is open except several known cases~\cite{wang2010pairwise,wang2011two}.

Here, we try to convert multiple-unicast problem to the control over LTI network problems. The main difference between multiple-unicast and broadcast problems is the multiple transmitters. To capture this, we will introduce multiple observers\footnote{In section~\ref{sec:example}, we argued that the sources of the information flows for control are unstable states. However, when only explicit observers can directly observe the unstable states, the observers can be thought of as the sources of information.} to the previous control over LTI network problems. 

Figure~\ref{fig:unicast} shows the resulting problem. The only difference compared with  Figure~\ref{fig:multicast} is the multiple observers which do not share their observations directly. In this problem, we can easily prove that if there exists multiple unicast communication scheme from the observers to the controllers which accommodates enough information flow to stabilize the plants, we can independently stabilize the system.

\section{Conclusion}
In this paper, we take a unified approach to network coding and decentralized control by considering both problems as linear time-invariant systems. LTI-stabilizability of decentralized linear systems is
found to be equivalent to having sufficient capacity in the relevant LTI
networks. This equivalence can be exploited in both network coding and decentralized context.

In network coding, we found network linearization by introducing internal states and circulation arcs. The linearized network has not only equivalent mincut and maxflow to the original network, but also a simple topology, acyclic single-hop relay. These properties lead to a simple and elegant proof of an algebraic mincut-maxflow theorem.

In decentralized control, we gave an algorithm to make explicit communication networks
that represent the implicit communication required to stabilize the
plant. The stabilizability condition of decentralized systems is then
easily interpreted using mincut conditions on the corresponding
networks. Each eigenvalue is viewed separately, and the number of
Jordan blocks corresponding to that eigenvalue corresponds to the
number of degrees-of-freedom of implicit communication required to
stabilize that eigenvalue. The algebraic condition for fixed modes
that was reported in \cite{Anderson_Algebraic} and had, in our
opinion, remained mysterious for 30 years turns out to be a special
case of the algebraic mincut-maxflow theorem. This also confirms that
LTI controllers in decentralized control systems implicitly
communicate via linear network coding.

The connection to network coding becomes even more clear when we consider stabilization problems
with an explicit communication network. By introducing the concepts of alternative stabilizability and independent stabilizability, we successfully convert the network coding results to the equivalent stabilizability results.

Taking a step back, the general idea of implicit communication
(signaling) between decentralized controllers and information flow in
decentralized systems has been recognized since Witsenhausen's
counterexample~\cite{Witsenhausen_Counterexample}. However, in
Witsenhausen's counterexample the need for communication
between controllers is justified by the suboptimality of linear
controllers, {\em i.e.}~if the decentralized controllers want to
communicate with each other for efficient control of the system, they
would do so using nonlinear controllers for
signaling~\cite{Witsenhausen_Separation,Ho_Teams,Pulkit_Witsen}. However,
we showed here that even if we restrict controllers to be linear
time-invariant, the controllers still can communicate via linear
network coding. To an extent, this paper does for implicit
communication what \cite{Sinopoli_Kalman,Elia_Remote} did vis-a-vis
\cite{Sahai_Anytime, Sahai_Thesis} for explicit communication --- it
finds a way to discuss the issue within a linear framework. In fact,
the existence of implicit communication between linear controllers in
decentralized systems has been conjectured for a long time
\cite{Anderson_Timevarying,Cormat_Decentralized,Anderson_Transfer,Yuksel_Decentralized}. In
a sense, we hope that this paper clarifies these discussions.

\bibliographystyle{plain}
\bibliography{IEEEabrv,seyongbib}

\appendix
\subsection{Network Linearization for General Information Flow}
\label{app:networklin}
In this section, we will extend the network linearization of Section~\ref{sec:networklin} to general information flow cases -- multicast, broadcast and multiple-unicast. The main idea for this generalization is the relationship between network linearization and control over LTI networks discussed in Section~\ref{sec:stablilizationoverLTI}. 
%

\subsubsection{Multicast}
\begin{figure}[top]
\includegraphics[width = 3in]{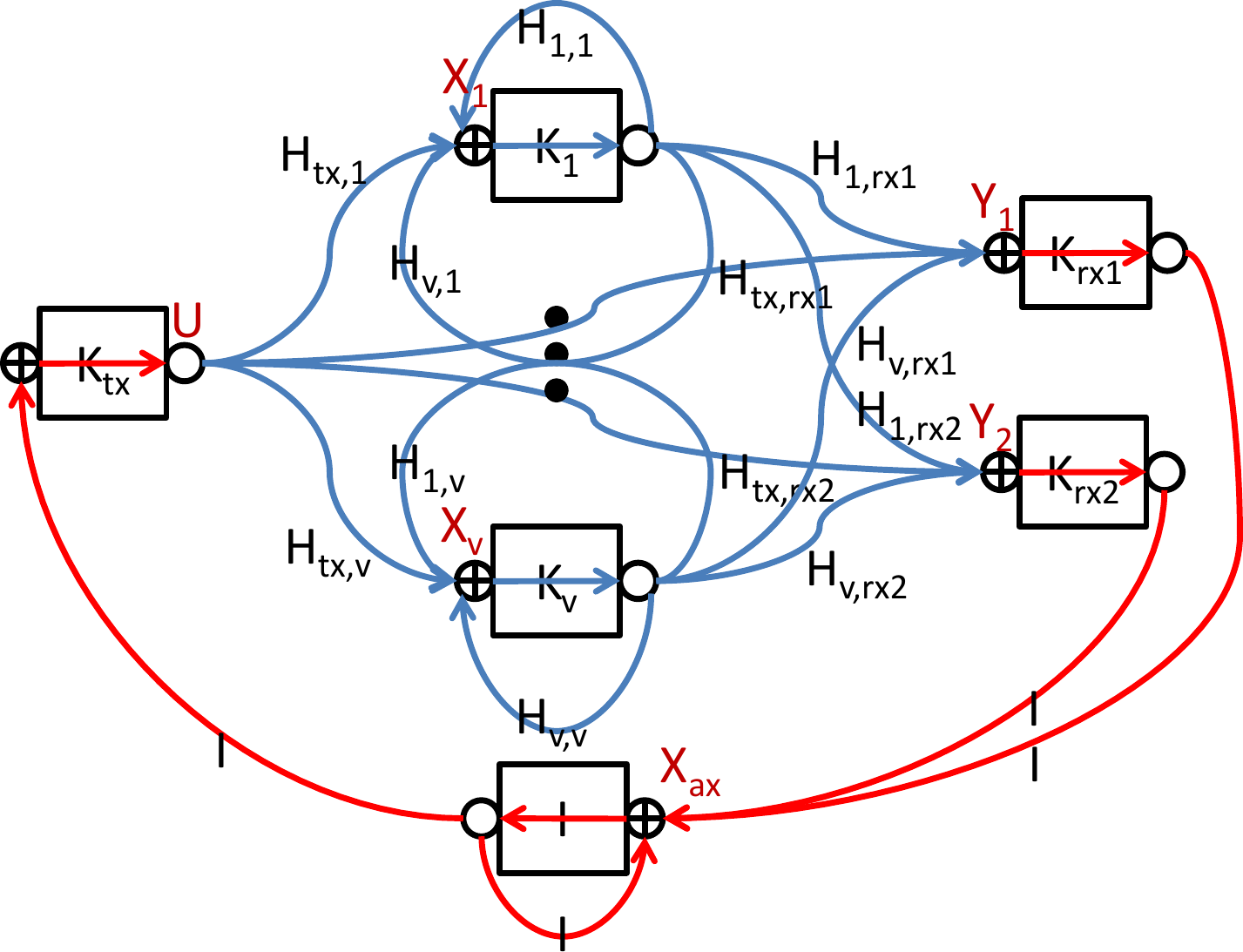}
\caption{Multicast LTI network $\mathcal{N}_{mul}(Z)$ with circulation arc added in}
\label{fig:LN_multicast}
\end{figure}
From the above discussion, we can expect that to linearize multicast problems, we have to introduce circulation arcs in a way that corresponds with Fig.~\ref{fig:LN_ptop}. Fig.~\ref{fig:LN_multicast} shows how the circulation arc has to be introduced. One circulation arc (which corresponds to an unstable plant as discussed in Section~\ref{sec:stablilizationoverLTI:ptop}) is connected to both receivers.

We will essentially use the same notation and assumptions as Section~\ref{sec:networklin}. Let the one-transmitter two-receiver LTI network of Fig.~\ref{fig:LN_multicast} without circulation arcs be $\mathcal{N}_{mul}(z)$. Denote the dimension of $Y_1$ as $d_{rx1}$ and $Y_2$ as $d_{rx2}$. Let the transfer function from the transmitter to the receiver $1$ of $\mathcal{N}_{mul}(z)$ be $G_{tx,rx1}(z,K)$, and the transfer function from the transmitter to the receiver $2$ be $G_{tx,rx2}(z,K)$. Here, the transfer function can be computed in the same way as Theorem~\ref{thm:transfer}.

Then, similar to Section~\ref{sec:networklin}, the following relation has to hold:
\begin{align}
&\begin{bmatrix}
X_{ax} \\
Y_1 \\
Y_2 \\
X_1 \\
\vdots \\
X_v
\end{bmatrix}
=
\begin{bmatrix}
I & K_{rx1} & K_{rx2} & 0 & \cdots & 0 \\
H_{tx,rx1}K_{tx} & 0 & 0 & H_{1,rx1} K_1 & \cdots & H_{v,rx1}K_v \\
H_{tx,rx2}K_{tx} & 0 & 0 & H_{1,rx2} K_1 & \cdots & H_{v,rx2}K_v \\
H_{tx,1} K_{tx} & 0 & 0 & H_{1,1} K_1 & \cdots & H_{v,1} K_v \\
\vdots & \vdots & \vdots & \vdots & \ddots & \vdots \\
H_{tx,v}K_{tx} & 0 & 0 & H_{1,v}K_1 & \cdots & H_{v,v}K_v
\end{bmatrix}
\begin{bmatrix}
X_{ax} \\
Y_1 \\
Y_2 \\
X_1 \\
\vdots \\
X_v
\end{bmatrix} \nonumber \\
&
(\Leftrightarrow)
\underbrace{\begin{bmatrix}
0 & -K_{rx1} & -K_{rx2} & 0 & \cdots & 0 \\
-H_{tx,rx1}K_{tx} & I & 0 & -H_{1,rx1} K_1 & \cdots & -H_{v,rx1}K_v \\
-H_{tx,rx2}K_{tx} & 0 & I & -H_{1,rx2} K_1 & \cdots & -H_{v,rx2}K_v \\
-H_{tx,1} K_{tx} & 0 & 0 & I-H_{1,1} K_1 & \cdots & -H_{v,1} K_v \\
\vdots & \vdots & \vdots & \vdots & \ddots & \vdots \\
-H_{tx,v}K_{tx} & 0 & 0 & -H_{1,v}K_1 & \cdots & I-H_{v,v}K_v
\end{bmatrix}}
_{:=G_{lin}(z,K)}
\begin{bmatrix}
X_{ax} \\
Y_1 \\
Y_2 \\
X_1 \\
\vdots \\
X_v
\end{bmatrix}
=
\begin{bmatrix}
0 \\ 0 \\ 0 \\ 0 \\ \vdots \\  0
\end{bmatrix} \nonumber
\end{align}
Then, we have
\begin{align}
G_{lin}(z,K)&=
\underbrace{\begin{bmatrix}
0 & 0 & 0 & 0 & \cdots & 0 \\
0 & I & 0 & 0 & \cdots & 0 \\
0 & 0 & I & 0 & \cdots & 0 \\
0 & 0 & 0 & I & \cdots & 0 \\
\vdots & \vdots & \vdots & \ddots & \vdots \\
0 & 0 & 0 & 0 & \cdots & I \\
\end{bmatrix}}_{:=A}
+
\underbrace{
\begin{bmatrix}
0 \\ H_{tx,rx1} \\ H_{tx,rx2} \\ H_{tx,1} \\ \vdots \\ H_{tx,v}
\end{bmatrix}}_{:=B_{tx}}
K_{tx}
\underbrace{
\begin{bmatrix}
-I & 0 & 0 & 0 & \cdots & 0
\end{bmatrix}}_{:=C_{tx}} \\
&+
\underbrace{
\begin{bmatrix}
0 \\ H_{1,rx1} \\ H_{1,rx2} \\ H_{1,1} \\ \vdots \\ H_{1,v}
\end{bmatrix}}_{:=B_1}
K_1
\underbrace{
\begin{bmatrix}
0 & 0 & 0 & -I & \cdots & 0
\end{bmatrix}}_{:=C_1}+\cdots+
\underbrace{
\begin{bmatrix}
0 \\ H_{v,rx1} \\ H_{v,rx2} \\ H_{v,1} \\ \vdots \\ H_{v,v}
\end{bmatrix}}_{:=B_v}
K_v
\underbrace{
\begin{bmatrix}
0 & 0 & 0 & 0 & \cdots & -I
\end{bmatrix}}_{:=C_v}\\
&+
\underbrace{
\begin{bmatrix}
I \\ 0 \\ 0 \\ 0 \\ \vdots \\ 0
\end{bmatrix}}_{:=B_{rx1}}
K_{rx1}
\underbrace{
\begin{bmatrix}
0 & -I & 0 & 0 & \cdots & 0
\end{bmatrix}}_{:=C_{rx1}}
+
\underbrace{
\begin{bmatrix}
I \\ 0 \\ 0 \\ \vdots \\ 0
\end{bmatrix}}_{:=B_{rx2}}
K_{rx2}
\underbrace{
\begin{bmatrix}
0 & 0 & -I & 0 & \cdots & 0
\end{bmatrix}}_{:=C_{rx2}}
\end{align}
Let
\begin{align}
&G_{tx',rx1'}(z,K):=A+\sum_{1 \leq i \leq v} B_i K_i C_i + B_{tx}K_{tx}C_{tx} + B_{rx1}K_{rx1}C_{rx1} \\
&G_{tx',rx2'}(z,K):=A+\sum_{1 \leq i \leq v} B_i K_i C_i + B_{tx}K_{tx}C_{tx} + B_{rx2}K_{rx2}C_{rx2}
\end{align}
and $d:=\dim \begin{bmatrix} Y_1 \\ Y_2 \\ X_1 \\ \vdots \\ X_v \end{bmatrix}$.
Let $\mathcal{N}_{lin}^{mul}(z)$ be the network shown in Fig.~\ref{fig:LN_multicastlin}.
Then, we can easily see $G_{tx',rx1'}(z,K)$ is the transfer function from $tx'$ to $rx_1'$ of $\mathcal{N}_{mul}^{lin}(z)$, and $G_{tx',rx2'}(z,K)$ is the transfer function from $tx'$ to $rx_2'$ of $\mathcal{N}_{mul}^{lin}(z)$.
\begin{figure}[top]
\includegraphics[width = 3in]{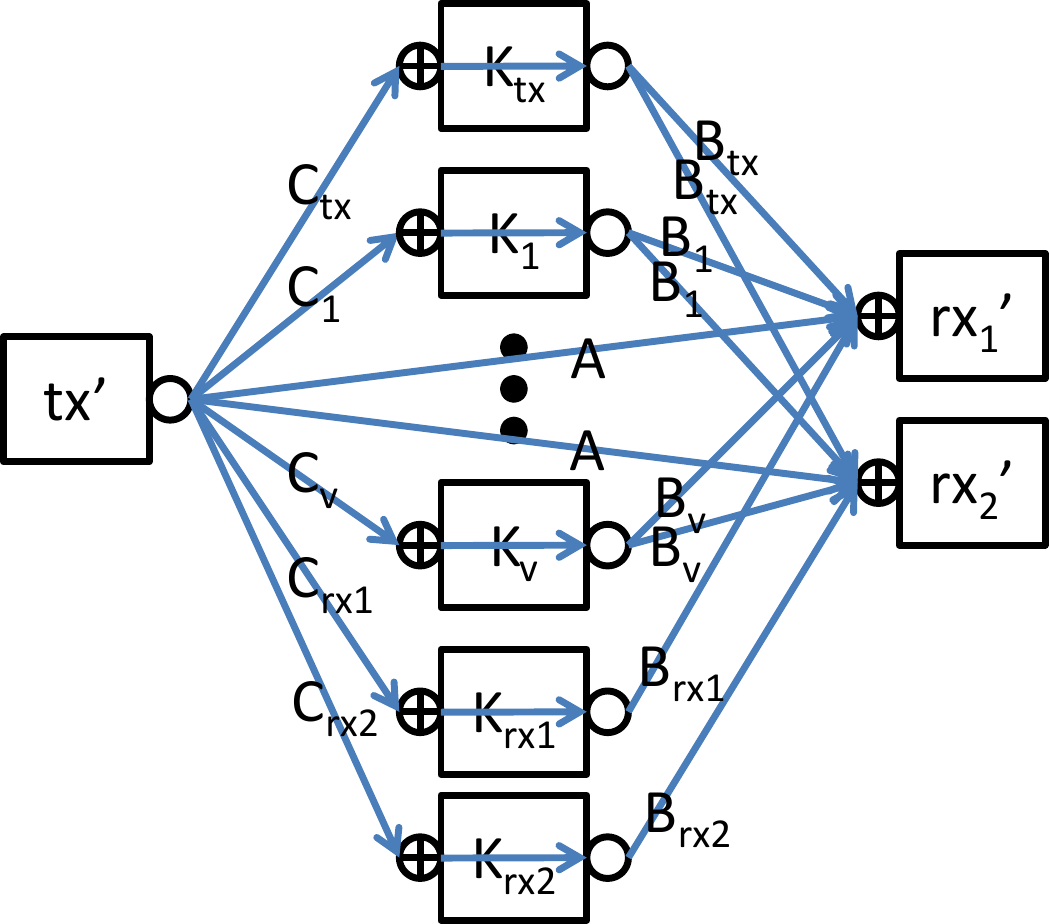}
\caption{Linearized LTI network of Multicast problem, $\mathcal{N}_{mul}^{lin}(z)$}
\label{fig:LN_multicastlin}
\end{figure}

Then, like Section~\ref{sec:networklin} we can show the equivalence between $\mathcal{N}_{mul}(z)$ and $\mathcal{N}_{mul}^{lin}(z)$.
\begin{theorem}
Let $K_{tx} \in \mathbb{F}[z]^{d_{tx}\times d_{ax}}$, $K_i \in \mathbb{F}[z]^{d_{i,in} \times d_{i,out}}$, $K_{rx1} \in \mathbb{F}[z]^{d_{ax} \times d_{rx1}}$ and $K_{rx2} \in \mathbb{F}[z]^{d_{ax} \times d_{rx2}}$. We also assume that
\begin{align}
\begin{bmatrix}
&I-H_{1,1}K_1 & \cdots & -H_{v,1} K_v \\
& \vdots & \ddots & \vdots \\
& -H_{1,v}K_1 & \cdots & I - H_{v,v} K_v
\end{bmatrix} \mbox{ is invertible.}
\end{align}
Then, for all $d_1, d_2 \in \mathbb{Z}^+$
\begin{align}
&(i) \rank (K_{rx1}(z) G_{tx,rx1}(z,K(z)) K_{tx}(z)) \geq d_1 \\
&(ii) \rank (K_{rx2}(z) G_{tx,rx2}(z,K(z)) K_{tx}(z)) \geq d_2
\end{align}
if and only if
\begin{align}
&(a) \rank G_{tx',rx1'}(z,K(z)) \geq d+d_1\\
&(b) \rank G_{tx',rx2'}(z,K(z)) \geq d+d_2
\end{align}
\label{thm:LIN:mul}
\end{theorem}
\begin{proof}
Similar to Lemma~\ref{lem:LIN:maxflow}.
\end{proof}

Remark 1. The result of this theorem can be easily generalized to multiple receivers, which we omit for simplicity.

Remark 2. To apply this theorem to multicast problems and send a message with rate $r$, we can simply put $d_1=d_2=r$. Moreover, just as we did in Figure~\ref{fig:networkunfoldinglin2}, the condition that $\begin{bmatrix}
&I-H_{1,1}K_1 & \cdots & -H_{v,1} K_v \\
& \vdots & \ddots & \vdots \\
& -H_{1,v}K_1 & \cdots & I - H_{v,v} K_v
\end{bmatrix}$ is invertible can be included as a part of communication problem by introducing an additional receiver. Following the similar procedure of Section~\ref{sec:linvsun}, we can design an LTI multicast scheme.


Remark 3. Fig.~\ref{fig:butterfly} and Fig.~\ref{fig:butterfly_lin} shows the famous butterfly example in network coding~\cite{Ahlswede_Network} and its corresponding linearized network. Here, we can see the linearized network has more input and output vertices, but is \textbf{topologically} simpler --- a single-hop multicast network.

\begin{figure}[top]
\includegraphics[width = 1.5in]{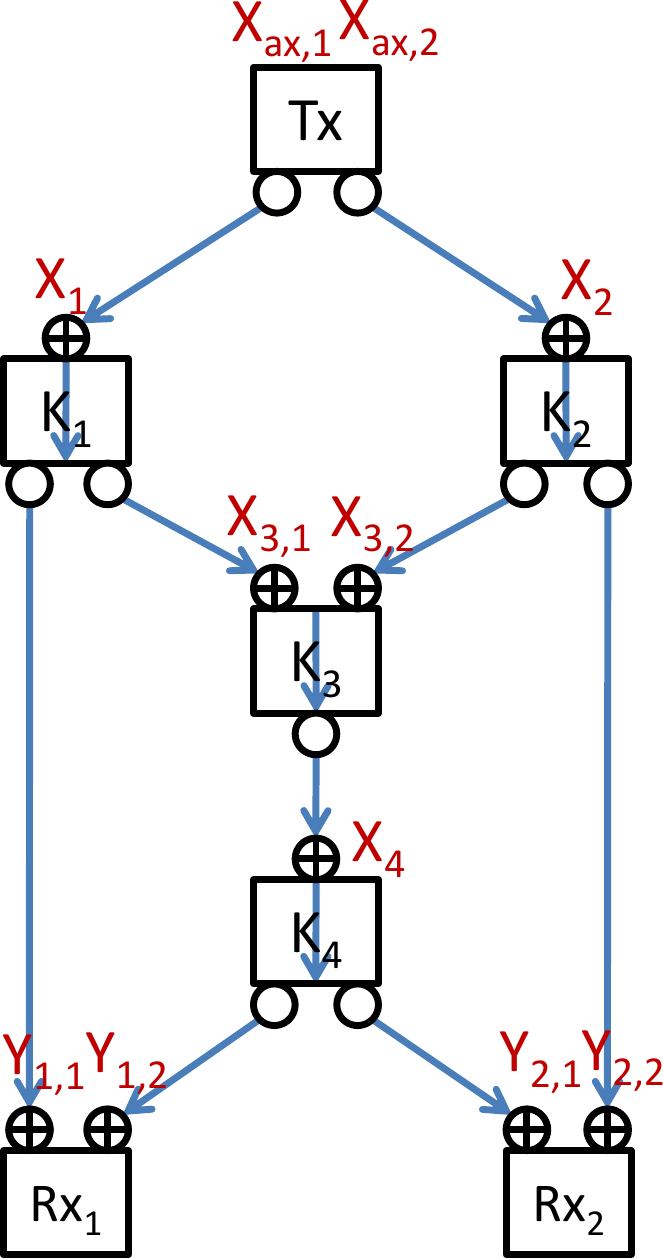}
\caption{Butterfly Example for Multicast. The gains of all edges are $1$.}
\label{fig:butterfly}
\end{figure}
\begin{figure}[top]
\includegraphics[width = 2in]{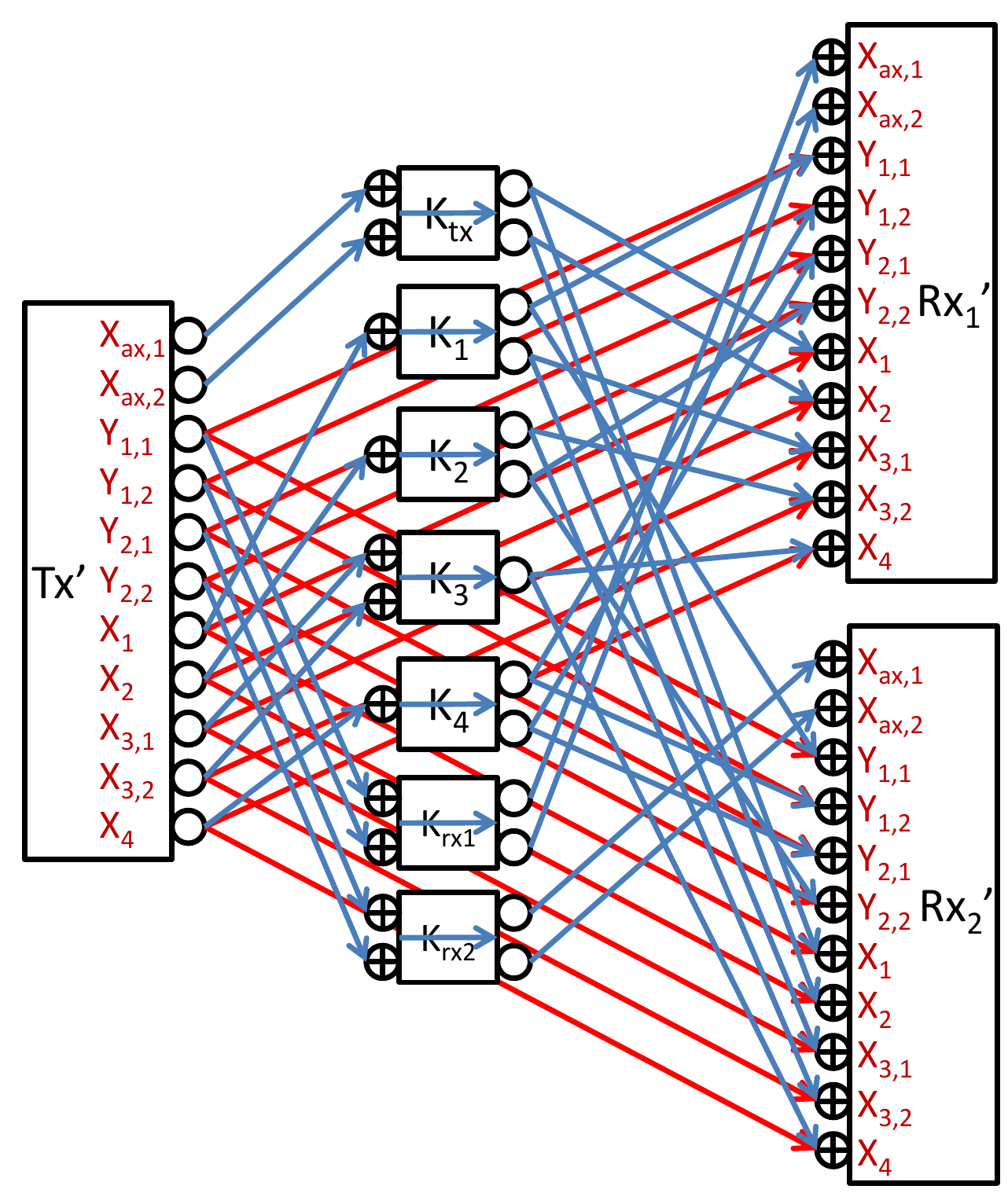}
\caption{Linearized Network for Butterfly Example of Fig.~\ref{fig:butterfly}. The gain of each edge from $Tx'$ to $K_{tx}$, $K_i$, $K_{rx1}$, $K_{rx2}$ is $-1$, and the gains for the other edges are all $1$.}
\label{fig:butterfly_lin}
\end{figure}

\subsubsection{Broadcast}
\begin{figure}[top]
\includegraphics[width = 3in]{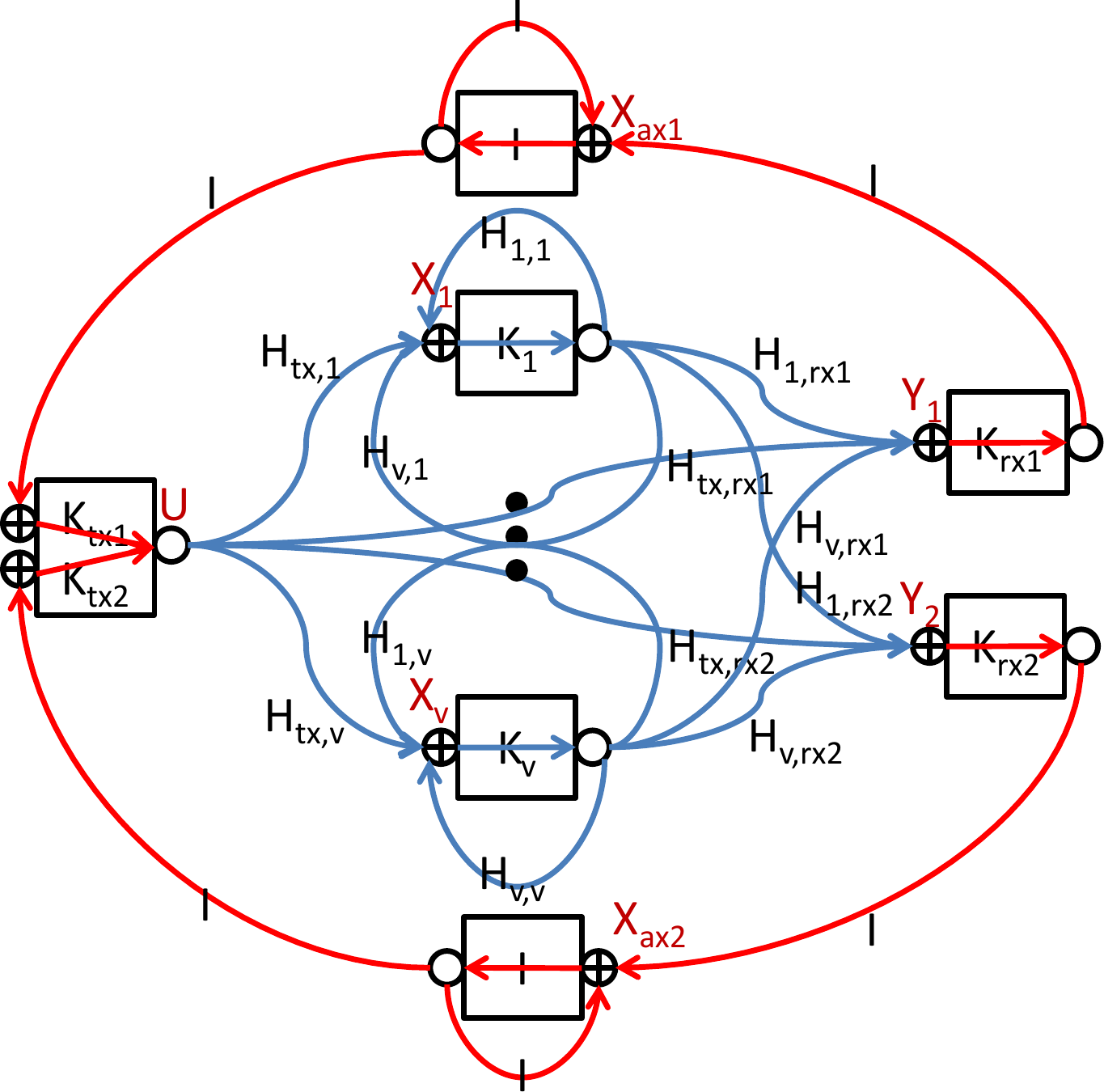}
\caption{Broadcast LTI network $\mathcal{N}_{br}(z)$ with circulation arcs added in}
\label{fig:LN_broadcast}
\end{figure}
Inspired by Figure~\ref{fig:broadcast}, we introduce circulation arcs as shown in Figure~\ref{fig:LN_broadcast} to linearize broadcast problems. We introduce two circulation arcs which correspond to two unstable plants of Figure~\ref{fig:broadcast}, and the two circulation arcs are connected to different receivers as two plants are controlled by different controllers in Figure~\ref{fig:broadcast}.

We basically use the same notations and assumptions of the previous section. Let the one-transmitter two-receiver LTI network of Fig.~\ref{fig:LN_broadcast} without circulation arcs be $\mathcal{N}_{br}(z)$. Denote the dimension of $X_{ax1}$ as $d_{ax1}$ and $X_{ax2}$ as $d_{ax2}$. Then, as we can see from the figure, $K_{tx1}$ is a $d_{tx} \times d_{ax1}$ matrix and $K_{tx2}$ is a $d_{tx} \times d_{ax2}$ matrix. Let the transfer function from the transmitter to the receiver $1$ of $\mathcal{N}_{mul}(z)$ be $G_{tx,rx1}(z,K)$, and the transfer function from the transmitter to the receiver $2$ be $G_{tx,rx2}(z,K)$.

Then, the following relation has to hold:
\begin{align}
&\begin{bmatrix}
X_{ax1} \\
X_{ax2} \\
Y_1 \\
Y_2 \\
X_1 \\
\vdots \\
X_v
\end{bmatrix}
=
\begin{bmatrix}
I & 0 & K_{rx1} & 0 & 0 & \cdots & 0 \\
0 & I & 0 & K_{rx2} & 0 & \cdots & 0 \\
H_{tx,rx1}K_{tx1} & H_{tx,rx1}K_{tx2} &  0 & 0 & H_{1,rx1} K_1 & \cdots & H_{v,rx1}K_v \\
H_{tx,rx2}K_{tx1} & H_{tx,rx2}K_{tx2} &  0 & 0 & H_{1,rx2} K_1 & \cdots & H_{v,rx2}K_v \\
H_{tx,1} K_{tx1} & H_{tx,1} K_{tx2} & 0 & 0 & H_{1,1} K_1 & \cdots & H_{v,1} K_v \\
\vdots & \vdots & \vdots & \vdots & \vdots & \ddots & \vdots \\
H_{tx,v} K_{tx1} & H_{tx,v} K_{tx2} & 0 & 0 & H_{1,v}K_1 & \cdots & H_{v,v}K_v
\end{bmatrix}
\begin{bmatrix}
X_{ax1} \\
X_{ax2} \\
Y_1 \\
Y_2 \\
X_1 \\
\vdots \\
X_v
\end{bmatrix} \nonumber \\
&
(\Leftrightarrow)
\underbrace{
\begin{bmatrix}
0 & 0 & -K_{rx1} & 0 & 0 & \cdots & 0 \\
0 & 0 & 0 & -K_{rx2} & 0 & \cdots & 0 \\
-H_{tx,rx1}K_{tx1} & -H_{tx,rx1}K_{tx2} &  I & 0 & -H_{1,rx1} K_1 & \cdots & -H_{v,rx1}K_v \\
-H_{tx,rx2}K_{tx1} & -H_{tx,rx2}K_{tx2} &  0 & I & -H_{1,rx2} K_1 & \cdots & -H_{v,rx2}K_v \\
-H_{tx,1} K_{tx1} & -H_{tx,1} K_{tx2} & 0 & 0 & I-H_{1,1} K_1 & \cdots & -H_{v,1} K_v \\
\vdots & \vdots & \vdots & \vdots & \vdots & \ddots & \vdots \\
-H_{tx,v} K_{tx1} & -H_{tx,v} K_{tx2} & 0 & 0 & -H_{1,v}K_1 & \cdots & I-H_{v,v}K_v
\end{bmatrix}
}
_{:=G_{br}^{lin}(z,K)}
\begin{bmatrix}
X_{ax1} \\
X_{ax2} \\
Y_1 \\
Y_2 \\
X_1 \\
\vdots \\
X_v
\end{bmatrix}
=
\begin{bmatrix}
0 \\ 0 \\ 0 \\ 0 \\ 0 \\ \vdots \\  0
\end{bmatrix} \nonumber
\end{align}
Thus, we have
\begin{align}
G_{br}^{lin}(z,K)&=
\underbrace{\begin{bmatrix}
0 & 0 & 0 & 0 & 0 & \cdots & 0 \\
0 & 0 & 0 & 0 & 0 & \cdots & 0 \\
0 & 0 & I & 0 & 0 & \cdots & 0 \\
0 & 0 & 0 & I & 0 & \cdots & 0 \\
0 & 0 & 0 & 0 & I & \cdots & 0 \\
\vdots & \vdots & \vdots & \vdots & \vdots & \ddots & \vdots \\
0 & 0 & 0 & 0 & 0 & \cdots & I \\
\end{bmatrix}}_{:=A}
+
\underbrace{
\begin{bmatrix}
0 \\ 0 \\ H_{tx,rx1} \\ H_{tx,rx2} \\ H_{tx,1} \\ \vdots \\ H_{tx,v}
\end{bmatrix}}_{:=B_{tx1}}
K_{tx1}
\underbrace{
\begin{bmatrix}
-I & 0 & 0 & 0 & 0 & \cdots & 0
\end{bmatrix}}_{:=C_{tx1}} \\
&
+
\underbrace{
\begin{bmatrix}
0 \\ 0 \\ H_{tx,rx1} \\ H_{tx,rx2} \\ H_{tx,1} \\ \vdots \\ H_{tx,v}
\end{bmatrix}}_{:=B_{tx2}}
K_{tx2}
\underbrace{
\begin{bmatrix}
0 & -I & 0 & 0 & 0 & \cdots & 0
\end{bmatrix}}_{:=C_{tx2}} \\
&+
\underbrace{
\begin{bmatrix}
0 \\ 0 \\ H_{1,rx1} \\ H_{1,rx2} \\ H_{1,1} \\ \vdots \\ H_{1,v}
\end{bmatrix}}_{:=B_1}
K_1
\underbrace{
\begin{bmatrix}
0 & 0 & 0 & 0 & -I & \cdots & 0
\end{bmatrix}}_{:=C_1}+\cdots+
\underbrace{
\begin{bmatrix}
0 \\ 0 \\ H_{v,rx1} \\ H_{v,rx2} \\ H_{v,1} \\ \vdots \\ H_{v,v}
\end{bmatrix}}_{:=B_v}
K_v
\underbrace{
\begin{bmatrix}
0 & 0 & 0 & 0 & 0 & \cdots & -I
\end{bmatrix}}_{:=C_v}\\
&+
\underbrace{
\begin{bmatrix}
I \\ 0 \\ 0 \\ 0 \\ 0 \\ \vdots \\ 0
\end{bmatrix}}_{:=B_{rx1}}
K_{rx1}
\underbrace{
\begin{bmatrix}
0 & 0 & -I & 0 & 0 & \cdots & 0
\end{bmatrix}}_{:=C_{rx1}}
+
\underbrace{
\begin{bmatrix}
0 \\ I \\ 0 \\ 0 \\ 0 \\ \vdots \\ 0
\end{bmatrix}}_{:=B_{rx2}}
K_{rx2}
\underbrace{
\begin{bmatrix}
0 & 0 & 0 & -I & 0 & \cdots & 0
\end{bmatrix}}_{:=C_{rx2}}
\end{align}

Let
\begin{align}
&G_{tx',rx11'}(z,K):=A+\sum_{1 \leq i \leq v} B_i K_i C_i + B_{tx1}K_{tx1}C_{tx1} + B_{rx1}K_{rx1}C_{rx1} \label{eqn:app:br1}\\
&G_{tx',rx22'}(z,K):=A+\sum_{1 \leq i \leq v} B_i K_i C_i + B_{tx2}K_{tx2}C_{tx2} + B_{rx2}K_{rx2}C_{rx2} \\
&G_{tx',rx12'}(z,K):=A+\sum_{1 \leq i \leq v} B_i K_i C_i + B_{tx2}K_{tx2}C_{tx2} + B_{rx1}K_{rx1}C_{rx1} \\
&G_{tx',rx12'}(z,K):=A+\sum_{1 \leq i \leq v} B_i K_i C_i + B_{tx1}K_{tx1}C_{tx1} + B_{rx2}K_{rx2}C_{rx2}
\end{align}
Let $\mathcal{N}_{lin}^{br}(z)$ be the network shown in Fig.~\ref{fig:LN_broadcastlin}.
Then, we can easily see $G_{tx',rx11'}(z,K),\cdots,G_{tx',rx12'}(z,K)$ corresponds to the transfer function from $tx'$ to $rx_{11}',\cdots,rx_{12}'$ of $\mathcal{N}_{lin}^{br}(z)$ respectively.

\begin{figure}[top]
\includegraphics[width = 3in]{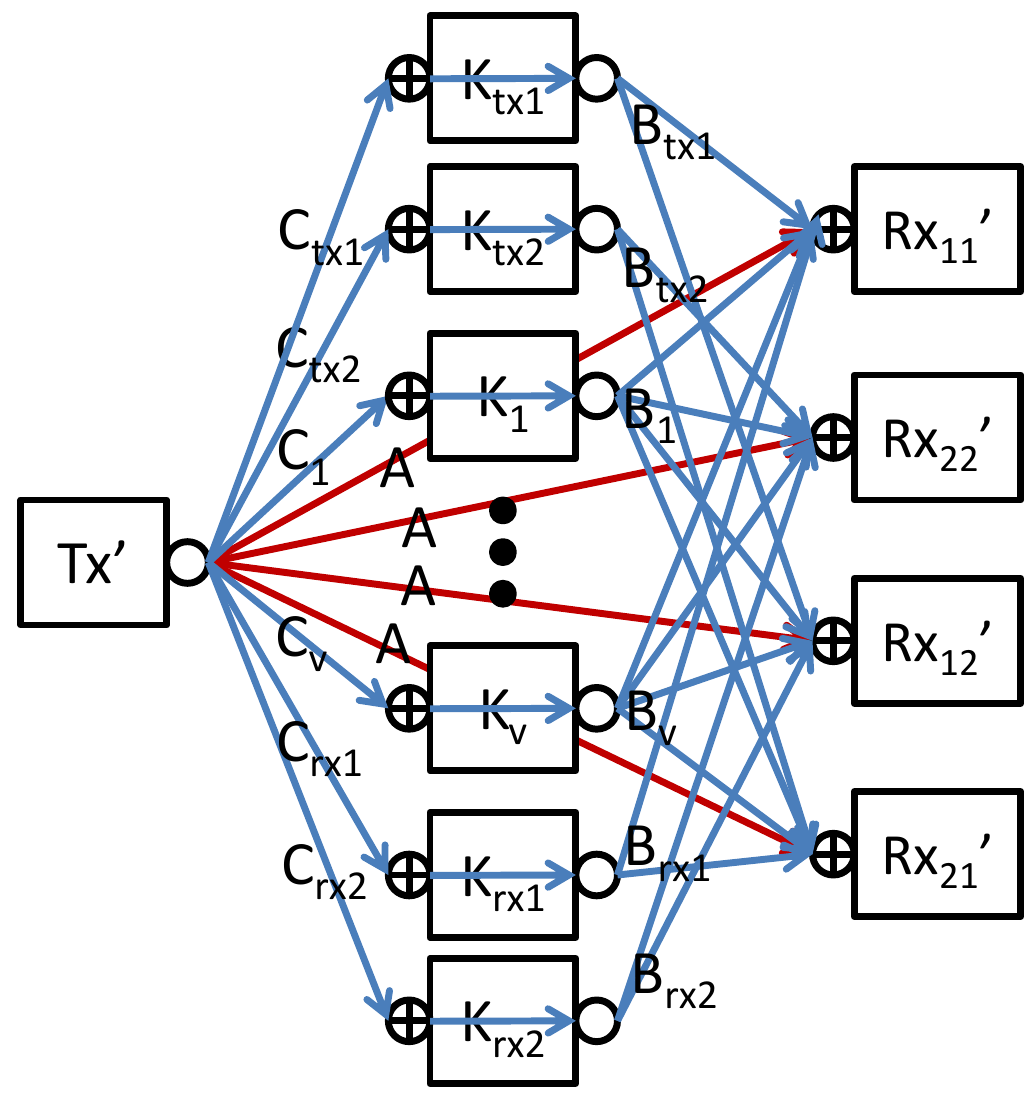}
\caption{Linearized LTI network of a Broadcast problem, $\mathcal{N}_{br}^{lin}(z)$}
\label{fig:LN_broadcastlin}
\end{figure}
Then, the relationship between $\mathcal{N}_{br}(z)$ and $\mathcal{N}_{br}^{lin}(z)$ is given as follows.
\begin{theorem}
Let $K_{tx1}(z) \in \mathbb{F}[z]^{d_{tx}\times d_{ax1}}$, $K_{tx2}(z) \in \mathbb{F}[z]^{d_{tx}\times d_{ax2}}$, $K_i(z) \in \mathbb{F}[z]^{d_{i,in} \times d_{i,out}}$, $K_{rx1}(z) \in \mathbb{F}[z]^{d_{ax1} \times d_{rx1}}$ and $K_{rx2}(z) \in \mathbb{F}[z]^{d_{ax2} \times d_{rx2}}$. We also assume that
\begin{align}
\begin{bmatrix}
&I-H_{1,1}K_1 & \cdots & -H_{v,1} K_v \\
& \vdots & \ddots & \vdots \\
& -H_{1,v}K_1 & \cdots & I - H_{v,v} K_v
\end{bmatrix} \mbox{ is invertible.}
\end{align}
Then, for all $d_1, d_2, d_3, d_4 \in \mathbb{Z}^+$, the following two conditions are equivalent.
\begin{align}
&(i) \rank K_{rx1}(z) G_{tx,rx1}(z,K(z)) K_{tx1}(z) \geq d_1 \\
&(ii) \rank K_{rx2}(z) G_{tx,rx2}(z,K(z)) K_{tx2}(z) \geq d_2 \\
&(iii) \rank K_{rx2}(z) G_{tx,rx2}(z,K(z)) K_{tx1}(z) \leq d_3 \\
&(iv) \rank K_{rx1}(z) G_{tx,rx1}(z,K(z)) K_{tx2}(z) \leq d_4
\end{align}
if and only if
\begin{align}
&(a) \rank G_{tx',rx11'}(z,K(z)) \geq d+d_1 \\
&(b) \rank G_{tx',rx22'}(z,K(z)) \geq d+d_2 \\
&(c) \rank G_{tx',rx12'}(z,K(z)) \leq d+d_3\\
&(d) \rank G_{tx',rx21'}(z,K(z)) \leq d+d_4
\end{align}
\label{thm:LIN:br}
\end{theorem}
\begin{proof}
Similar to Lemma~\ref{lem:LIN:maxflow}.
\end{proof}

Remark 1. The result of this theorem can be easily generalized to multiple receivers. In three receiver case, we will see 9 conditions, and for general $n$ receiver case, we will see $n^2$ conditions.

Remark 2. To design a broadcast scheme which communicates a message with rate $r_1$ to receiver $1$ and at the same time another message with rate $r_2$ to receiver $2$, we can choose the problem parameters as $d_1=r_1, d_2=r_2, d_3=0, d_4=0$. Any scheme which satisfies the condition $(a)-(d)$, and the existence condition of transfer functions can be immediately applied to the original problem and give a broadcast communication scheme.

Remark 3. The linearized network of Figure~\ref{fig:LN_broadcastlin} can be understood as a two-receiver and two-eavesdropper secrecy problem. The receivers $rx11'$ and $rx22'$ want to receive $d+d_1$ and $d+d_2$ dimensional information about the messages (possibly, common) respectively. While at the same time, we do not want to give more than $d+d_3$ and $d+d_4$ dimensions about the message to the eavesdroppers $rx12'$ and $rx21'$.

The receivers $rx11'$ and $rx22'$ in the linearized network reflect that the desired messages have to be received in the original problem. The eavesdropper $rx12'$ and $rx21'$ in the linearized network reflects that the undesired messages has to be removable in the original problem.

\subsubsection{Multiple-Unicast}
\begin{figure}[top]
\includegraphics[width = 3in]{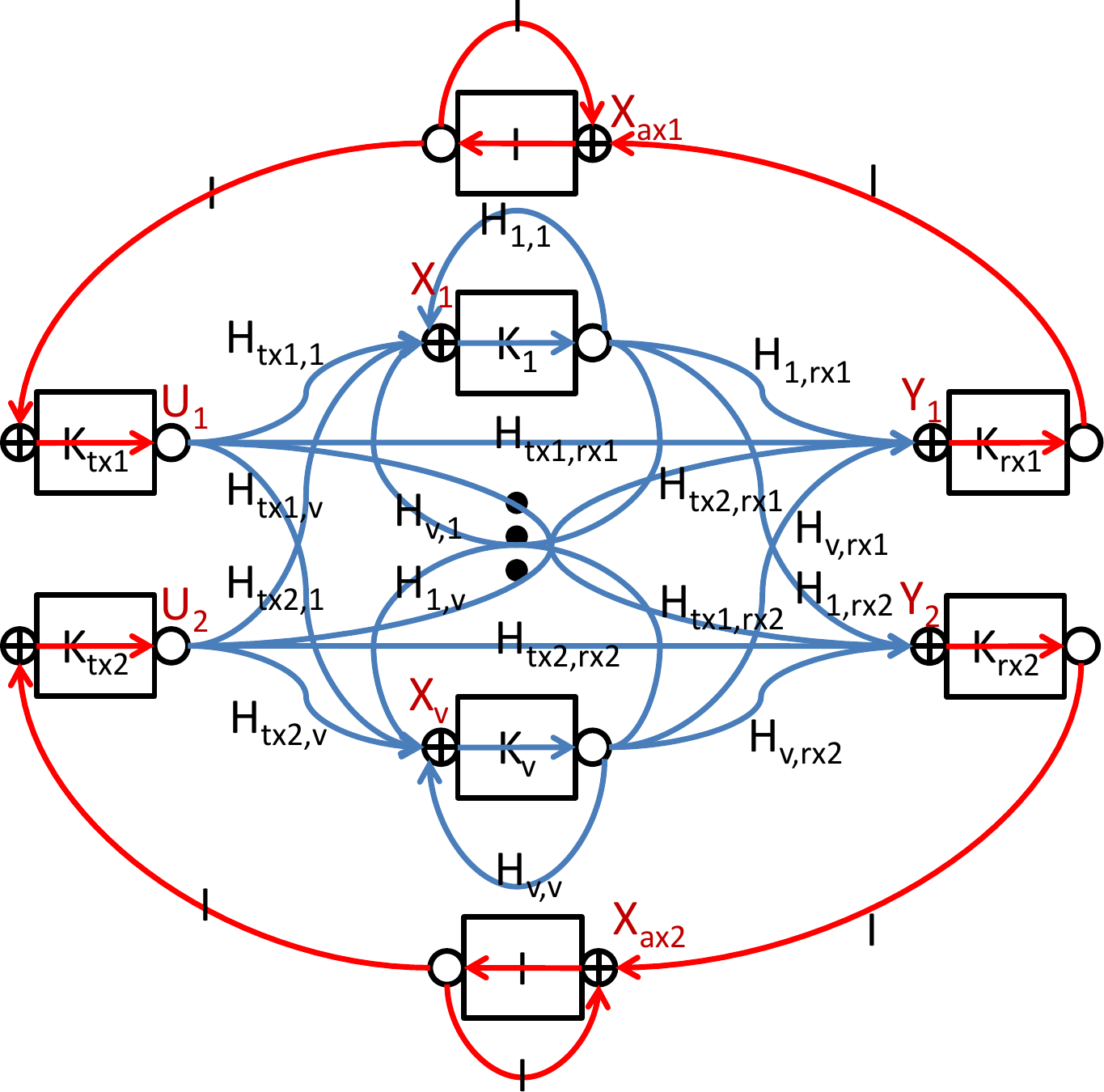}
\caption{Multiple Unicast LTI network $\mathcal{N}_{uni}(z)$ with circulation arc added in}
\label{fig:LN_unicast}
\end{figure}

As the only difference between Figure~\ref{fig:broadcast} and Figure~\ref{fig:unicast} is the observers, we introduce circulation arcs in the same way as the broadcast problems in Figure~\ref{fig:LN_broadcast}. Fig.~\ref{fig:LN_unicast} shows the multiple-unicast LTI network $\mathcal{N}_{uni}(z)$ with the circulation arcs. 

We essentially repeat the previous argument. Let's use the same notations and assumptions of the previous section. Denote the dimension of $U_1, U_2, Y_1, Y_2$ as $d_{tx1}, d_{tx2}, d_{rx1}, d_{rx2}$ respectively. The transfer functions between the transmitters and the receivers are denoted as $G_{tx1,rx1}(z,K)$, $G_{tx1,rx2}(z,K)$, $G_{tx2,rx1}(z,K)$, $G_{tx2,rx2}(z,K)$.

Then, we have the following relationship.
\begin{align}
&\begin{bmatrix}
X_{ax1} \\
X_{ax2} \\
Y_1 \\
Y_2 \\
X_1 \\
\vdots \\
X_v
\end{bmatrix}
=
\begin{bmatrix}
I & 0 & K_{rx1} & 0 & 0 & \cdots & 0 \\
0 & I & 0 & K_{rx2} & 0 & \cdots & 0 \\
H_{tx1,rx1}K_{tx1} & H_{tx2,rx1}K_{tx2} &  0 & 0 & H_{1,rx1} K_1 & \cdots & H_{v,rx1}K_v \\
H_{tx1,rx2}K_{tx1} & H_{tx2,rx2}K_{tx2} &  0 & 0 & H_{1,rx2} K_1 & \cdots & H_{v,rx2}K_v \\
H_{tx1,1} K_{tx1} & H_{tx2,1} K_{tx2} & 0 & 0 & H_{1,1} K_1 & \cdots & H_{v,1} K_v \\
\vdots & \vdots & \vdots & \vdots & \vdots & \ddots & \vdots \\
H_{tx1,v} K_{tx1} & H_{tx2,v} K_{tx2} & 0 & 0 & H_{1,v}K_1 & \cdots & H_{v,v}K_v
\end{bmatrix}
\begin{bmatrix}
X_{ax1} \\
X_{ax2} \\
Y_1 \\
Y_2 \\
X_1 \\
\vdots \\
X_v
\end{bmatrix} \nonumber \\
&
(\Leftrightarrow)
\underbrace{
\begin{bmatrix}
0 & 0 & -K_{rx1} & 0 & 0 & \cdots & 0 \\
0 & 0 & 0 & -K_{rx2} & 0 & \cdots & 0 \\
-H_{tx1,rx1}K_{tx1} & -H_{tx2,rx1}K_{tx2} &  I & 0 & -H_{1,rx1} K_1 & \cdots & -H_{v,rx1}K_v \\
-H_{tx1,rx2}K_{tx1} & -H_{tx2,rx2}K_{tx2} &  0 & I & -H_{1,rx2} K_1 & \cdots & -H_{v,rx2}K_v \\
-H_{tx1,1} K_{tx1} & -H_{tx2,1} K_{tx2} & 0 & 0 & I-H_{1,1} K_1 & \cdots & -H_{v,1} K_v \\
\vdots & \vdots & \vdots & \vdots & \vdots & \ddots & \vdots \\
-H_{tx1,v} K_{tx1} & -H_{tx2,v} K_{tx2} & 0 & 0 & -H_{1,v}K_1 & \cdots & I-H_{v,v}K_v
\end{bmatrix}
}
_{:=G_{uni}^{lin}(z,K)}
\begin{bmatrix}
X_{ax1} \\
X_{ax2} \\
Y_1 \\
Y_2 \\
X_1 \\
\vdots \\
X_v
\end{bmatrix}
=
\begin{bmatrix}
0 \\ 0 \\ 0 \\ 0 \\ 0 \\ \vdots \\  0
\end{bmatrix} \nonumber
\end{align}

Therefore, we have
\begin{align}
G_{uni}^{lin}(z,K)&=
\underbrace{\begin{bmatrix}
0 & 0 & 0 & 0 & 0 & \cdots & 0 \\
0 & 0 & 0 & 0 & 0 & \cdots & 0 \\
0 & 0 & I & 0 & 0 & \cdots & 0 \\
0 & 0 & 0 & I & 0 & \cdots & 0 \\
0 & 0 & 0 & 0 & I & \cdots & 0 \\
\vdots & \vdots & \vdots & \vdots & \vdots & \ddots & \vdots \\
0 & 0 & 0 & 0 & 0 & \cdots & I \\
\end{bmatrix}}_{:=A}
+
\underbrace{
\begin{bmatrix}
0 \\ 0 \\ H_{tx1,rx1} \\ H_{tx1,rx2} \\ H_{tx1,1} \\ \vdots \\ H_{tx1,v}
\end{bmatrix}}_{:=B_{tx1}}
K_{tx1}
\underbrace{
\begin{bmatrix}
-I & 0 & 0 & 0 & 0 & \cdots & 0
\end{bmatrix}}_{:=C_{tx1}} \\
&
+
\underbrace{
\begin{bmatrix}
0 \\ 0 \\ H_{tx2,rx1} \\ H_{tx2,rx2} \\ H_{tx2,1} \\ \vdots \\ H_{tx2,v}
\end{bmatrix}}_{:=B_{tx2}}
K_{tx2}
\underbrace{
\begin{bmatrix}
0 & -I & 0 & 0 & 0 & \cdots & 0
\end{bmatrix}}_{:=C_{tx2}} \\
&+
\underbrace{
\begin{bmatrix}
0 \\ 0 \\ H_{1,rx1} \\ H_{1,rx2} \\ H_{1,1} \\ \vdots \\ H_{1,v}
\end{bmatrix}}_{:=B_1}
K_1
\underbrace{
\begin{bmatrix}
0 & 0 & 0 & 0 & -I & \cdots & 0
\end{bmatrix}}_{:=C_1}+\cdots+
\underbrace{
\begin{bmatrix}
0 \\ 0 \\ H_{v,rx1} \\ H_{v,rx2} \\ H_{v,1} \\ \vdots \\ H_{v,v}
\end{bmatrix}}_{:=B_v}
K_v
\underbrace{
\begin{bmatrix}
0 & 0 & 0 & 0 & 0 & \cdots & -I
\end{bmatrix}}_{:=C_v}\\
&+
\underbrace{
\begin{bmatrix}
I \\ 0 \\ 0 \\ 0 \\ 0 \\ \vdots \\ 0
\end{bmatrix}}_{:=B_{rx1}}
K_{rx1}
\underbrace{
\begin{bmatrix}
0 & 0 & -I & 0 & 0 & \cdots & 0
\end{bmatrix}}_{:=C_{rx1}}
+
\underbrace{
\begin{bmatrix}
0 \\ I \\ 0 \\ 0 \\ 0 \\ \vdots \\ 0
\end{bmatrix}}_{:=B_{rx2}}
K_{rx2}
\underbrace{
\begin{bmatrix}
0 & 0 & 0 & -I & 0 & \cdots & 0
\end{bmatrix}}_{:=C_{rx2}}
\end{align}
Use the same definitions of \eqref{eqn:app:br1} for $G_{tx1,rx1}(z,K), \cdots, G_{tx2,rx2}(z,K)$. These transfer functions are the transfer functions of $\mathcal{N}_{uni}^{lin}(z)$ as before.

Then, Theorem~\ref{thm:LIN:br} essentially holds for multiple unicast problems as well.
\begin{theorem}
Let $K_{tx1}(z) \in \mathbb{F}[z]^{d_{tx1}\times d_{ax1}}$, $K_{tx2}(z) \in \mathbb{F}[z]^{d_{tx2}\times d_{ax2}}$, $K_i(z) \in \mathbb{F}[z]^{d_{i,in} \times d_{i,out}}$, $K_{rx1}(z) \in \mathbb{F}[z]^{d_{ax1} \times d_{rx1}}$ and $K_{rx2}(z) \in \mathbb{F}[z]^{d_{ax2} \times d_{rx2}}$. We also assume that
\begin{align}
\begin{bmatrix}
&I-H_{1,1}K_1 & \cdots & -H_{v,1} K_v \\
& \vdots & \ddots & \vdots \\
& -H_{1,v}K_1 & \cdots & I - H_{v,v} K_v
\end{bmatrix} \mbox{ is invertible.}
\end{align}
Then, for all $d_1, d_2, d_3, d_4 \in \mathbb{Z}^+$, the following two conditions are equivalent.
\begin{align}
&(i) \rank K_{rx1}(z) G_{tx1,rx1}(z,K(z)) K_{tx1}(z) \geq d_1 \\
&(ii) \rank K_{rx2}(z) G_{tx2,rx2}(z,K(z)) K_{tx2}(z) \geq d_2 \\
&(iii) \rank K_{rx2}(z) G_{tx1,rx2}(z,K(z)) K_{tx1}(z) \leq d_3 \\
&(iv) \rank K_{rx1}(z) G_{tx2,rx1}(z,K(z)) K_{tx2}(z) \leq d_4
\end{align}
if and only if
\begin{align}
&(a) \rank G_{tx',rx11'}(z,K(z)) \geq d+d_1 \\
&(b) \rank G_{tx',rx22'}(z,K(z)) \geq d+d_2 \\
&(c) \rank G_{tx',rx12'}(z,K(z)) \leq d+d_3\\
&(d) \rank G_{tx',rx21'}(z,K(z)) \leq d+d_4
\end{align}
\label{thm:LIN:uni}
\end{theorem}
\begin{proof}
Similar to Lemma~\ref{lem:LIN:maxflow}.
\end{proof}

Remark 1. The linearized problem of this theorem is essentially the same as that of broadcast problems. Compared with Theorem~\ref{thm:LIN:br}, the only difference is that $B_{tx1}$ and $B_{tx2}$ of $G_{uni}^{lin}(z,K)$ are different in multiple-unicast problems while they are the same in broadcast problems.

Remark 2. Like the broadcast problem, to design a two-unicast scheme which communicates a rate $r_1$ message to receiver $1$ and a rate $r_2$ message to receiver $2$, we have to choose $d_1=r_1$, $d_2=r_2$, $d_3=0$, $d_4=0$. The linearized network of Figure~\ref{fig:LN_unicast} can be understood as a two-receiver and two-eavesdropper secrecy problem.

\subsection{Jordan Form Externalization Example}
\label{app:jordanex}
In this section, we show how the Jordan form externalization of the implicit communication works by an explicit example. Let
\begin{align}
&A=
\begin{bmatrix}
\begin{matrix}
\lambda & 1 & 0 \\
0 & \lambda & 1 \\
0 & 0 & \lambda
\end{matrix} & {0} & {0} \\
{0} & \begin{matrix} \lambda & 1 \\ 0 & \lambda  \end{matrix} & {0} \\
{0} & {0} & \lambda' \\
\end{bmatrix} \\
&C_i=\begin{bmatrix}
C_{i,1} & C_{i,2} &  C_{i,3} & C_{i,4} &  C_{i,5} & C_{i,6}
\end{bmatrix} \\
&B_i=\begin{bmatrix}
B_{i,1} \\
B_{i,2} \\
B_{i,3} \\
B_{i,4} \\
B_{i,5} \\
B_{i,6} \\
\end{bmatrix}
\end{align}
where $\lambda \neq \lambda'$, $B_{i,j}$ are row vectors, $C_{i,j}$ are column vectors. We will externalize at the frequency $z=\lambda$. 

As mentioned in Section~\ref{sec:jordanex}, we will move the third and fifth rows and the first and fourth columns of $\lambda I - A$ to the left-top of the matrix. For this, we will define the permutation matrices $P_{L,\lambda}$ and $P_{R,\lambda}$.

The definitions of Section~\ref{sec:jordanex} is given as follows:
\begin{align}
\kappa_{L,\lambda}(0)=0, \kappa_{L,\lambda}(1)=0, \kappa_{L,\lambda}(2)=0,\kappa_{L,\lambda}(3)=1, \kappa_{L,\lambda}(4)=1, \kappa_{L,\lambda}(5)=2, \kappa_{L,\lambda}(6)=2
\end{align}
\begin{align}
\kappa_{R,\lambda}(0)=0, \kappa_{R,\lambda}(1)=1, \kappa_{R,\lambda}(2)=1,\kappa_{R,\lambda}(3)=1, \kappa_{R,\lambda}(4)=2, \kappa_{R,\lambda}(5)=2, \kappa_{R,\lambda}(6)=2
\end{align}
\begin{align}
m_{\lambda}=2
\end{align}
\begin{align}
\iota_{L,\lambda}(0)=0, \iota_{L,\lambda}(1)=3, \iota_{L,\lambda}(2)=5
\end{align}
\begin{align}
\iota_{R,\lambda}(0)=0, \iota_{R,\lambda}(1)=1, \iota_{L,\lambda}(2)=4
\end{align}
\begin{align}
\pi_{L,\lambda}(1)=3, \pi_{L,\lambda}(2)=4, \pi_{L,\lambda}(3)=1, \pi_{L,\lambda}(4)=5, \pi_{L,\lambda}(5)=2, \pi_{L,\lambda}(6)=6
\end{align}
\begin{align}
\pi_{R,\lambda}(1)=1, \pi_{R,\lambda}(2)=3, \pi_{R,\lambda}(3)=4, \pi_{R,\lambda}(4)=2, \pi_{R,\lambda}(5)=5,
\pi_{R,\lambda}(6)=6
\end{align}
\begin{align}
P_{L,\lambda}=
\begin{bmatrix}
0 & 0 & 1 & 0 & 0 & 0 \\
0 & 0 & 0 & 1 & 0 & 0 \\
1 & 0 & 0 & 0 & 0 & 0 \\
0 & 0 & 0 & 0 & 1 & 0 \\
0 & 1 & 0 & 0 & 0 & 0 \\
0 & 0 & 0 & 0 & 0 & 1 \\
\end{bmatrix},
P_{R,\lambda}=
\begin{bmatrix}
1 & 0 & 0 & 0 & 0 & 0 \\
0 & 0 & 1 & 0 & 0 & 0 \\
0 & 0 & 0 & 1 & 0 & 0 \\
0 & 1 & 0 & 0 & 0 & 0 \\
0 & 0 & 0 & 0 & 1 & 0 \\
0 & 0 & 0 & 0 & 0 & 1 \\
\end{bmatrix}
\end{align}

By multiplying $P_{L,\lambda}^T$ and $P_{R,\lambda}$ to the left and right side of $(zI-A)$, we get the following:
\begin{align}
P_{L,\lambda}^{T} (zI-A) P_{R,\lambda}
&=
\begin{bmatrix}
0 & 0 & 1 & 0 & 0 & 0 \\
0 & 0 & 0 & 1 & 0 & 0 \\
1 & 0 & 0 & 0 & 0 & 0 \\
0 & 0 & 0 & 0 & 1 & 0 \\
0 & 1 & 0 & 0 & 0 & 0 \\
0 & 0 & 0 & 0 & 0 & 1 \\
\end{bmatrix}^{T}
\begin{bmatrix}
\begin{matrix}
z-\lambda & -1 & 0 \\
0 & z-\lambda & -1 \\
0 & 0 & z-\lambda
\end{matrix} & {0} & {0} \\
{0} & \begin{matrix} z-\lambda & -1 \\ 0 & z-\lambda  \end{matrix} & {0} \\
{0} & {0} & z-\lambda' \\
\end{bmatrix}
\begin{bmatrix}
1 & 0 & 0 & 0 & 0 & 0 \\
0 & 0 & 1 & 0 & 0 & 0 \\
0 & 0 & 0 & 1 & 0 & 0 \\
0 & 1 & 0 & 0 & 0 & 0 \\
0 & 0 & 0 & 0 & 1 & 0 \\
0 & 0 & 0 & 0 & 0 & 1 \\
\end{bmatrix} \\
&=
\begin{bmatrix}
0 & 0 & z-\lambda & 0 & 0 & 0 \\
0 & 0 & 0 & 0 & z-\lambda & 0 \\
z-\lambda & -1 & 0 & 0 & 0 & 0 \\
0 & z-\lambda & -1 & 0 & 0 & 0 \\
0 & 0 & 0 & z-\lambda & -1 & 0 \\
0 & 0 & 0 & 0 & 0 & z-\lambda' \\
\end{bmatrix}
\begin{bmatrix}
1 & 0 & 0 & 0 & 0 & 0 \\
0 & 0 & 1 & 0 & 0 & 0 \\
0 & 0 & 0 & 1 & 0 & 0 \\
0 & 1 & 0 & 0 & 0 & 0 \\
0 & 0 & 0 & 0 & 1 & 0 \\
0 & 0 & 0 & 0 & 0 & 1 \\
\end{bmatrix} \\
&=
\begin{bmatrix}
0 & 0 & 0 & z-\lambda & 0 & 0 \\
0 & 0 & 0 & 0 & z-\lambda & 0 \\
z-\lambda & 0 & -1 & 0 & 0 & 0 \\
0 & 0 & z-\lambda & -1 & 0 & 0 \\
0 & z-\lambda & 0 & 0 & -1 & 0 \\
0 & 0 & 0 & 0 & 0 & z-\lambda'
\end{bmatrix}
\end{align}

Here, we can notice that the $2 \times 2$ left-top sub-matrix is a zero matrix. Furthermore, $P_{L,\lambda}^T(\lambda I - A) P_{R,\lambda}$ is a diagonal matrix.

$A_{\lambda,1,1}(z),A_{\lambda,1,2}(z),A_{\lambda,2,1}(z),A_{\lambda,2,2}(z)$ are defined as
\begin{align}
&A_{\lambda,1,1}(z)=\begin{bmatrix} 0 & 0 \\ 0 & 0 \end{bmatrix}, A_{\lambda,1,2}(z)=\begin{bmatrix} 0 & z-\lambda & 0 & 0 \\  0 & 0 & z-\lambda & 0 \end{bmatrix} \\
&A_{\lambda,2,1}(z)=\begin{bmatrix}
z-\lambda & 0 \\
0 & 0 \\
0 & z-\lambda \\
0 & 0 \\
\end{bmatrix}, A_{\lambda,2,2}(z)=\begin{bmatrix} -1 & 0 & 0 & 0 \\ z-\lambda & -1 & 0 & 0 \\ 0 & 0 & -1 & 0 \\ 0 & 0 & 0 & z-\lambda' \end{bmatrix}.
\end{align}

We also multiply $P_{R,\lambda}$ and $P_{L,\lambda}$ to $C_i$ and $B_i$ respectively.
\begin{align}
C_i P_{R,\lambda}=\begin{bmatrix} C_{i,1} & C_{i,4} & C_{i,2} &  C_{i,3} & C_{i,5} & C_{i,6} \end{bmatrix}
\end{align}
\begin{align}
P_{L,\lambda}^T B_i=\begin{bmatrix} B_{i,3} \\ B_{i,5} \\ B_{i,1} \\ B_{i,2} \\ B_{i,4} \\ B_{i,6} \end{bmatrix}
\end{align}

Therefore, $C_{i,\lambda,1},C_{i,\lambda,2},B_{i,\lambda,1},B_{i,\lambda,2}$ are defined as follows.
\begin{align}
C_{i,\lambda,1}=\begin{bmatrix} C_{i,1} & C_{i,4} \end{bmatrix}, C_{i,\lambda,2}=\begin{bmatrix}C_{i,2} &  C_{i,3} & C_{i,5} & C_{i,6}\end{bmatrix}
\end{align}
\begin{align}
B_{i,\lambda,1}=
\begin{bmatrix}
B_{i,3} \\ B_{i,5}
\end{bmatrix},
B_{i,\lambda,2}
=
\begin{bmatrix}
B_{i,1} \\ B_{i,2} \\ B_{i,4} \\ B_{i,6}
\end{bmatrix}
\end{align}

We also introduce auxiliary inputs and outputs which access to each Jordan block. For this, we define $C_{\lambda}$ and $B_{\lambda}$ as follows. 
\begin{align}
C_{\lambda}=\begin{bmatrix}
1 & 0 & 0 & 0 & 0 & 0 \\
0 & 0 & 0 & 1 & 0 & 0 \\
\end{bmatrix}
,
B_{\lambda}=\begin{bmatrix}
0 & 0 \\ 0 & 0 \\ 1 & 0 \\ 0 & 0 \\ 0 & 1 \\ 0 & 0
\end{bmatrix}
\end{align}

With these definitions, we can construct the network $\mathcal{N}_{jd.\lambda}$. The channel matrices of $\mathcal{N}_{jd,\lambda}(\lambda)$ are given as follows:
\begin{align}
&H_{tx,rx}(\lambda)=0 \\
&H_{tx,i}(\lambda)=\begin{bmatrix} C_{i,1} & C_{i,4} \end{bmatrix} \\
&H_{i,rx}(\lambda)=\begin{bmatrix} B_{i,3} \\ B_{i,5}  \end{bmatrix} \\
&H_{i,j}(\lambda)=
\begin{bmatrix}C_{i,2} &  C_{i,3} & C_{i,5} & C_{i,6}\end{bmatrix}
\begin{bmatrix}
-1 & 0 & 0 & 0 \\
0 & -1 & 0 & 0 \\
0 & 0 & -1 & 0 \\
0 & 0 & 0 & \lambda-\lambda' \\
\end{bmatrix}^{-1}
\begin{bmatrix}B_{i,1} \\ B_{i,2} \\ B_{i,4} \\ B_{i,6}\end{bmatrix}
\end{align}

\subsection{Externalization of Implicit Communication in Proper Systems}
\label{app:proper}
In this section, we extend the discussion of Section~\ref{sec:external} to proper systems. The extension of  fixed modes to proper systems can be found in \cite{Davison_Decentralized}. Formally, the proper decentralized linear system, $\mathcal{L}(A,B_i,C_i,D_{ij})$, is defined as follows:
\begin{align}
&x[n+1]=Ax[n]+\sum^{v}_{i=1}B_i u_i[n]\\
&y_i[n]=C_i x[n]+\sum^{v}_{j=1}D_{ij}u_j[n]
\end{align}
Then, the definition of fixed modes can be extended to proper decentralized systems as follows.
\begin{definition}\cite[Definition 2]{Davison_Decentralized}
$\lambda$ is called a fixed mode of $\mathcal{L}(A,B_i,C_i,D_{ij})$ if
\begin{align}
\lambda \in \bigcap_{(K_1,\cdots,K_i) \in \mathcal{K}}  \sigma(
A+\begin{bmatrix} B_1 K_1 & \cdots & B_v K_v \end{bmatrix}\left( I -
\begin{bmatrix}
D_{11} K_1  & \cdots & D_{1v} K_v \\
\vdots & \ddots & \vdots \\
D_{v1} K_1 & \cdots & D_{vv} K_v
\end{bmatrix}
\right)^{-1} \begin{bmatrix} C_1 \\ \vdots \\ C_v \end{bmatrix}
)
\end{align}
where $\sigma(\cdot)$ implies the set of the eigenvalues of the matrix and $\mathcal{K}=\{ (K_1,\cdots,K_v) : K_i \in \mathbb{C}^{q_i \times r_i}, I-
\begin{bmatrix}
D_{11} K_1  & \cdots & D_{1v} K_v \\
\vdots & \ddots & \vdots \\
D_{v1} K_1 & \cdots & D_{vv} K_v
\end{bmatrix}
\mbox{ is invertible} \}$.
\end{definition}
As before, the stabilizability condition is charaterized by the fixed modes of the system.
\begin{theorem}\cite[Theorem 3]{Davison_Decentralized}
$\mathcal{L}(A,B_i,C_i,D_{ij})$ is stabilizable if and only if all of its fixed modes are within the unit circle.
\end{theorem}
Then, we can externalize information flows to stabilize the proper system as before.

\subsection{Canonical Externalization I}
We will introduce the gain $K_i$ to the $i$th controller, and the auxiliary input $u[n]$ and output $y[n]$ (which can access to all states and observations, $x[n], y_1[n], \cdots, y_v[n]$) to the system. Then, the system equation can be written as follows:
\begin{align}
&\begin{bmatrix}
x[n+1] \\
y_1[n] \\
\vdots \\
y_{v}[n]
\end{bmatrix}
=
\begin{bmatrix}
A & B_1 K_1 & \cdots & B_{v} K_{v}  \\
C_1 & D_{11} K_1 & \cdots & D_{1v} K_{v} \\
\vdots & \vdots & \ddots & \vdots \\
C_{v} & D_{v 1} K_1 & \cdots  & D_{v v} K_{v} \\
\end{bmatrix}
\begin{bmatrix}
x[n] \\
y_1[n] \\
\vdots \\
y_{v}[n]
\end{bmatrix}+u[n] \\
&y[n]=\begin{bmatrix}
x[n] \\
y_1[n] \\
\vdots \\
y_{v}[n] \\
\end{bmatrix}
\end{align}
Then, the transfer function from $y(z)$ to $u(z)$, $G_{cn I}(z,K)$ , is given as follows.
\begin{align}
G_{cnI}(z,K_i)&=
\begin{bmatrix}
zI & 0 & \cdots & 0 \\
0 & I & \cdots & 0 \\
\vdots & \vdots & \ddots & \vdots \\
0 & 0 & \cdots & I
\end{bmatrix}
-
\begin{bmatrix}
A & B_1 K_1 & \cdots & B_{v} K_{v}  \\
C_1 & D_{11} K_1 & \cdots & D_{1v} K_{v} \\
\vdots & \vdots & \ddots & \vdots \\
C_{v} & D_{v 1} K_1 & \cdots  & D_{v v} K_{v} \\
\end{bmatrix} \\
&=
\underbrace{
\begin{bmatrix}
zI-A & 0 & \cdots & 0 \\
-C_1 & I & \cdots & 0 \\
\vdots & \vdots & \ddots & \vdots \\
-C_{v} & 0 & \cdots & I
\end{bmatrix}}_{:=A_{cnI}(z)}
+
\underbrace{\begin{bmatrix}
B_1 \\ D_{11} \\ \vdots \\ D_{v 1}
\end{bmatrix}}_{:=B_{cnI,1}}
K_1
\underbrace{\begin{bmatrix}
0 & -I & \cdots & 0
\end{bmatrix}}_{:=C_{cnI,1}}
+
\cdots
+
\underbrace{
\begin{bmatrix}
B_{v} \\ D_{1 v} \\ \vdots \\ D_{v v}
\end{bmatrix}}_{:=B_{cnI,v}}
K_{v}
\underbrace{
\begin{bmatrix}
0 & 0 & \cdots & -I
\end{bmatrix}}_{:=C_{cnI,v}}
\end{align}
By Lemma~\ref{lem:LTI:trans}, the standard network, $\mathcal{N}_s (A_{cnI}(z);B_{cnI,i},0;C_{cnI,i},0;0,0;0,0)$, has $G_{cnI}(z,K)$ as a transfer function. Denote this network as $\mathcal{N}_{cnI}(z)$.  Then, we can prove the similar theorem as before.
\begin{theorem}
Given the above definitions, the following statements are equivalent.
\begin{align}
&(1)\ \lambda \mbox{ is a fixed mode of the decentralized linear system }\mathcal{L}(A,B_i,C_i,D_{ij})\nonumber \\
&(2)\ \rank(G_{cnI}(\lambda,K)) < \dim(A_{cnI}) \nonumber \\
&(3)\ (\mbox{transfer matrix rank of the LTI network }\mathcal{N}_{cnI}(\lambda)) < \dim(A_{cnI}) \nonumber \\
&(4)\ (\mbox{mincut rank of the LTI network }\mathcal{N}_{cnI}(\lambda)) < \dim(A_{cnI}) \nonumber \\
&(5)\ \min_{V \subseteq \{1,\cdots,v \}} \rank \begin{bmatrix} A_{cnI}(\lambda) & B_{cnI,V} \\ C_{cnI,V^c} & 0 \end{bmatrix} < \dim(A_{cnI}) \nonumber
\end{align}
\end{theorem}
\begin{proof}
Similar to theorem~\ref{thm:equivalence1}.
\end{proof}

\subsection{Canonical Externalization II}
Like the discussion of section~\ref{sec:external}, we only need the auxiliary input and output to be connected to the unstable states. Thus, we can reduce the dimension of the auxiliary input and output by allowing them only to access the state $x[n]$. Now, the system equation is given as follows:
\begin{align}
&\begin{bmatrix}
x[n+1] \\
y_1[n] \\
\vdots \\
y_{v}[n]
\end{bmatrix}
=
\begin{bmatrix}
A & B_1 K_1 & \cdots & B_{v} K_{v}  \\
C_1 & D_{11} K_1 & \cdots & D_{1v} K_{v} \\
\vdots & \vdots & \ddots & \vdots \\
C_{v} & D_{v 1} K_1 & \cdots  & D_{v v} K_{v} \\
\end{bmatrix}
\begin{bmatrix}
x[n] \\
y_1[n] \\
\vdots \\
y_{v}[n]
\end{bmatrix}+
\begin{bmatrix} I \\ 0 \\ \vdots \\ 0 \end{bmatrix}
u[n] \\
&y[n]=
\begin{bmatrix}
I & 0 & \cdots &  0
\end{bmatrix}
\begin{bmatrix}
x[n] \\
y_1[n] \\
\vdots \\
y_{v}[n] \\
\end{bmatrix}
\end{align}
The transfer function from $u(z)$ to $y(z)$ is the following.
\begin{align}
y(z)=&
\begin{bmatrix}
I & 0 & \cdots & 0
\end{bmatrix}
\left(
\begin{bmatrix}
zI & 0 & \cdots & 0 \\
0 & I & \cdots & 0 \\
\vdots & \vdots & \ddots & \vdots \\
0 & 0 & \cdots & I
\end{bmatrix}
-
\begin{bmatrix}
A & B_1 K_1 & \cdots & B_{v} K_{v}  \\
C_1 & D_{11} K_1 & \cdots & D_{1v} K_{v} \\
\vdots & \vdots & \ddots & \vdots \\
C_{v} & D_{v 1} K_1 & \cdots  & D_{v v} K_{v} \\
\end{bmatrix}
\right)^{-1}
\begin{bmatrix}
I \\ 0 \\ \vdots \\ 0
\end{bmatrix}
u(z) \\
\end{align}
By Lemma~\ref{lem:ext:matrix}, the transfer function from $y(z)$ to $u(z)$, $G_{cn II}(z,K)$, is given as follows:
\begin{align}
G_{cn II}(z,K)&=
(zI-A)-
\begin{bmatrix}
-B_1 K_1 & \cdots & - B_v K_v
\end{bmatrix}
\left(I-
\begin{bmatrix}
D_{11} K_1 &  \cdots  & D_{1 v} K_{v} \\
\vdots & \ddots & \vdots \\
D_{v 1} K_1 & \cdots  & D_{v v} K_{v}
\end{bmatrix}
\right)^{-1}
\begin{bmatrix}
-C_{1} \\
\vdots \\
-C_{v}
\end{bmatrix}\nonumber \\
&=\underbrace{(zI-A)}_{:=A_{cnII}(z)} +
(\underbrace{-B_1}_{:=B_{cnII,1}} K_1 \underbrace{\begin{bmatrix} I & \cdots & 0 \end{bmatrix}}_{:=C_{cnII,1}'} - \cdots  \underbrace{-B_v}_{:=B_{cnII,v}} K_v \underbrace{\begin{bmatrix} 0 & \cdots & I \end{bmatrix}}_{:=C_{cnII,v}'} )\nonumber \\
&\cdot
\left(\underbrace{I}_{:=S_{cnII}^{-1}}
-\left(
\underbrace{\begin{bmatrix} D_{11} \\ \vdots \\ D_{v 1} \end{bmatrix}}_{:=B_{cnII,1}'}
K_1 \begin{bmatrix} I & \cdots & 0 \end{bmatrix}
+\cdots+
\underbrace{\begin{bmatrix} D_{1 v} \\ \vdots \\ D_{v v} \end{bmatrix}}_{:=B_{cnII,v}'}
K_v \begin{bmatrix} 0 & \cdots & I \end{bmatrix}
\right)\right)^{-1}
\underbrace{\begin{bmatrix}
C_1 \\
\vdots \\
C_v
\end{bmatrix}}_{:=D_{cnII}'}
\end{align}
Then, by Lemma~\ref{lem:LTI:trans}, we can see that $G_{cn II}(z,K)$ is the transfer function of the standard network\\
$\mathcal{N}_s(A_{cnII}(z);B_{cnII,i},B_{cnII,i}';0,C_{cnII,i}';0,D_{cnII}';S_{cnII},0)$. Denote this network as $\mathcal{N}_{cnII}(z)$. Furthermore, by lemma~\ref{lem:LTI:trans} the channel between the nodes and the channel for the cut $V=\{tx,i_1,\cdots, i_k \}$ are given as follows:
\begin{align}
&H_{tx,rx}(z)=zI-A\\
&H_{tx,i}= C_i\\
&H_{i,rx}=-B_i\\
&H_{i,j}= D_{ji}\\
&H_{V,V^c}(z)=
\begin{bmatrix}
zI-A & -B_{i_1} & \cdots & -B_{i_k} \\
C_{i_{k+1}} & D_{i_{k+1},i_1} & \cdots & D_{i_{k+1},i_k} \\
\vdots & \vdots & \ddots & \vdots \\
C_{i_{v}} & D_{i_{v},i_1} & \cdots & D_{i_{v},i_k} \\
\end{bmatrix}
\end{align}
Then, we can give the capacity-stabilizability equivalence theorem as before.
\begin{theorem}
Given the above definitions, the following statements are equivalent.
\begin{align}
&(1)\ \lambda \mbox{ is a fixed mode of the decentralized linear system }\mathcal{L}(A,B_i,C_i,D_{ij})\nonumber \\
&(2)\ \rank(G_{cnII}(\lambda,K)) < \dim(A) \nonumber \\
&(3)\ (\mbox{transfer matrix rank of the LTI network }\mathcal{N}_{cnII}(\lambda)) < \dim(A) \nonumber \\
&(4)\ (\mbox{mincut rank of the LTI network }\mathcal{N}_{cnII}(\lambda)) < \dim(A) \nonumber \\
&(5)\ \min_{V \subseteq \{1,\cdots,v \}} \rank \begin{bmatrix} \lambda I-A & -B_V \\ C_{V^c} & D_{V^c,V} \end{bmatrix} < \dim(A) \nonumber
\end{align}
\label{thm:equivalence2-2}
\end{theorem}
\begin{proof}
Similar to Theorem~\ref{thm:equivalence1}.
\end{proof}
Here, it has to be mentioned that the equivalence of (1) and (5) was already shown in \cite{Davison_Decentralized}.

\subsection{Jordan Form Externalization}
\label{app:proper_jd}
Like section~\ref{sec:jordanex}, we can minimize the dimension of the auxiliary input and output by using the Jordan form. Without loss of generality, we assume that $A$ is in Jordan form and use the same notations of section~\ref{sec:jordanex}. Then, the system equation with the auxiliary input $u_\lambda[n]$ and output $y_\lambda[n]$ is given as follows:
\begin{align}
&\begin{bmatrix}
x[n+1] \\
y_1[n] \\
\vdots \\
y_{v}[n]
\end{bmatrix}
=
\begin{bmatrix}
A & B_1 K_1 & \cdots & B_{v} K_{v}  \\
C_1 & D_{11} K_1 & \cdots & D_{1v} K_{v} \\
\vdots & \vdots & \ddots & \vdots \\
C_{v} & D_{v 1} K_1 & \cdots  & D_{v v} K_{v} \\
\end{bmatrix}
\begin{bmatrix}
x[n] \\
y_1[n] \\
\vdots \\
y_{v}[n]
\end{bmatrix}+
\underbrace{
\begin{bmatrix} C_{\lambda} \\ 0 \\ \vdots \\ 0 \end{bmatrix}
}_{:=C_{\lambda}'}
u_{\lambda}[n] \\
&y_{\lambda}[n]=
\underbrace{
\begin{bmatrix}
B_{\lambda} & 0 & \cdots &  0
\end{bmatrix}
}_{:=B_{\lambda}'}
\begin{bmatrix}
x[n] \\
y_1[n] \\
\vdots \\
y_{v}[n] \\
\end{bmatrix}
\end{align}
We also expand the dimension of the permutation matrices $P_{L,\lambda}$ and $P_{R,\lambda}$.
\begin{align}
P_{L,\lambda}':=\begin{bmatrix} P_{L,\lambda} & 0 & \cdots & 0 \\ 0 & I & \cdots & 0 \\
\vdots & \vdots & \ddots & \vdots \\ 0 & 0 & \cdots & I \end{bmatrix} \\
P_{R,\lambda}':=\begin{bmatrix} P_{R,\lambda} & 0 & \cdots & 0 \\ 0 & I & \cdots & 0 \\
\vdots & \vdots & \ddots & \vdots \\ 0 & 0 & \cdots & I
\end{bmatrix}
\end{align}
The transfer function from $u_\lambda(z)$ to $y_\lambda(z)$ is the following.
\begin{align}
y_{\lambda}(z)&=
C_\lambda'(
\begin{bmatrix}
zI & 0 & \cdots & 0 \\
0 & I & \cdots & 0 \\
\vdots & \vdots & \ddots & \vdots \\
0 & 0 & \cdots & I
\end{bmatrix}
-
\begin{bmatrix}
A & B_1 K_1 & \cdots & B_v K_v \\
C_1 & D_{11} K_1 & \cdots & D_{1v} K_v \\
\vdots & \vdots & \ddots & \vdots \\
C_v & D_{v1} K_v & \cdots & D_{vv} K_v
\end{bmatrix}
)^{-1}
B_\lambda'
u_\lambda(z) \\
&=
C_\lambda'(P'_{L,\lambda} {P'_{L,\lambda}}^T
(
\begin{bmatrix}
zI-A & -B_1 K_1 & \cdots & -B_v K_v \\
-C_1 & I-D_{11} K_1 & \cdots & -D_{1v} K_v \\
\vdots & \vdots & \ddots & \vdots \\
-C_v & -D_{v1} K_v & \cdots & I-D_{vv} K_v
\end{bmatrix}
)
P_{R,\lambda}' {P_{R,\lambda}'}^T
)^{-1}
B_\lambda'
u_\lambda(z) \\
&=
C_\lambda' {P_{R,\lambda}'} ( {P'_{L,\lambda}}^T
(
\begin{bmatrix}
zI-A & -B_1 K_1 & \cdots & -B_v K_v \\
-C_1 & I-D_{11} K_1 & \cdots & -D_{1v} K_v \\
\vdots & \vdots & \ddots & \vdots \\
-C_v & -D_{v1} K_v & \cdots & I-D_{vv} K_v
\end{bmatrix}
)
P_{R,\lambda}'
)^{-1}
{P'_{L,\lambda}}^T
B_\lambda'
u_\lambda(z) \\
&=
C_\lambda' {P_{R,\lambda}'}
\begin{bmatrix}
P_{L,\lambda}^T(zI-A)P_{R,\lambda} & -P_{L,\lambda}^T B_1 K_1 & \cdots & -P_{L,\lambda}^T B_v K_v \\
-C_1 P_{R,\lambda} & I-D_{11}K_1 & \cdots & -D_{1v} K_v \\
\vdots & \vdots & \ddots & \vdots \\
-C_v P_{R,\lambda} & -D_{v1} K_v & \cdots & I-D_{vv}K_v \\
\end{bmatrix}^{-1}
{P'_{L,\lambda}}^T
B_\lambda'
u_\lambda(z)\\
&=
\begin{bmatrix}
I & 0 & 0 & \cdots & 0
\end{bmatrix}
\begin{bmatrix}
A_{\lambda,1,1}(z) & A_{\lambda,1,2}(z) & -B_{1,\lambda,1}K_1 & \cdots & -B_{v,\lambda,1}K_v\\
A_{\lambda,2,1}(z) & A_{\lambda,2,2}(z) & -B_{1,\lambda,2}K_1 & \cdots & -B_{v,\lambda,2}K_v\\
-C_{1,\lambda,1} & -C_{1,\lambda,2} & I-D_{11}K_1 & \cdots & -D_{1v} K_v \\
\vdots & \vdots & \vdots & \ddots & \vdots \\
-C_{v,\lambda,1} & -C_{v,\lambda,2} & -D_{v1}K_v & \cdots & I-D_{vv} K_v \\
\end{bmatrix}^{-1}
\begin{bmatrix}
I \\
0 \\
0 \\
\vdots \\
0
\end{bmatrix}u_\lambda(z)
\end{align}
By Lemma~\ref{lem:ext:matrix}, the transfer matrix from $y_\lambda(z)$ to $u_\lambda(z)$, $G_{jd}(z)$ , is given as
\begin{align}
&G_{jd}(z)=
A_{\lambda,1,1}(z)-
\begin{bmatrix}
A_{\lambda,1,2}(z) & -B_{1,\lambda,1}K_1 & \cdots & -B_{v,\lambda,1}K_v
\end{bmatrix}
\begin{bmatrix}
A_{\lambda,2,2}(z) & -B_{1,\lambda,2}K_1 & \cdots & -B_{v,\lambda,2}K_v \\
-C_{1,\lambda,2} & I-D_{11}K_1 & \cdots & -D_{1v}K_v \\
\vdots & \vdots & \ddots & \vdots \\
-C_{v,\lambda,2} & -D_{v1}K_1 & \cdots & I-D_{vv}K_v \\
\end{bmatrix}^{-1}
\begin{bmatrix}
A_{\lambda,2,1}(z) \\
-C_{1,\lambda,1} \\
\vdots \\
-C_{v,\lambda,1}
\end{bmatrix} \nonumber \\
&=
\underbrace{A_{\lambda,1,1}(z)}_{:=A_{jd}(z)}
+(
\underbrace{\begin{bmatrix} A_{\lambda,1,2}(z) & 0 & \cdots & 0 \end{bmatrix}}_{:=D_{jd}(z)}
\underbrace{-B_{1,\lambda,1}}_{:=B_{jd,1}} K_1 \underbrace{\begin{bmatrix} 0 & I & \cdots & 0 \end{bmatrix}}_{:=C'_{jd,1}}
- \cdots
\underbrace{-B_{v,\lambda,1}}_{:=B_{jd,v}} K_v \underbrace{\begin{bmatrix} 0 & 0 & \cdots & I \end{bmatrix}}_{:=C'_{jd,v}}
)\\
&\cdot (
\underbrace{I}_{:=S_{jd}^{-1}}
-
(
\underbrace{
\begin{bmatrix}
I-A_{\lambda,2,2}(z) & 0 & \cdots & 0 \\
C_{1,\lambda,2} & 0 & \cdots & 0 \\
\vdots & \vdots & \ddots & \vdots \\
C_{v,\lambda,2} & 0 & \cdots & 0
\end{bmatrix}
}_{:=S_{jd}'(z)}
+
\underbrace{
\begin{bmatrix}
B_{1,\lambda,2} \\ D_{11} \\ \vdots \\ D_{v1}
\end{bmatrix}}_{:=B_{jd,1}'}
K_1
\begin{bmatrix}
0 & I & \cdots & 0
\end{bmatrix}
+
\cdots
+
\underbrace{
\begin{bmatrix}
B_{v,\lambda,2} \\ D_{1v} \\ \vdots \\ D_{vv}
\end{bmatrix}}_{:=B_{jd,v}'}
K_v
\begin{bmatrix}
0 & 0 & \cdots & I
\end{bmatrix}
)
)^{-1}\\
&\cdot \underbrace{\begin{bmatrix}
-A_{\lambda,2,1}(z) \\
C_{1,\lambda,1} \\
\vdots \\
C_{v,\lambda,1}
\end{bmatrix}}_{:=D_{jd}'(z)}
\end{align}
Then, we can easily check that $G_{jd}(z)$ is the transfer function of the standard network
\begin{align}
\mathcal{N}_s(A_{jd}(z);B_{jd,i},B_{jd,i}';0,C_{jd,i}';D_{jd}(z),D_{jd}'(z);S_{jd},S_{jd}'(z)).
\end{align}
Moreover, we have
\begin{align}
(S_{jd}^{-1}-S_{jd}')^{-1}=
\begin{bmatrix}
A_{\lambda,2,2}(z) & 0 & \cdots & 0 \\
-C_{1,\lambda,2}(z) & I & \cdots & 0 \\
\vdots & \vdots & \ddots & \vdots \\
-C_{v,\lambda,2}(z) & 0 & \cdots & I \\
\end{bmatrix}^{-1}
=
\begin{bmatrix}
A_{\lambda,2,2}(z)^{-1} & 0 & \cdots & 0 \\
C_{1,\lambda,2} A_{\lambda,2,2}(z)^{-1} & I & \cdots & 0 \\
\vdots & \vdots & \ddots & \vdots \\
C_{v,\lambda,2} A_{\lambda,2,2}(z)^{-1} & 0 & \cdots & I
\end{bmatrix}.
\end{align}
Thus, by lemma~\ref{lem:LTI:trans} the channel matrix between the nodes and the channel matrix for the cut $V=\{tx,i_1,\cdots, i_k \}$ is given as follows:
\begin{align}
&H_{tx,rx}(\lambda)=0 \\
&H_{tx,i}(\lambda)=C_{i,\lambda,1}\\
&H_{i,rx}(\lambda)=-B_{i,\lambda,1}\\
&H_{i,j}(\lambda)=C_{j,\lambda,2}A_{\lambda,2,2}(\lambda)^{-1}B_{i,\lambda,2}+D_{ji}\\
&H_{V,V^c}(\lambda):=
\begin{bmatrix}
0 & -B_{i_1,\lambda,1} & \cdots & -B_{i_k,\lambda,1} \\
C_{i_{k+1},\lambda,1} & C_{i_{k+1},\lambda,2}A_{\lambda,2,2}(\lambda)^{-1}B_{i_1,\lambda,2}+D_{i_{k+1}i_1} & \cdots & C_{i_{k+1},\lambda,2}A_{\lambda,2,2}(\lambda)^{-1}B_{i_k,\lambda,2}+D_{i_{k+1}i_k} \\
\vdots & \vdots & \ddots & \vdots \\
C_{i_v,\lambda,1} & C_{i_v,\lambda,2}A_{\lambda,2,2}(\lambda)^{-1}B_{i_1,\lambda,2}+D_{i_v i_1} & \cdots & C_{i_v,\lambda,2}A_{\lambda,2,2}(\lambda)^{-1}B_{i_k,\lambda,2}+D_{i_v i_k} \\
\end{bmatrix}
\end{align}
Then, we can write a similar theorem as before.
\begin{theorem}
Given the above definitions, the following statements are equivalent.
\begin{align}
&(1)\ \lambda \mbox{ is the fixed mode of the decentralized linear system }\mathcal{L}(A,B_i,C_i,D_{ij})\nonumber \\
&(2)\ \rank(G_{jd}(\lambda,K)) < m_\lambda \nonumber \\
&(3)\ (\mbox{transfer matrix rank of the LTI network }\mathcal{N}_{jd}(\lambda)) < m_\lambda \nonumber \\
&(4)\ (\mbox{mincut rank of the LTI network }\mathcal{N}_{jd}(\lambda)) < m_\lambda \nonumber \\
&(5)\ \min_{V \subseteq \{1,\cdots,v \}} \rank
\begin{bmatrix}
0 & -B_{V,\lambda,1} \\
C_{V^c,\lambda,1} & C_{V^c,\lambda,2} A_{\lambda,2,2}(\lambda)^{-1}B_{V,\lambda,2}+D_{V^c,V}
\end{bmatrix}
< m_\lambda \nonumber
\end{align}
\label{thm:equivalence2-3}
\end{theorem}
\begin{proof}
Similar to theorem~\ref{thm:equivalence1}.
\end{proof}

\subsection{Realization of Closed LTI Network}
\label{app:realization}
\begin{figure}
\includegraphics[width = 1in]{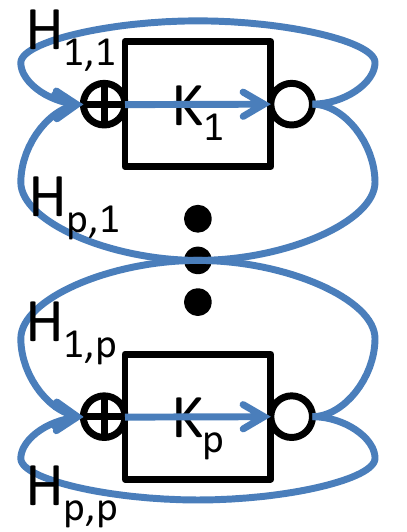}
\caption{General Closed LTI Network}
\label{fig:closedLTI}
\end{figure}
In this section, we will discuss how to realize the problem of Figure~\ref{fig:ptop} to a decentralized linear system form.
First, we can notice that the system of Figure~\ref{fig:ptop} can be thought as a special case of the closed LTI network of Figure~\ref{fig:closedLTI}. We can put $p$ of Figure~\ref{fig:closedLTI} as $v+2$, and consider the relay $i$ of Figure~\ref{fig:ptop} as the node $i$ of Figure~\ref{fig:closedLTI}, the observer as the node $v+1$, and the controller as the node $v+2$. Then, by connecting the node $v+1$ with the node $v+2$ with $H_{(v+2)(v+1)}(z)$ which is equivalent to the plant of Figure~\ref{fig:ptop}, the two problems are equivalent. Therefore, we can focus on the realization of the closed LTI network of Fig.~\ref{fig:closedLTI}.

As we can see in Figure~\ref{fig:ptop}, for $1 \leq i,j \leq p$ the input of node $i$ is connected to the output of node $j$ by the channel $H_{ij}(z)$. When $i=j$, it corresponds to a self-loop. In other words, $y_j(z)=H_{ij}(z) u_i(z)$ where $u_i(z)$ is the input of the node $i$ and $y_j(z)$ is the output of the node $j$. Since this relationship can be considered as a centralized input-output system, it can be realized by the usual realization method shown in \cite[chapter 7]{Chen}. Let's say the resulting linear system is given as follows:
\begin{align}
&x_{ij}[n+1]=A_{ij}x_{ij}[n]+B_{ij}u_i[n] \\
&y_j[n]=C_{ij}x_{ij}[n]+D_{ij}u_i[n]
\end{align}
Let the dimension of $u_i[n]$ be $q_i$, the dimension of $y_i[n]$ be $r_i$ and the dimension of $x_{ij}[n]$ be $m_{ij}$. Then, the dimensions of the other matrices are uniquely determined. When there is no connection between the nodes, simply $m_{ij}$ becomes $0$.

The main idea for the realization of a closed LTI network is to augment the states $x_{ij}[n]$. Denote $x[n]$, $A$, $B_{i}$ and $C_i$ as follows:\\
$x[n]:=\begin{bmatrix} x_{11}[n+1] \\ \vdots \\ x_{1p}[n+1] \\ x_{21}[n+1] \\ \vdots \\ x_{pp}[n+1] \end{bmatrix}$\\
$A:=diag(A_{11},\cdots,A_{1p},A_{21},\cdots, A_{pp})$\\
$B_i:=\begin{bmatrix} 0_{(\sum_{1 \leq j < i} \sum_{1 \leq k \leq p} m_{jk}) \times q_i} \\ B_{i1} \\ \vdots \\ B_{ip} \\
0_{(\sum_{i < j \leq p} \sum_{1 \leq k \leq p} m_{jk}) \times q_i} \end{bmatrix}$\\
$C_{ij}':= \begin{bmatrix} 0_{r_j \times \sum_{1 \leq k < j} m_{ik}} & C_{ij} & 0_{r_j \times \sum_{j < k \leq p} m_{ik}} \end{bmatrix}$\\
$C_i:=\begin{bmatrix}C_{1i}' & \cdots & C_{pi}' \end{bmatrix}$.

Then, we can easily check that the decentralized linear system
\begin{align}
&x[n+1]=Ax[n]+ \sum_i B_i u_i[n] \\
&y_i[n]=C_ix[n]+ \sum_{i,j}D_{ij} u_i[n]
\end{align}
is the realization of the closed LTI network of Fig.~\ref{fig:closedLTI}.
\end{document}